\documentclass[10pt]{amsart}
\usepackage{pgf,tikz,pgfplots,color}
\pgfplotsset{compat=1.15}
\usepackage{mathrsfs} 
\usetikzlibrary{arrows}
\usepackage{fancyhdr}
\usepackage{lastpage} 

\usepackage[margin=1in]{geometry}
\usepackage{hyperref}
\hypersetup{
    colorlinks=true,
    linkcolor=blue,
    citecolor=brown,
    filecolor=magenta,
    urlcolor=blue,
}
\usepackage{amsmath}
\usepackage{amsfonts}
\usepackage{amssymb}
\usepackage{amsthm}
\usepackage{setspace}
\usepackage{inputenc}
\usepackage{comment}
\usepackage[shortlabels]{enumitem}
\numberwithin{equation}{section}

\newcommand{\dist}{\operatorname{dist}}

\newcommand{\supp}{\operatorname{supp}}
\newcommand{\loc}{\mathrm{loc}}

\newtheorem{thm}{Theorem}[section]
\newtheorem{cor}[thm]{Corollary} 
\newtheorem{prop}[thm]{Proposition}
\newtheorem{lemma}[thm]{Lemma}

\theoremstyle{definition}
\newtheorem{defn}[thm]{Definition}

\newtheorem{note}[thm]{Notation}
\newtheorem{conv}[thm]{Convention}

\theoremstyle{remark}
\newtheorem{rem}[thm]{Remark}

\newtheorem*{ack}{Acknowledgment}

\newcommand{\B}{\mathbb{B}} 
\newcommand{\Z}{\mathbb{Z}}
\newcommand{\R}{\mathbb{R}} 
\newcommand{\C}{\mathbb{C}}

\newcommand{\N}{\mathbb{N}}
\newcommand{\1}{\mathbf{1}}

\newcommand{\As}{\mathscr{A}}

\newcommand{\Co}{\mathscr{C}}
\newcommand{\Ds}{\mathscr{D}}

\newcommand{\Bs}{\mathscr{B}}

\newcommand{\Ec}{\mathcal{E}}

\newcommand{\Fs}{\mathscr{F}}

\newcommand{\Gf}{\mathfrak{G}}

\newcommand{\Kb}{\mathbb{K}}

\newcommand{\Nb}{\mathbb{N}}
\newcommand{\Pc}{\mathcal{P}}
\newcommand{\Cf}{\mathfrak{C}}
\newcommand{\Sc}{\mathcal{S}}
\newcommand{\Ss}{\mathscr{S}}
\newcommand{\Uc}{\mathcal{U}}

\newcommand{\eps}{\varepsilon}

\makeatletter
\def\@tocline#1#2#3#4#5#6#7{\relax
  \ifnum #1>\c@tocdepth 
  \else
    \par \addpenalty\@secpenalty\addvspace{#2}%
    \begingroup \hyphenpenalty\@M
    \@ifempty{#4}{%
      \@tempdima\csname r@tocindent\number#1\endcsname\relax
    }{%
      \@tempdima#4\relax
    }%
    \parindent\z@ \leftskip#3\relax \advance\leftskip\@tempdima\relax
    \rightskip\@pnumwidth plus4em \parfillskip-\@pnumwidth
    #5\leavevmode\hskip-\@tempdima
      \ifcase #1
       \or\or \hskip 1em \or \hskip 2em \else \hskip 3em \fi%
      #6\nobreak\relax
    \hfill\hbox to\@pnumwidth{\@tocpagenum{#7}}\par
    \nobreak
    \endgroup
  \fi}
\makeatother

	\title[Rychkov's Universal Extension Operator]{New Estimates of Rychkov's Universal Extension Operator for Lipschitz Domains and Some Applications}          

\author[]{Ziming Shi}
\address{Ziming Shi,Department of Mathematics, University of California - Irvine, Irvine, CA, 92697}  
\email{zimings3@uci.edu}
\author[]{Liding Yao} 
\address{Liding Yao, Department of Mathematics,
	The Ohio State University, Columbus, OH 43210}
\email[Corresponding Author]{yao.1015@osu.edu}

\subjclass[2020]{46E35 (primary), 42B25 and 46F10 (secondary)}

\begin{document}

\begin{abstract} 
Given a bounded Lipschitz domain $\Omega\subset\mathbb R^n$, Rychkov showed that there is a linear extension operator $\mathcal E$ for $\Omega$ which is bounded in Besov and Triebel-Lizorkin spaces. In this paper we introduce some new estimates for the extension operator $\mathcal E$ and give some applications. We prove the equivalent norms $\|f\|_{\mathscr A_{pq}^s(\Omega)}\approx\sum_{|\alpha|\le m}\|\partial^\alpha f\|_{\mathscr A_{pq}^{s-m}(\Omega)}$ for general Besov and Triebel-Lizorkin spaces. We also derive some quantitative smoothing estimates of the extended function and all its derivatives on $\overline\Omega^c$ up to the boundary. 
\end{abstract} 
\maketitle
\tableofcontents

\section{Introduction}
In this paper we prove the following result on the equivalence of norms.
\begin{thm}[Equivalent norms in Lipschitz domains]\label{Thm::Intro::Folklore!}
	Let $\Omega\subset\R^n$ be a bounded Lipschitz domain. Let $0<p,q\le\infty$ ($p<\infty$ for the $\Fs$-cases) and $s\in\R$. For any positive integer $m$, there is a $C=C_{\Omega,p,q,s,m}>0$ such that
	\begin{gather}\label{Eqn::Intro::Folklore!::Bs}
	C^{-1}\|f\|_{\Bs_{pq}^s(\Omega)}\le\sum_{|\alpha|\le m}\|\partial^\alpha f\|_{\Bs_{pq}^{s-m}(\Omega)}\le C\|f\|_{\Bs_{pq}^s(\Omega)},\\
	\label{Eqn::Intro::Folklore!::Fs}
	C^{-1}\|f\|_{\Fs_{pq}^s(\Omega)}\le\sum_{|\alpha|\le m}\|\partial^\alpha f\|_{\Fs_{pq}^{s-m}(\Omega)}\le C\|f\|_{\Fs_{pq}^s(\Omega)}.
	\end{gather}
\end{thm}
At first look, this appears to be a well-known result, but surprisingly 
no complete and correct proof has appeared in the literature to the best of the authors' knowledge. One can indeed find the proof in the case of $C^\infty$ domains (\cite[Theorem~3.3.5]{TriebelTheoryOfFunctionSpacesI}) and in the case of Sobolev and H\"older spaces (\cite[Chapter V]{SteinSingularIntegral}). See also the discussions in Section \ref{Section::HisRmkEqvNorm}.

In fact, Theorem~\ref{Thm::Intro::Folklore!} was stated by Triebel in \cite[Proposition~4.21]{TriebelWavelet}. However, there is a gap in his proof, in which the following statement is claimed: 
\begin{equation}\label{Eqn::Intro::EqvNormRmk}
    \|f\|_{\As_{pq}^s(\Omega)}\approx\|\Ec f\|_{\As_{pq}^s(\R^n)}\approx\sum_{|\alpha|\le m}\|\partial^\alpha \Ec f\|_{\As_{pq}^{s-m}(\R^n)}=\sum_{|\alpha|\le m}\|\Ec\partial^\alpha  f\|_{\As_{pq}^{s-m}(\R^n)}\lesssim\sum_{|\alpha|\le m}\|\partial^\alpha  f\|_{\As_{pq}^{s-m}(\Omega)}.
\end{equation}
Here $\Ec$ is an extension operator (see \eqref{Eqn::Intro::ExtOp} below) which is bounded on $\As_{pq}^s(\Omega)\to\As_{pq}^s(\R^n)$ for all $s\in\R$ and $0<p,q\le\infty$, where $p<\infty$ if $\As=\Fs$. The problem with this argument lies in the fact that the commutativity used in the equality $\partial^\alpha\circ\Ec=\Ec\circ\partial^\alpha$ in \eqref{Eqn::Intro::EqvNormRmk} (see \cite[(4.70)]{TriebelWavelet}) cannot be achieved, see Remark~\ref{Rmk::Intro::RmkComm}.

In order to fix the proof, our idea is to modify Triebel's argument by introducing certain operators $\Ec^{\alpha,\beta}$, which have ``similar structure'' and boundedness property as $\Ec$ (see Proposition~\ref{Prop::AntiDev::ForklorePrep} and Remark~\ref{Rmk::SecCor::RmkFolklore}) such that 
\begin{equation}\label{Eqn::Intro::Eab}
    \partial^\alpha\circ\Ec=\sum_{\beta}\Ec^{\alpha,\beta}\circ\partial^\beta.
\end{equation}

For $\Ec$, we use Rychkov's extension operator \cite{ExtensionLipschitz}: for a Lipschitz domain $\Omega$ there exists an operator $\Ec=\Ec_\Omega$ that is bounded on Besov spaces $\Ec:\Bs_{pq}^s(\Omega)\to\Bs_{pq}^s(\R^n)$ for all $0<p,q\le\infty$, $s\in\R$ and on Triebel-Lizorkin spaces $\Ec:\Fs_{pq}^s(\Omega)\to\Fs_{pq}^s(\R^n)$ for all $0<p<\infty$, $0< q\le \infty$, $s\in\R$. Since $\Ec$ is one single operator which works for all such spaces, we call it an \textbf{universal extension operator} on $\Omega$. 

By carefully studying Rychkov's construction, we also derive the following quantitative smoothing estimate. 
\begin{thm}\label{Prop::Intro::SmoothBddDom}
	Let $\Omega\subset\R^n$ be a bounded Lipschitz domain. Then Rychkov's extension operator $\Ec_\Omega$ (see \eqref{Eqn::SecCor::ConstructEc::DefEc}) satisfies the following estimate on the complement domain $\overline\Omega^c=\R^n \backslash \overline \Omega$:
	\begin{equation}\label{Eqn::SecCor::SmoothBddDom::Bdd}
	    [f\mapsto(\Ec_\Omega  f)|_{\overline{\Omega}^c}]:\Fs_{p\infty}^{s}(\Omega)\to W^{m,p}(\overline\Omega^c,\dist_\Omega^{m-s}),\quad\text{ for }1\le p\le\infty,\quad m=0,1,2,\dots,\quad\text{ and }s<m.
	\end{equation}
	
In particular, $\Ec_\Omega f\in C^\infty_\loc(\overline\Omega^c)$ for every $f\in \Ss'(\Omega)$.
\end{thm}

As a special case to Theorem~\ref{Prop::Intro::SmoothBddDom}, for $f$ in the Sobolev space $W^{k,p}(\Omega) (\subsetneq \Fs_{p\infty}^k(\Omega)$) where $1<p\le\infty$, we have $\Ec_\Omega  f\in W^{m,p}(\overline{\Omega}^c,\dist_\Omega^{m-k})$ for all $m> k$. This estimate is known if $\Ec_\Omega$ is the Stein extension operator \cite{SteinSingularIntegral} for $k\ge0$, see \cite[Theorem~6.3, Page 285]{BurenkovBook}. In our case we allow $k<0$ as well (while still requiring $m\ge 0$). The estimate is optimal in the sense that the derivatives of higher orders of $\Ec_\Omega f$ on $\overline{\Omega}^c$ have the minimal possible growth when approaching $\partial\Omega$, see \cite[Remark~6.15, Page 286]{BurenkovBook}.



\subsection{The general results for Rychkov-type operators}

We now explain Rychkov's construction of the universal extension operator on a bounded Lipschitz domain. By partition of unity it suffices to construct the operator on special Lipschitz domains. By a special Lipschitz domain we mean an open set $\omega\subset\R^n$ of the form $\omega=\{(x',x_n):x_n>\rho(x')\}$ where $\rho$ is a Lipschitz function such that $\|\nabla\rho\|_{L^\infty}<1$. (See Definition~\ref{Defn::Intro::SpeLipDom} and Remark~\ref{Rmk::Intro::RmkSpeDom}.) 

On a special Lipschitz domain $\omega$, Rychkov uses a double-convolution construction to define a universal extension operator $E_\omega$ on $\Ss'(\omega)=\{\tilde f|_\omega:\tilde f\in\Ss'(\R^n)\}$, the space of restrictions of tempered distributions to $\omega$, as 
\begin{equation}\label{Eqn::Intro::ExtOp}
E_\omega f: =\sum_{j=0}^\infty\psi_j\ast(\1_{\omega}\cdot(\phi_j\ast f)),\qquad f\in\Ss'(\omega).
\end{equation}
Here $(\psi_j)_{j=0}^\infty$ and $(\phi_j)_{j=0}^\infty$ are carefully chosen families of Schwartz functions that depend on $\phi_0,\psi_0,\phi_1,\psi_1$ and satisfy the following properties: 
\begin{itemize}
	\item \textit{Scaling condition}: $\phi_j(x)=2^{(j-1)n}\phi_1(2^{j-1}x)$ and $\psi_j(x)=2^{(j-1)n}\psi_1(2^{j-1}x)$ for $j\ge2$.
	\item \textit{Moment condition}: $\int\phi_0=\int\psi_0=1$, $\int x^\alpha\phi_0(x)dx=\int x^\alpha\psi_0(x)dx=0$ for all multi-indices $|\alpha|>0$, and $\int x^\alpha\phi_1(x)dx=\int x^\alpha\psi_1(x)dx=0$ for all $|\alpha|\ge0$.
	\item\textit{Approximate identity}: $\sum_{j=0}^\infty\psi_j\ast\phi_j=\delta_0$ is the Dirac delta measure, which is the identity element for the convolution.
	\item \textit{Support condition}: $\phi_0,\psi_0,\phi_1,\psi_1$ are all supported in the negative cone $-\Kb:=\{(x',x_n):x_n<-|x'|\}$.
\end{itemize}

The first three conditions show that $(\psi_j\ast\phi_j)_{j=0}^\infty$ induces a Littlewood-Paley decomposition on $\Ss'(\R^n)$.
The support condition guarantees $E_\omega$ is well-defined for functions defined only on $\omega$, since we have $\omega+\Kb\subseteq\omega$. That $E_\omega$ is an extension operator follows from the approximate identity and support condition:  $E_\omega f|_\omega=\sum_{j=0}^\infty\psi_j\ast\phi_j\ast f=f$, whenever the summation converges in the sense of distributions.

The smoothing effect of $E_\omega$ can be seen as follows: for any $f\in\Ss'(\omega)$, $E_\omega f$ is a smooth function on $\overline\omega^c$. This is because when $x_0\in\overline\omega^c$, the summand $\psi_j\ast(\1_{\omega}\cdot(\phi_j\ast f))(x)$ and all its derivatives decay rapidly in a neighborhood $x_0$.

\begin{rem}We note that there is no extension operator defined on $\Ds'(\omega)$, since a distribution $f\in\Ds'(\omega)$ admits an extension to $\R^n$ only when it has finite order (as a distribution) near $\partial\omega$. Cf. \cite[Page 4, (3)]{ExtensionLipschitz}.  

\end{rem}
To prove Theorems \ref{Thm::Intro::Folklore!} and \ref{Prop::Intro::SmoothBddDom},  we need a definition that generalizes the above construction of $(\phi_j,\psi_j)_j$. 
\begin{defn}\label{Defn::Intro::GenRes}
	A \textit{generalized Littlewood-Paley family} $\eta=(\eta_j)_{j=1}^\infty$ is a collection of Schwartz functions depending only on $\eta_1$, such that
	\begin{itemize}
		\item $\eta_1\in\Ss(\R^n)$ and all its moments vanish. That is, $\int x^\alpha \eta_1(x)dx=0$ for all multi-indices $\alpha$.
		\item $\eta_j(x)=2^{(j-1)n}\eta_1(2^{j-1}x)$ for $j\ge2$, $x\in\R^n$.
	\end{itemize}
	We denote by $\Gf=\Gf(\R^n)$ the set of all such families $\eta$.
	
	For an open set $V\subseteq\R^n$, we define $\Gf(V)$ as the set of all $(\eta_j)_{j=1}^\infty\in\Gf$ such that $\supp\eta_j\subset V$ for all $j\ge1$.
\end{defn}
Clearly, the map $\eta=(\eta_j)_{j=1}^\infty\mapsto\eta_1$ is a bijection between $\Gf$ and $\dot\Ss(\R^n)$, the space of Schwartz functions all of whose moments vanish (see Definition~\ref{Defn::Intro::S0Space}).

The main component of our proof relies on the following construction of a family of Rychkov-type operators \eqref{Eqn::Intro::ExtOp}:
\begin{thm}\label{Thm::MainThm}
Let $\omega\subset\R^n$ be a special Lipschitz domain. Denote $-\Kb=\{(x',x_n)\in\R^n:x_n<-|x'|\}$. 
Let $\eta=(\eta_j)_{j=1}^\infty,\theta=(\theta_j)_{j=1}^\infty\in\Gf(-\Kb)$. 
	For $r\in\R$, we define $T^{\eta,\theta,r}_\omega : \Ss'(\omega)\to\Ss'(\R^n)$ by 
	\begin{equation}\label{Eqn::MainThm::DefT}
	T^{\eta,\theta,r}_\omega f:=\sum_{j=1}^\infty 2^{jr}\eta_j\ast(\1_{\omega}\cdot(\theta_j\ast f)),\quad f\in\Ss'(\omega).
	\end{equation}
	Then $T^{\eta,\theta,r}_\omega$ is a well-defined linear operator in the sense that the sum in \eqref{Eqn::MainThm::DefT} converges in $\Ss'(\R^n)$. Furthermore,
	\begin{enumerate}[(i)]
		\item\label{Item::MainThm::Bdd}{\normalfont(Boundedness)} $T^{\eta,\theta,r}_\omega$ is bounded in Besov spaces $T^{\eta,\theta,r}_\omega:\Bs_{pq}^s(\omega)\to\Bs_{pq}^{s-r}(\R^n)$ for all $0< p,q\le\infty$, $s\in\R$ and in Triebel-Lizorkin spaces $T^{\eta,\theta,r}_\omega:\Fs_{pq}^s(\omega)\to\Fs_{pq}^{s-r}(\R^n)$ for all $0< p<\infty$, $0< q\le\infty$, $s\in\R$.

		\item\label{Item::MainThm::AntiDev}{\normalfont(Anti-derivatives)} For any positive integer $m$, we can find some $\tilde \eta^\alpha=(\tilde\eta^\alpha_j)_{j=1}^\infty,\tilde \theta^\alpha=(\tilde\theta^\alpha_j)_{j=1}^\infty\in\Gf(-\Kb)$ for multi-indices $|\alpha|=m$ such that
		\begin{equation}\label{Eqn::MainThm::AtDChar}
		\eta_j=2^{-jm}\sum_{|\alpha|=m}\partial^\alpha\tilde\eta^\alpha_j,\quad\theta_j=2^{-jm}\sum_{|\alpha|=m}\partial^\alpha\tilde\theta^\alpha_j,\quad j=1,2,3,\dots.
		\end{equation} 		 
		\begin{equation*}
		T^{\eta,\theta,r}_\omega=\sum_{|\alpha|=m}\partial^\alpha\circ T^{\tilde \eta^\alpha,\theta,r-m}_\omega=\sum_{|\alpha|=m}T^{\eta,\tilde \theta^\alpha,r-m}_\omega\circ\partial^\alpha.
		\end{equation*}
		\item\label{Item::MainThm::Smoothing}{\normalfont(Quantitive smoothing)} By restricting the target functions to $\overline\omega^c$, $T^{\eta,\theta,r}_\omega$ satisfies the estimate 
	\begin{equation}\label{Eqn::MainThm::SmooBdd}
		\big[f\mapsto (T^{\eta,\theta,r}_\omega f)|_{\overline\omega^c}\big]:\Fs_{p\infty}^{s}(\omega)\to W^{m,p}(\overline\omega^c,\dist_\omega^{m+r-s}),
	\end{equation}
		for all\footnote{ Here when $p=\infty$, $\Fs_{\infty\infty}^s=\Bs_{\infty\infty}^s$ are also Besov spaces. See Remark~\ref{Rmk::Prem::BsFsRmk}~\ref{Item::Prem::BsFsRmk::InftyInfty}} $1\le p\le\infty$, $m=0,1,2,\dots$ and $s< m+r$.  
		Here $\dist_\omega(x)=\dist(x, \omega)$.
	\end{enumerate}
\end{thm}

In \ref{Item::MainThm::Smoothing} we use the notation of weighted Sobolev spaces from \cite[Section 7.2.3]{TriebelTheoryOfFunctionSpacesI}. See \eqref{Eqn::Intro::WeiSobo} below.

\begin{rem}
    The operators $T^{\eta,\theta,r}_\omega$ from \eqref{Eqn::MainThm::DefT} essentially generalize $E_\omega$ in \eqref{Eqn::Intro::ExtOp}, since we can write $E_\omega$ as 
\begin{equation}\label{Eqn::Intro::EfromT}
E_\omega f= \psi_0\ast(\1_\omega\cdot(\phi_0\ast f))+T^{\psi,\phi,0}_\omega f,\quad\text{where }\psi=(\psi_j)_{j=1}^\infty\text{ and }\phi=(\phi_j)_{j=1}^\infty.
\end{equation}
The operator $f\mapsto\psi_0\ast(\1_\omega\cdot(\phi_0\ast f))$ takes tempered distributions to smooth functions and is bounded on all good function spaces we consider; see Proposition~\ref{Prop::Prem::BddE0} \ref{Item::Prem::BddE0::F0}.
\end{rem}

We now explain how Theorem~\ref{Thm::MainThm} is used to prove Theorem~\ref{Thm::Intro::Folklore!}. 


Indeed, applying Theorem~\ref{Thm::MainThm} \ref{Item::MainThm::AntiDev} with $\eta=\psi$ and $\theta=\phi$, we can write $E_\omega$ in the following way:
\begin{align}\label{Eqn::Intro::ExtAtd2}
E_\omega f&=\psi_0\ast(\1_\omega\cdot(\phi_0\ast f))+\sum_{|\beta|=m} T^{\psi,\tilde\theta^\beta,-m}_\omega (\partial^\beta f),\quad f\in\Ss'(\omega),\quad m=1,2,3,\dots.
\end{align}

By Theorem~\ref{Thm::MainThm} \ref{Item::MainThm::Bdd},  $T^{\psi,\tilde\theta^\beta,-m}_\omega$ gains $m$-derivatives in Besov and Triebel-Lizorkin spaces. 
Taking $\Ec^{\alpha,0}f:=\psi_0\ast(\1_\omega\cdot(\phi_0\ast f))$ and $\Ec^{\alpha,\beta}:=\partial^\alpha\circ T^{\psi,\tilde\theta^\beta,-m}_\omega$ for $|\beta|=|\alpha|$, we get \eqref{Eqn::Intro::Eab} in the case $\Omega$ is a special Lipschitz domain.
This proves \eqref{Eqn::Intro::Folklore!::Bs} and \eqref{Eqn::Intro::Folklore!::Fs} for special Lipschitz domains (Proposition~\ref{Prop::AntiDev::ForklorePrep}). The proof for bounded Lipschitz domains then follows from standard partition of unity argument, see Section \ref{Section::SecCor::AtdBddDom}.

\subsection{An application to several complex variables}  Applying Theorem~\ref{Thm::MainThm} \ref{Item::MainThm::AntiDev} in a different way, we have the following formula analogous to \eqref{Eqn::Intro::ExtAtd2}:
\begin{equation}\label{Eqn::Intro::ExtAtd1}
  \tilde f=\psi_0\ast(\1_\omega\cdot(\phi_0\ast\tilde f))+\sum_{|\beta|=m}\partial^\beta\Big(\sum_{j=1}^\infty\tilde\eta^\beta_j\ast(\1_\omega\cdot(\phi_j\ast\tilde f))\Big),\quad\tilde  f\in\Ss'(\R^n),\quad m=1,2,3,\dots.
\end{equation}

If  $\tilde f|_\omega \equiv 0$ then $\psi_0\ast(\1_\omega\cdot(\phi_0\ast\tilde f))|_\omega=\tilde\eta^\beta_j\ast(\1_\omega\cdot(\phi_j\ast\tilde f))|_\omega \equiv 0$ as well. This leads to the following:

\begin{prop}[Anti-derivatives with support constraint]\label{Prop::Intro::AntiDevBddDom}
	Let $\Omega\subset\R^n$ be a bounded Lipschitz domain. Then for any positive integer $m$, there exist linear operators $\Sc^{m,\alpha}_\Omega:\Ss'(\R^n)\to\Ss'(\R^n)$, $|\alpha|\le m$, such that
	\begin{enumerate}[(i)]
		\item\label{Item::SecCor::AntiDevBddDom::Bdd} $\Sc^{m,\alpha}_\Omega:\Bs_{pq}^s(\R^n)\to\Bs_{pq}^{s+m}(\R^n)$ for all $0<p,q\le\infty$, $s\in\R$ and $\Sc^{m,\alpha}_\Omega:\Fs_{pq}^s(\R^n)\to\Fs_{pq}^{s+m}(\R^n)$ for all $0<p<\infty$, $0<q\le\infty$, $s\in\R$.
		\item\label{Item::SecCor::AntiDevBddDom::Sum} $g=\sum_{|\alpha|\le m}\partial^\alpha(\Sc^{m,\alpha}_\Omega g)$ for all $g\in\Ss'(\R^n)$.
		\item\label{Item::SecCor::AntiDevBddDom::Vanish} If $g\in\Ss'(\R^n)$ satisfies $g|_\Omega \equiv 0$, then $(\Sc^{m,\alpha}_\Omega g)|_\Omega \equiv 0$ for all $|\alpha|\le m$.
	\end{enumerate}
\end{prop}

In a recent paper \cite{ShiYaoCk}, the authors applied Proposition~ \ref{Prop::Intro::AntiDevBddDom} to several complex variables and obtained a new Sobolev estimate for the $\overline\partial$ equation on strongly pseudoconvex domains. For the equation $\overline\partial u = \varphi$, our solution operator contains an integral of the form
\begin{equation}\label{Eqn::Intro::SCVInt}
    \int_{\Uc\backslash\overline\Omega} K(z, \cdot) \wedge[\overline\partial, \Ec_\Omega] \varphi, \quad z\in\Omega,
\end{equation}
where $\Uc$ is some small neighborhood of $\overline\Omega$.

To estimate the integral, the differential form $\varphi$ cannot be too low in regularity, in particular we need $\varphi$ to be in $H^{s,p}$, where $s$ cannot be negative. In order to estimate the solution when $\varphi \in H^{s,p}$, for $s<0$, we apply Proposition~\ref{Prop::Intro::AntiDevBddDom} to the commutator so that the integral becomes
\begin{equation*}\label{Eqn::Intro::SCVInt2}
    \sum_{|\alpha|\le m}\int_{\Uc \backslash\overline{\Omega}} K(z, \cdot) \wedge
  \partial^{\alpha} \Sc^{m, \alpha}_\Omega([\overline\partial, \Ec_\Omega] \varphi),\quad z\in\Omega.
\end{equation*}
Using ``integration by parts'', we can move the derivatives to the kernel:
\begin{equation*}  
  \sum_{|\alpha|\le m}(-1)^{|\alpha|} \int_{\Uc\backslash\overline\Omega} \partial^{\alpha} (K (z, \cdot)) \wedge
  \Sc^{m, \alpha}_\Omega([\overline\partial, \Ec_\Omega] \varphi),\quad z\in\Omega.
\end{equation*}

Now, the commutator $[\overline\partial, \Ec_\Omega] \varphi = \overline\partial (\Ec_\Omega \varphi) - \Ec_\Omega (\overline\partial \varphi)$ has an important property that it vanishes identically inside $\Omega$. By Proposition~\ref{Prop::Intro::AntiDevBddDom} with $m$ chosen to be sufficiently large, we can make $ \Sc^{m, \alpha}_\Omega([\overline\partial, \Ec_\Omega] \varphi) \in H^{s,p}(\Omega)$ for $s>0$, while $\Sc^{m, \alpha}_\Omega([\overline\partial, \Ec_\Omega] \varphi)$ still vanishes inside $\Omega$. As a result, $\Sc^{m, \alpha}_\Omega([\overline\partial, \Ec_\Omega] \varphi)(\zeta)$ will contribute some vanishing order as $\zeta\in\overline{\Omega}^c$ approaches the boundary, which cancels with the singularity of the kernel $\partial^{\alpha}_\zeta K(z,\zeta)$ and enables us derive the sharp estimate. 

\begin{rem}\label{Rmk::Intro::RmkComm}
 On any bounded strongly pseudoconvex domain $\Omega\subset\C^n$, for $1\le q\le n-1$ one can construct a $\overline{\partial}$-closed continuous $(0,q)$-form $\varphi$ such that any solution $u$ to the equation $\overline{\partial}u=\varphi$ is not in
 $\Co^{\frac12+\eps}(\Omega)$ for all $\eps>0$ (see e.g. \cite[Section V.2.5]{RangeSCV}). This is one way to see that $[\overline{\partial},\Ec_\Omega] = \overline\partial \circ \Ec_\Omega- \Ec_\Omega \circ \overline\partial$ in \eqref{Eqn::Intro::SCVInt} cannot be the zero operator, since otherwise the solution operator can gain one derivative (see e.g. \cite[Section 2]{Gong}). 
\end{rem}

\subsection{Terminology and organization}
We denote by $\N=\{0,1,2,\dots\}$ the set of non-negative integers, and $\Z_+=\{1,2,3,\dots\}$ the set of positive integers.

Let $E\subset\R^n$ be an arbitrary set; we use $\1_E:\R^n\to\{0,1\}$ for the characteristic function for $E$.

We will use the notation $x \lesssim y$ to mean that $x \leq Cy$ where $C$ is a constant independent of $x,y$, and $x \approx y$ for ``$x \lesssim y$ and $y \lesssim x$''. 

We write $\Ss(\R^n)$ as the space of Schwartz functions, and $\Ss'(\R^n)$ as the space of tempered distributions. For an arbitrary open set $\Omega\subseteq\R^n$, we denote by $\Ss'(\Omega)=\{\tilde f|_\Omega:\tilde f\in\Ss'(\R^n)\}\subset\Ds'(\Omega)$ the space of restrictions of tempered distributions on $\R^n$ to $\Omega$. There is also an intrinsic definition for $\Ss'(\Omega)$, see \cite[(3.1) and Proposition~3.1]{ExtensionLipschitz}.

\begin{defn}\label{Defn::Intro::S0Space}
	We use $\dot\Ss(\R^n)$ to denote the space\footnote{This spaces is also denoted by $Z(\R^n)$, see for example \cite[Section 5.1.2]{TriebelTheoryOfFunctionSpacesI}.} of all infinite order moment vanishing Schwartz functions. That is, all $f\in\Ss(\R^n)$ such that $\int x^\alpha f(x)dx=0$ for all $\alpha\in\N^n$, or equivalently, all $f\in\Ss(\R^n)$ such that its Fourier transform satisfies $\partial^\alpha\widehat f(0)=0$ for all $\alpha\in\Nb^n$. 
\end{defn}

Note that $\dot\Ss(\R^n)$ is closed under convolutions.

\begin{note}\label{Note::Intro::CanCone}
	In $\R^n$ we use the $x_n$-directional cone $\Kb:=\{(x',x_n):x_n>|x'|\}$ and its reflection $$-\Kb:=\{(x',x_n):x_n<-|x'|\}.$$
\end{note}
\begin{defn}\label{Defn::Intro::SpeLipDom}
	A \emph{special Lipschitz domain} $\omega\subset\R^n$ is a domain of the form $\omega=\{(x',x_n):x_n>\rho(x')\}$ where $\rho:\R^{n-1}\to\R$ is a Lipschitz function with $\|\nabla\rho\|_{L^\infty}<1$. 
\end{defn}
\begin{rem}\label{Rmk::Intro::RmkSpeDom}
	In many literature, for example \cite[Section 1.11.4 (1.322) p. 63]{TriebelTheoryOfFunctionSpacesIII}, the definition for a special Lipschitz domain only requires $\rho$ to be a Lipschitz function. In other words,  $\|\nabla\rho\|_{L^\infty}$ is finite but can be arbitrary large.
	By taking invertible linear transformations we can make $\nabla\rho$ small in the new coordinates. 
	
  The assumption $\|\nabla\rho\|_{L^\infty}<1$ is necessary in order to guarantee the property $\omega+\Kb\subseteq\omega$, which is convenient in constructing the extension operator.
\end{rem}

We define a \emph{bounded Lipschitz domain} to be an open set $\Omega\subset\R^n$ such that the boundary $\partial\Omega$ is locally a Lipschitz graph. See also \cite[Definition 1.103(ii)]{TriebelTheoryOfFunctionSpacesIII}. 

\begin{conv}\label{Conv::Intro::KDyaPair}
    A \textit{$\Kb$-Littlewood-Paley pair} (abbreviated as $\Kb$-pair) is a collection of Schwartz functions $(\phi_j,\psi_j)_{j=0}^\infty$ depending only on $\phi_0,\psi_0,\phi_1,\psi_1$, such that
    \begin{itemize}
        \item $\supp\phi_j,\supp\psi_j\subset -\Kb\cap\{x_n<-2^{-j}\}$ for all $j\ge0$;
        \item $\phi=(\phi_j)_{j=1}^\infty$ and $\psi=(\psi_j)_{j=1}^\infty\in \Gf(-\Kb)$;
        \item $\sum_{j=0}^\infty\phi_j=\sum_{j=0}^\infty\psi_j\ast\phi_j=\delta_0$ is the Dirac delta measure at $0\in\R^n$.
    \end{itemize} 
\end{conv}

Such $(\phi_j,\psi_j)_j$ exists due to \cite[Theorem~4.1(a)]{ExtensionLipschitz}. Here the condition $\supp\phi_j,\supp\psi_j\subset\{x_n<-2^{-j}\}$ can be obtained by scaling (see Lemma~\ref{Lem::Smooth::ScalingRmk}).

\medskip
Let $\Omega\subseteq\R^n$ be an arbitrary open set. Let $\varphi:\Omega\to\R^+$ be a positive continuous function, we define
\begin{equation}\label{Eqn::Intro::WeiSobo}
\begin{gathered}
W^{m,p}(\Omega,\varphi):=\{f\in W^{m,p}_\loc(\Omega):\|f\|_{W^{m,p}(\Omega,\varphi)}<\infty\},
\\
\|f\|_{W^{m,p}(\Omega,\varphi)}:=\bigg(\sum_{|\alpha|\le m}\int_\Omega |\varphi\partial^\alpha f|^p\bigg)^\frac1p \quad 1 \le p < \infty;\quad
\|f\|_{W^{m,\infty}(\Omega,\varphi)}:=\sup\limits_{|\alpha|\le m}\|\varphi\partial^\alpha f\|_{L^\infty(\Omega)}.
\end{gathered}
\end{equation}

Let $0<p,q<\infty$. For a vector-valued function $g=(g_j)_{j\in\Z}:\Omega\to\ell^q(\Z)$, we adopt the following notations for (quasi-)norms $\|\cdot\|_{\ell^q(L^p)}=\|\cdot\|_{\ell^q(\Z;L^p(\Omega))}$ and $\|\cdot\|_{L^p(\ell^q)}=\|\cdot\|_{L^p(\Omega;\ell^q(\Z))}$ as
\begin{equation*}
\|g\|_{\ell^q(L^p)} :=\bigg(\sum_{j\in\Z}\Big(\int_\Omega |g_j(x)|^pdx\Big)^\frac qp\bigg)^\frac1q,\quad \|g\|_{L^p(\ell^q)}:= \bigg(\int_\Omega\Big(\sum_{j\in\Z} |g_j(x)|^q\Big)^\frac pqdx\bigg)^\frac1p.
\end{equation*}
For $p=\infty$ or $q=\infty$ we use the usual modification.

When we are dealing with the boundedness of Besov spaces and Triebel-Lizorkin spaces, we use the following often-used notion of admissible indices:
\begin{conv}\label{Conv::Intro::AdmissibleIndex}
	Let $U,V\subseteq\R^n$ and $s,t\in\R$. We say an operator $T$ is bounded linear $T:\As_{pq}^s(U)\to\As_{pq}^t(V)$ for all \textit{admissible indices} $(p,q)$, if $T:\Bs_{pq}^s(U)\to\Bs_{pq}^t(V)$ is bounded linear for all $0<p,q\le\infty$, and $T:\Fs_{pq}^s(U)\to\Fs_{pq}^t(V)$ is bounded linear for all $0<p<\infty $, $0<q\le\infty$, and $p=q=\infty$. We shall always use $\As$ as either the Besov class $\Bs$ or the Triebel-Lizorkin class $\Fs$.
\end{conv}



\medskip
The paper is organized as follows.

In Section \ref{Section::Prelim}, we review some basic properties of Besov and Triebel-Lizorkin spaces via Littlewood-Paley decomposition, and we prove Theorem~\ref{Thm::MainThm} \ref{Item::MainThm::Bdd} using similar arguments as in \cite{RychkovThmofBui,ExtensionLipschitz}.
In Section \ref{Section::AntiDev}, we prove Theorem~\ref{Thm::MainThm} \ref{Item::MainThm::AntiDev} using an argument of homogeneous Littlewood-Paley decomposition. 
In Section \ref{Section::Smoothing}, we prove Theorem~\ref{Thm::MainThm} \ref{Item::MainThm::Smoothing}. The proof is somewhat technical, so we first give an outline in Section \ref{Section::SmoothingIdea} to provide the reader with the main ideas of the proof, and then we fill in all the details in Section \ref{Section::SmoothingPrecise}. 
In Section \ref{Section::SecCor}, we prove the results for bounded Lipschitz domains. We prove Proposition~\ref{Thm::Intro::SmoothSpeDom} using Theorem~\ref{Thm::MainThm}  \ref{Item::MainThm::Smoothing}; and we prove Theorem~\ref{Thm::Intro::Folklore!} and Proposition~\ref{Prop::Intro::AntiDevBddDom} using Theorem~\ref{Thm::MainThm} \ref{Item::MainThm::Bdd} and \ref{Item::MainThm::AntiDev}.

\section{Historical Remarks}\label{Section::HisRmk}
\subsection{Some previous extension operators}

When $\Omega$ is a bounded $C^\infty$-domain, for any $\eps>0$ one can find an extension operator $S=S^\eps$ that is bounded linear $S:\Bs_{pq}^s(\Omega)\to\Bs_{pq}^s(\R^n)$ and $S:\Fs_{pq}^s(\Omega)\to\Fs_{pq}^s(\R^n)$ for all $-\eps^{-1}<s<\eps^{-1}$, $\eps<p\le\infty$ and $0<q\le\infty$ ($p<\infty$ for the $\Fs$-cases). See \cite[Section 2.9]{TriebelTheoryOfFunctionSpacesI} and \cite[Section 1.11.5]{TriebelTheoryOfFunctionSpacesIII}. The construction is based on the half-plane reflections: for $f:\R^n_+\to\R$, we define
\begin{equation}\label{Eqn::Intro::HalfPlaneExt}
S f(x',x_n):=\begin{cases}f(x',x_n)&x_n>0,
\\\sum_{j=1}^Ma_jf(x',-b_jx_n)    &x_n<0. 
\end{cases}
\end{equation}
Here $M$ is some large integer (depending on $\eps$), $a_j\in\R$ and $b_j>0$ are numbers satisfying $\sum_{j=1}^Ma_j(-b_j)^k=1$ for all integers $k$ with $-\frac M2<k<\frac M2$.

The half-plane construction can be traced back to Lichtenstein \cite{LichtensteinHalfPlane}.  By modifying \eqref{Eqn::Intro::HalfPlaneExt}, Seeley \cite{SeeleyHalfPlane} constructed an extension operator that is bounded $C^\infty(\overline{\R^n_+})\to C^\infty(\R^n)$.
	
The boundedness of $S$ on Besov and Triebel-Lizorkin spaces are mostly done by Triebel (see \cite[Section 4.5.1]{TriebelTheoryOfFunctionSpacesII}). See also \cite[Section 6.4]{YSYMorrey} for some boundedness results of $S$ on hybrid spaces.

Though $S=S^\eps$ is defined on Besov and Triebel-Lizorkin spaces with index $s\in(-\eps^{-1},0]$, in general $Sf|_{\R^n_-}$ is not more regular than $f$ as it reflects all the singularities of $f$. In particular, unlike Rychkov's extension operator, we cannot expect $Sf|_{\R^n_-}$ to be a smooth function if $f$ is not smooth.

\medskip
When $\Omega$ is a bounded Lipschitz domain, for $k\in\N$, $1\le p\le \infty$, we define Sobolev space $W^{k,p}(\Omega)$ in the classical way, which is \eqref{Eqn::Intro::WeiSobo} with $\varphi=1$.

Calder\'on \cite{CalderonExtension} showed that for each $k$ there is a bounded linear extension operator $\Ec^k:W^{k,p}(\Omega)\to W^{k,p}(\R^n)$ for all $1<p<\infty$. Later Stein \cite{SteinSingularIntegral} used an alternative approach to construct an extension operator $\Ec^{S}$ which is independent of $k$, such that $\Ec^{S}:W^{k,p}(\Omega)\to W^{k,p}(\R^n)$ is bounded for all $k\in\N$ and  $1\le p\le \infty$. 

Both results show that $W^{k,p}(\Omega)=\{\tilde f|_\Omega:\tilde f\in W^{k,p}(\R^n)\}$. Indeed $W^{k,p}(\R^n)=\Fs_{p2}^k(\R^n)$ for $1<p<\infty$ (see \cite[Section 2.5.6]{TriebelTheoryOfFunctionSpacesI}), so by definition (see Definition~\ref{Defn::Prem::BsFsDef}) we get 
\begin{equation}\label{Eqn::Intro::Wkp=Fp2Domain}
\Fs_{p2}^k(\Omega)=\{\tilde f|_\Omega:\tilde f\in \Fs_{p2}^k(\R^n)\}=\{\tilde f|_\Omega:\tilde f\in W^{k,p}(\R^n)\}=W^{k,p}(\Omega).
\end{equation}

The reader can also refer to \cite{KalyabinExtension} for some discussion of Stein's extension operator on Triebel-Lizorkin spaces.

\medskip
For some domains whose boundaries are not Lipschitz, it is still possible to define useful extension operators. Jones \cite{JonesLocallyUniform} introduced \textit{locally uniform domains}, for which Lipschitz domains are a special case. For certain locally uniform domains $\Omega$ (in particular all Lipschitz domains) and for each $k\in\N$, Jones constructed an extension operator that maps $W^{k,p}(\Omega)$ to $W^{k,p}(\R^n)$ for all $1\le p\le \infty$. Later, the construction was refined by Rogers \cite{RogersLocallyUniform} where the extension operator was chosen to be independent of $k$. 

We also refer the reader to \cite{ChristExtension}, \cite[Theorem~2]{SeegerNote} and \cite{ExtBMO} for more discussions on Jones' extension.
It would be interesting to see whether the results of our paper can be extended to certain locally uniformly domains. 

There are also a few studies on the case where the domain $\Omega\subseteq\R^n$ is an \textit{arbitrary subset} with no assumption on the boundary, starting from \cite{FeffermanExtension}. For an arbitrary one-dimensional subset $\Omega\subseteq\R^1$, in \cite[Proposition~3.5]{IsraelExtension} there is a Besov extension $E=E_\Omega:\Bs_{pp}^{2-1/p}(\Omega)\to\Bs_{pp}^{2-1/p}(\R)$ for $2<p<\infty$. For higher dimension, in  \cite{IsraelExtension,FeffermanIsraelLuli}, it was shown that for any $n<p<\infty$ and integer $k\ge1$, there exists a Sobolev extension $E=E_{\Omega,k,p}:\Fs_{p2}^k(\Omega)\to\Fs_{p2}^k(\R^n)$.  It is an open problem whether we can make the extension operators independent of $k$. 


\subsection{Two approaches to equivalent norms on domains}\label{Section::HisRmkEqvNorm}

As mentioned earlier, Theorem~\ref{Thm::Intro::Folklore!} was known for bounded $C^\infty$ domains with all $\Bs_{pq}^s$, $\Fs_{pq}^s$, and for bounded Lipschitz domains with $\Bs_{\infty\infty}^s$, $\Fs_{p2}^k$ for $s>m$ and integer $k\ge m$. The proofs for these two special cases have different approaches. We sketch them below.

\medskip
For the case where $\Omega$ is bounded $C^\infty$ domain, by partition of unity and composing with diffeomorphisms, we only need to prove \eqref{Eqn::Intro::Folklore!::Bs} and \eqref{Eqn::Intro::Folklore!::Fs} for the half plane $\Omega=\R^n_+$. On half plane the idea is the following: for $\alpha\neq0$ there is an operator $S_\alpha$ (which turns out to be an extension operator as well) such that
\begin{enumerate}[label=(S.\alph*)]
	\item $\partial^\alpha\circ S=S_\alpha\circ\partial^\alpha$.
	\item $S_\alpha$ has the similar expression to \eqref{Eqn::Intro::HalfPlaneExt}: we can find some numbers $(a^\alpha_j,b^\alpha_j)_{j=1}^M$ such that
	\begin{equation}\label{Eqn::Intro::SAlphaEqn}
	\textstyle S_\alpha f(x',x_n)=f(x',x_n),\quad x_n>0;\qquad S_\alpha f(x',x_n)=\sum_{j=1}^Ma^\alpha_jf(x',-b^\alpha_jx_n),\quad x_n<0.
	\end{equation}
\end{enumerate}
Similar to $S$, one can prove the  boundedness for $S_\alpha$ in Besov and Triebel-Lizorkin spaces.
Therefore for every $m\ge1$ and for $\As\in\{\Bs,\Fs\}$,
\begin{align*}
\|f\|_{\As_{pq}^s(\R^n_+)}&\le\|Sf\|_{\As_{pq}^s(\R^n)}\approx\sum_{|\alpha|\le m}\|\partial^\alpha Sf\|_{\As_{pq}^{s-m}(\R^n)}=\sum_{|\alpha|\le m}\|S_\alpha\partial^\alpha f\|_{\As_{pq}^{s-m}(\R^n)}\lesssim\sum_{|\alpha|\le m}\|\partial^\alpha f\|_{\As_{pq}^{s-m}(\R^n_+)}.
\end{align*}
This proves the essential part of \eqref{Eqn::Intro::Folklore!::Bs} and \eqref{Eqn::Intro::Folklore!::Fs} with $\Omega=\R^n_+$. See \cite[Theorem~3.3.5]{TriebelTheoryOfFunctionSpacesI} for details.

To prove Theorem~\ref{Thm::Intro::Folklore!}, we use a similar argument along with the decomposition \eqref{Eqn::Intro::ExtAtd2}. See Proposition~\ref{Prop::AntiDev::ForklorePrep} and Remark~\ref{Rmk::AntiDev::ForklorePrep}.

\medskip
In the case of Sobolev spaces $W^{k,p}(\Omega)$, where $k\in\N$ and $\Omega$ is a bounded Lipschitz domain, Theorem~\ref{Thm::Intro::Folklore!} can be proved easily by using the extension operator and the intrinsic characterization of the spaces: for every $m\in\Z_+$ and integer $k\ge m$,
\begin{equation}\label{Eqn::Intro::WkpEqvNorm}
\|f\|_{W^{k,p}(\Omega)}\approx\sum_{|\alpha|\le k}\|\partial^\alpha f\|_{L^p(\Omega)}\approx\sum_{|\alpha|\le m}\sum_{|\beta|\le k-m}\|\partial^\alpha\partial^\beta f\|_{L^p(\Omega)}\approx\sum_{|\alpha|\le m}\|\partial^\alpha f\|_{W^{k-m,p}(\Omega)}.
\end{equation}
On the other hand, we have $W^{k,p}(\Omega)=\{\tilde f|_\Omega:\tilde f\in W^{k,p}(\R^n)\}$ (see \cite[Theorem~1.122 (ii)]{TriebelTheoryOfFunctionSpacesIII}). Since $W^{k,p}(\R^n)=\Fs_{p2}^k(\R^n)$ (see \cite[Section 2.5.6]{TriebelTheoryOfFunctionSpacesI}) we have $W^{k,p}(\Omega)=\{\tilde f|_\Omega:\tilde f\in \Fs_{p2}^s(\R^n)\}=\Fs_{p2}^k(\Omega)$. So \eqref{Eqn::Intro::Folklore!::Fs} is done for the $\Fs_{p2}^k$-spaces for $1<p<\infty$ with integer $k\ge m$.

Note that for the $\Fs_{p2}^k$-spaces, \eqref{Eqn::Intro::Folklore!::Fs} is also true when $\Omega$ is a locally uniform domain (which may not be Lipschitz) that satisfies the hypothesis used by \cite{RogersLocallyUniform}. 
For general Triebel-Lizorkin class $\Fs_{pq}^s$, however, we do not have a good intrinsic characterization like \eqref{Eqn::Intro::Folklore!::Fs}, so this approach does not work.

Similarly for the H\"older-Zygmund spaces $\Co^s(\Omega)=\{f\in C^0(\overline{\Omega}):\|f\|_{\Co^s(\Omega)}<\infty\}$ we can use the norm
\begin{align}\notag
\|f\|_{\Co^s(\Omega)}&\textstyle:=\sup|f|+\sup_{x,y\in\Omega;x\neq y}\frac{|f(x)-f(y)|}{|x-y|^s},&\text{provided }0<s<1,
\\\notag
\|f\|_{\Co^s(\Omega)}&\textstyle:=\sup|f|+\sup_{x,y\in\Omega;x\neq y,\frac{x+y}2\in\Omega}\frac{|f(x)+f(y)-2f(\frac{x+y}2)|}{|x-y|},&\text{provided }s=1,
\\\label{Eqn::Intro::CsNorm}
\|f\|_{\Co^s(\Omega)}&\textstyle:=\|f\|_{\Co^{s-1}(\Omega)}+\sum_{j=1}^n\|\partial_{x_j}f\|_{\Co^{s-1}(\Omega)},&\text{provided }s>1.
\end{align}
By \eqref{Eqn::Intro::CsNorm}, we see that $\|f\|_{\Co^s(\Omega)}\approx\sum_{|\alpha|\le m}\|\partial^\alpha f\|_{\Co^{s-m}(\Omega)}$ if $s>m$.
Again by $\Co^s(\Omega)=\{\tilde f|_\Omega:\tilde f\in\Co^s(\R^n)\}$ (see \cite[Theorem~1.122 (ii)]{TriebelTheoryOfFunctionSpacesIII}), since $\Co^s(\R^n)=\Bs_{\infty\infty}^s(\R^n)$ (see \cite[Section 2.5.7]{TriebelTheoryOfFunctionSpacesI}) we have $\Co^s(\Omega)=\{\tilde f|_\Omega:\tilde f\in \Bs_{\infty\infty}^s(\R^n)\}=\Bs_{\infty\infty}^s(\Omega)$, and \eqref{Eqn::Intro::Folklore!::Bs} is done for the $\Bs_{\infty\infty}^s$-spaces when $s>m$.

\section{Boundedness On Besov and Triebel-Lizorkin Spaces} \label{Section::Prelim}

\subsection{Preliminary and preparations}

We recall the characterizations of Besov and Triebel-Lizorkin spaces via the Littlewood-Paley decomposition.
\begin{defn}\label{Defn::Prem::DyaRes}
A \textit{classical dyadic resolution} is a sequence $\lambda=(\lambda_j)_{j=0}^\infty$ of Schwartz functions on $\R^n$, denoted by $\lambda\in\Cf$, such that the Fourier transforms $\widehat\lambda_j(\xi)=\int_{\R^n}\lambda_j(x)e^{-2\pi ix \cdot \xi} dx$ satisfies
\begin{itemize}
    \item $\widehat\lambda_0\in C_c^\infty\{|\xi|<2\}$, $\widehat\lambda_0|_{\{|\xi|<1\}}\equiv1$.
    \item $\widehat\lambda_j(\xi)=\widehat\lambda_0(2^{-j}\xi)-\widehat\lambda_0(2^{-(j-1)} \xi)$ for $j\ge1$ and $\xi\in\R^n$.
\end{itemize}
\end{defn}
\begin{rem}\label{Rmk::Prem::DyaResRmk}
Note that for each $j \geq 1$, $\widehat\lambda_j$ is supported on the annulus 
$ \{ 2^{j-1} < |\xi| < 2^{j+1} \}$. Thus for any $j \geq 0$ we have $\sum_{k= 0}^{\infty} \widehat{\lambda}_k|_{\supp \widehat\lambda_j} = \sum_{k=j-1}^{j+1} \widehat\lambda_k|_{\supp\widehat\lambda_j}\equiv1$, where we set $\lambda_{-1}=0$. 
If we denote $\mu_j:=\sum_{k=j-1}^{j+1}\lambda_j$ for $j\ge0$, then 
    \begin{enumerate}[(a)]
        \item\label{Item::Prem::DyaResRmk::Conv} $\lambda_j=\mu_j\ast \lambda_j$ for all $j\ge0$.
        \item\label{Item::Prem::DyaResRmk::Scaling} $\mu_j(x)=2^{(j-2)n}\mu_2(2^{j-2}x)$ for $j\ge2$.
    \end{enumerate}
\end{rem}

For any $f\in\Ss'(\R^n)$, we have the decomposition $f=\sum_{j=0}^\infty(\widehat\lambda_j\widehat f)^\vee=\sum_{j=0}^\infty\lambda_j\ast f$ and both sums converge as tempered distribution. We can define the Besov and Triebel-Lizorkin spaces using such $\lambda$.

Recall from Convention~\ref{Conv::Intro::AdmissibleIndex} that the pair $(p,q)$ is said to be \textit{admissible}, if we are discussing the function class $\Bs_{pq}^s$ for $0<p,q\le\infty$, and the function class $\Fs_{pq}^s$ for $0<p<\infty$, $0<q\le\infty$ or for $p=q=\infty$. 

\begin{defn}\label{Defn::Prem::BsFsDef}
Let $(\lambda_j)_{j=0}^\infty\in\Cf$, $s\in\R$ and let $(p,q)$ be admissible. We denote
\begin{align*}
    \|f\|_{\Bs_{pq}^s(\lambda)}:=\|(2^{js}\lambda_j\ast f)_{j=0}^\infty\|_{\ell^q(L^p)}&=\left\|\big(2^{js}\|\lambda_j\ast f\|_{L^p(\R^n)}\big)_{j=0}^\infty\right\|_{\ell^q(\N)},\\
    \|f\|_{\Fs_{pq}^s(\lambda)}:=\|(2^{js}\lambda_j\ast f)_{j=0}^\infty\|_{L^p(\ell^q)}&=\left\|\big\|(2^{js}\lambda_j\ast f)_{j=0}^\infty\big\|_{\ell^q(\N)}\right\|_{L^p(\R^n)}.
\end{align*}

We define the \emph{Besov space $\Bs_{pq}^s(\R^n)$} to be the set of all $f \in \Ss'(\R^n)$ such that $\|f\|_{\Bs_{pq}^s(\lambda)}<\infty$; and the \emph{Triebel-Lizorkin space $\Fs_{pq}^s(\R^n)$} to be the set of all $f\in\Ss'(\R^n)$ such that $\|f\|_{\Fs_{pq}^s(\lambda)}<\infty$.

We use the norm $\|\cdot\|_{\Bs_{pq}^s(\R^n)}=\|\cdot\|_{\Bs_{pq}^s(\lambda)}$ and $\|\cdot\|_{\Fs_{pq}^s(\R^n)}=\|\cdot\|_{\Fs_{pq}^s(\lambda)}$ for an (implicitly chosen) $\lambda\in\Cf$.

Let $\Omega\subseteq\R^n$ be an arbitrary open subset. We define $\Bs_{pq}^s(\Omega):=\{\tilde f|_\Omega:\tilde f\in\Bs_{pq}^s(\R^n)\}$ and $\Fs_{pq}^s(\Omega):=\{\tilde f|_\Omega:\tilde f\in\Fs_{pq}^s(\R^n)\}$ as the subspace of distributions in $\Omega$, with norms
\begin{equation*}
    \|f\|_{\Bs_{pq}^s(\Omega)}:=\inf\{\|\tilde f\|_{\Bs_{pq}^s(\R^n)}:\tilde f|_\Omega=f\},\quad \|f\|_{\Fs_{pq}^s(\Omega)}:=\inf\{\|\tilde f\|_{\Fs_{pq}^s(\R^n)}:\tilde f|_\Omega=f\}.
\end{equation*}
\end{defn}

\begin{rem}\label{Rmk::Prem::BsFsRmk}
\begin{enumerate}[(i)]
    \item  Different choices of $(\lambda_j)_j$ produce equivalent Besov and Triebel-Lizorkin norms, see \cite[Proposition~2.3.2]{TriebelTheoryOfFunctionSpacesI}. So we can use  $\|\cdot\|_{\Bs_{pq}^s(\R^n)}$ and $\|\cdot\|_{\Fs_{pq}^s(\R^n)}$ by dropping $\lambda$ in the notations.
    In application we fix one such choice of $(\lambda_j)\in\Cf$.
    \item When $1\le p,q\le\infty$, $\|\cdot\|_{\Bs_{pq}^s}$ and $\|\cdot\|_{\Fs_{pq}^s}$ are norms. When $p<1$ or $q<1$ they are only quasi-norms.
    \item\label{Item::Prem::BsFsRmk::Special} For $1<p<\infty$, $\Fs_{p2}^s(\R^n)=H^{s,p}(\R^n)$ is the fractional Sobolev space (also known as the Bessel potential space). For  $s>0$,  $\Bs_{\infty\infty}^s(\R^n)=\Co^s(\R^n)$ is the H\"older-Zygmund space. 
    \item It is clear from definition that if $(\lambda_j)_{j=0}^\infty\in\Cf$ then $(\lambda_j)_{j=1}^\infty\in\Gf$ (Definition~\ref{Defn::Intro::GenRes}). However, by a more sophisticated version of the uncertainty principle (for example, the Nazarov's uncertainty principle \cite{JamingNazarovUP}), if a function has compact support on the Fourier side, then it cannot vanish outside any open subset on the physical side. In particular, for a cone $K\subsetneq\R^n$, there is no  $(\lambda_j)_{j=0}^\infty\in\Cf$ such that $(\lambda_j)_{j=1}^\infty\in\Gf(K)$. 
    \item We do not investigate $\Fs_{\infty q}^s(\R^n)$ or $\Fs_{\infty q}^s(\Omega)$ for $0<q<\infty$ in this paper, since $\|\cdot\|_{\Fs_{\infty q}^s(\lambda)}$ are not equivalent norms for different $\lambda\in\Cf$. Using a correctly modified version of the norm (see \cite{TriebelTheoryOfFunctionSpacesIV,YSYMorrey} for examples), the analogs of Theorems \ref{Thm::Intro::Folklore!} and \ref{Thm::MainThm} \ref{Item::MainThm::Bdd} are still true. See \cite{YaoMorrey}. 
    \item\label{Item::Prem::BsFsRmk::InftyInfty} In literatures like \cite{TriebelTheoryOfFunctionSpacesI}, the results for Triebel-Lizorkin space $\Fs_{pq}^s$ are usually only given for $0<p<\infty$, $0<q\le\infty$, and the case $p=q=\infty$ is omitted.  
   Note that $\|\cdot\|_{\Bs_{\infty\infty}^s(\lambda)}=\|\cdot\|_{\Fs_{\infty\infty}^s(\lambda)}$ for every $\lambda=(\lambda_j)_{j=0}^\infty$, so if we have a result for Besov space $\Bs_{\infty\infty}^s$, then the same holds for $\Fs_{\infty\infty}^s$ since $\Fs_{\infty\infty}^s=\Bs_{\infty\infty}^s$ is identically the same space (also see \cite[Remark~2.3.4/3]{TriebelTheoryOfFunctionSpacesI}). This is the reason we include $(p,q)=(\infty,\infty)$ as admissible indices (in 
Convention~\ref{Conv::Intro::AdmissibleIndex}) in the discussion of Triebel-Lizorkin spaces. 
\end{enumerate}
\end{rem}

In fact the norms $\|\cdot\|_{\Bs_{pq}^s(\lambda)}$ ($\|\cdot\|_{\Fs_{pq}^s(\lambda)}$) are still equivalent if we merely assume $\lambda_0,\lambda'_0\in\Ss(\R^n)$ that satisfy $\widehat\lambda_0(\xi)=1+O(|\xi|^\infty)$ and $\widehat\lambda'_0(\xi)=1+O(|\xi|^\infty)$ as $\xi\to0$; see \cite[Proposition~1.2]{ExtensionLipschitz}. The proof relies on Heideman type estimate \cite{Heideman}, a more general version of which is the following. 
\begin{prop}\label{Prop::Prem::Heideman}

Let $\eta=(\eta_j)_{j=1}^\infty\in\Gf$. Then for any $M\in\Z_+$ there is a $C=C_{\eta,M}>0$ such that
\begin{equation}\label{Eqn::Prem::Heideman::Main}
    \int_{|x|>2^{-k}} |\eta_j \ast g(x)|(1+2^j|x|)^Mdx\le C\|g\|_{2M+n}\cdot 2^{-M\max(j,j-k)},\quad\forall j\in\Z_+,k\in\Z,\quad g\in \Ss(\R^n).
\end{equation}

Here for $l\in\N$, we use 
\begin{equation}\label{Eqn::Prem::Heideman::AxuNorm}
    \|g\|_l:=\sum_{|\alpha|\le l}\sup_{x\in\R^n}(1+|x|)^l|\partial^\alpha g(x)|.
\end{equation}
In particular, by letting $k\to+\infty$ in \eqref{Eqn::Prem::Heideman::Main}, we get    
\begin{equation}\label{Eqn::Prem::Heideman::Rmk}
    \int_{\R^n} |\eta_j \ast g(x)|(1+2^j|x|)^Mdx\le C\|g\|_{2M+n}\cdot 2^{-Mj},\quad\forall j\in\Z_+,\quad g\in\Ss(\R^n),
\end{equation}
with the same constant $C=C_{\eta,M}>0$.
\end{prop}

\begin{proof}
For simplicity we denote $\eta_0(x):=2^{-n}\eta_1(2^{-1}x)$. So $\eta_j(x)=2^{nj}\eta_0(2^jx)$ for all $j\ge1$.

By assumption $\eta_0\in\dot\Ss(\R^n)$ (see Definitions \ref{Defn::Intro::GenRes} and \ref{Defn::Intro::S0Space}), we have $\int x^\alpha\eta_0(x)dx=0$ for all $\alpha$. Therefore
\begin{align*}
    \eta_j \ast g(x) &=\int\eta_0( y)g(x-2^{-j}y)dy= \int\eta_0(y)\Big(g(x-2^{-j}y)-\sum_{|\alpha|<N}(-2^{-j}y)^\alpha\frac{\partial^\alpha g(x)}{\alpha!}\Big)dy.
\end{align*}
Here $N$ is some large number to be chosen. By Taylor's theorem,
\begin{equation*}
    \Big|g(x-2^{-j}y)-\sum_{|\alpha|<N}(-2^{-j}y)^\alpha\frac{\partial^\alpha g(x)}{\alpha!}\Big|\le\frac{1}{N!} | 2^{-j} y |^{N} \sup_{B(x, 2^{-j} |y| )}  |\nabla^{N} g|,\quad x,y\in\R^n.
\end{equation*}

Using the notion $\|g\|_N=\sum_{|\alpha|\le N}\sup_x(1+|x|)^N|\partial^\alpha g(x)|$, we have  
\begin{equation}
    \sup\limits_{B(x,2^{-j}|y|)}| \nabla^{N} g|  \le\|g\|_{N}\cdot\begin{cases}2^N(1+ |x|)^{- N}, & |x| \geq 2^{1-j} |y|,\\ 1,&|x|<2^{1-j} |y|,
  \end{cases}\quad x,y\in\R^n.
\end{equation}
Therefore
\begin{align*}  
    &\int_{|x|>2^{-k} } (1+2^j|x|)^M |\eta_j  \ast g(x)|dx
    \\
    \le &\int_{|x|  \geq 2^{-k}}(1+2^j|x|)^M dx \bigg( \int_{|y|<2^{j-1} |x| } + \int_{|y| \geq 2^{j-1}|x|}  \bigg)\frac{| 2^{-j} y |^{N} }{N!}\Big( \sup\limits_{B(x, 2^{-j} |y| )} |\nabla^{N}g|\Big)|\eta_0(y)|dy
    \\ 
    \le&\|g\|_{N}\int_{|x|  \geq 2^{-k}}(1+2^j|x|)^M dx \bigg( \int_{|y|< 2^{j-1}  |x| } | 2^{-j} y |^{N} (1+ |x|)^{-{N}} |\eta_0(y) |dy + \int_{|y| \geq 2^{j-1}  |x|} | 2^{-j} y |^{N} | \eta_0 (y) |dy \bigg).
\end{align*}   

Using spherical coordinates $\rho=|x|$, $r=|y|$, and the fact that $\eta_0$ is rapidly decay,  
\begin{align*} 
    &\int_{|x|  \geq 2^{-k}}(1+2^j|x|)^M dx \bigg( \int_{|y|< 2^{j-1}  |x| } | 2^{-j} y |^{N} (1+ |x|)^{-{N}} |\eta_0(y) |dy + \int_{|y| \geq 2^{j-1}  |x|} | 2^{-j} y |^{N} | \eta_0 (y) |dy \bigg)
    \\
    \lesssim&_{\eta_0}2^{-j{N}} \int_{2^{-k}} ^{\infty}(1+2^j\rho)^M \rho^{n-1}d \rho  \bigg( (1+ \rho)^{-{N}}   \int_0^{\infty}  r^{N}  (1+r)^{-2{N} -n}r^{n-1}dr + \int_{2^{j-1} \rho}^{\infty} r^{N} (1+r)^{-2{N} -n} r^{n-1}dr \bigg)
    \\
    \lesssim&_N2^{-j {N}}   \int_{2^{-k}} ^{\infty}  \rho^{n-1}(1+2^j\rho)^M  \left(  (1 + \rho)^{-N} + \int_{2^{j-1} \rho}^{ \infty}  (1+ r)^{-{N}-1}dr \right)d \rho
    \\
    \lesssim&_N2^{-j{N}}  \int_{2^{-k}}^{\infty} \rho^{n-1}(1+2^j\rho)^M  \left( 2^{jM}(1+2^j\rho)^{-M}( 1+ \rho )^{M-N} + (1+ 2^{j-1} \rho)^{-{N}} \right)d  \rho
    \\
    \lesssim&_N2^{-j(N-M)}  \int_{2^{-k}}^{\infty} \rho^{n-1} ( 1+ \rho )^{M-N}d  \rho,\hspace{2.85in}(\text{assuming }N>M)
    \\
    \lesssim&_{M,N} 2^{-j(N-M)}\min(1,2^{-(M+n-N)k}),\hspace{2.75in}(\text{assuming }N>M+n).
\end{align*}   

Taking ${N}=2M + n$, the above is bounded by $2^{-Mj}  \min(1,2^{Mk})\le2^{-M\max(j,j-k)}$, which concludes the proof.
\end{proof}

\begin{cor}\label{Cor::Prem::HeidemanCor}
Let $\eta_0,\theta_0\in\dot\Ss(\R^n)$, define $\eta_j(x):=2^{jn}\eta_0(2^jx)$ and $\theta_j(x):=2^{jn}\theta_0(2^jx)$ for $j\in\Z$.
Then for any $M,N\ge0$ there is $C=C(\eta,\theta,M,N)>0$ such that
\begin{gather}\label{Eqn::Prem::Heideman1}
     \int_{\R^n}|\eta_j\ast\theta_k(x)|(1+2^{\max(j,k)}|x|)^{N}dx\le C 2^{-M|j-k|},\quad\forall j,k\in\Z.
\end{gather}
\end{cor}

\begin{proof}
This is the special case of \cite[Lemma~2.1]{Bui}. Replacing $M$ or $N$ by $\max(M,N)$ we can assume $M=N$. Using \eqref{Eqn::Prem::Heideman::Rmk} with scaling we have 
$$\|(1+2^j|x|)^M(\eta_j\ast\theta_k)\|_{L^1}=\|(1+2^{j-k}|x|)^M(\eta_{j-k}\ast\theta_0)\|_{L^1}\overset{\eqref{Eqn::Prem::Heideman::Rmk}}\lesssim 2^{-M(j-k)},\quad\text{if }j\ge k.$$
By symmetry we get the corresponding estimate for $j\le k$.
\end{proof}

The proof of the boundedness of $T^{\eta,\theta,r}_\omega$ in Triebel-Linzorkin spaces follows from the method in \cite[Section 2]{ExtensionLipschitz}, which uses the Peetre's maximal functions.
\begin{defn}
Let $N>0$ and let $\eta=(\eta_j)_j\subset\Ss(\R^n)$ be a sequence of Schwartz functions (for either $j\in\Z_+$, $j\in\N$ or $j\in\Z$). The associated \textit{Peetre maximal operators} $(\Pc^{\eta,N}_j)_j$ to $\eta$ are given by
\begin{equation}
    \Pc^{\eta,N}_jf(x):=\sup\limits_{y\in\R^n}\frac{|\eta_j\ast f(x-y)|}{(1+2^j|y|)^{N}},\quad f\in\Ss'(\R^n),\quad x\in\R^n.
\end{equation} 
\end{defn}

We use the following Peetre's maximal estimate. See \cite{Peetre} (also see \cite{PeetreCorrection} for a correction).
\begin{prop}[Peetre's maximal estimate]\label{Prop::Prem::PeetreMax}
   Let $(p,q)$ be admissible (see Convention~\ref{Conv::Intro::AdmissibleIndex}). Let $s\in\R$ and $N>n/\min(p,q)$. Let $\eta=(\eta_j)_{j=1}^\infty\in\Gf$. Then there is a $C=C_{p,q,s,N,\eta}>0$ such that for every $f\in\Ss'(\R^n)$,
\begin{align}\label{Eqn::Prem::PeetreMaxBs}
\|(2^{js}\Pc^{\eta,N}_jf)_{j=1}^\infty\|_{\ell^q(L^p)}&\le C\|f\|_{\Bs_{pq}^s(\R^n)}.
\\
\label{Eqn::Prem::PeetreMax}
    \|(2^{js}\Pc^{\eta,N}_jf)_{j=1}^\infty\|_{L^p(\ell^q)}&\le C\|f\|_{\Fs_{pq}^s(\R^n)}.
\end{align}

Let $\eta_0\in\Ss(\R^n)$. Then there is a $C=C_{p,q,s,N,\eta_0}>0$ such that 
\begin{align}\label{Eqn::Prem::PeetreMaxSing}
\Big\|\sup_{y\in\R^n}\frac{|\eta_0\ast f(y)|}{(1+|\cdot-y|)^N}\Big\|_{L^p(\R^n)}\le C\|f\|_{\As_{pq}^s(\R^n)}.
\end{align}
\end{prop}

In Peetre's original proof, only \eqref{Eqn::Prem::PeetreMax} is  explicitly proved and his proof was done in the setting of homogeneous Triebel-Lizorkin spaces. In our case we use non-homogeneous Besov and Triebel-Lizorkin spaces, and the corresponding results were proved for example in \cite[Theorem~5.1]{Bui}. Note that \cite[(5.4)]{Bui} is a stronger version to both \eqref{Eqn::Prem::PeetreMaxBs} and \eqref{Eqn::Prem::PeetreMaxSing}.

In the proof of Theorem~\ref{Thm::MainThm} \ref{Item::MainThm::Bdd}, especially when $p<1$ or $q<1$, we need the following.
\begin{lemma}\label{Lem::Prem::ASumLemma}
   Let $0<p,q\le\infty$ and $\delta>0$.
   There are $C_{q,\delta}$ and $C_{p,q,\delta}>0$ such that
   \begin{gather}\label{Eqn::Prem::ASumLemma::Fs}
       \Big\|\Big(\sum_{j\in\Z}2^{-\delta|j-k|}g_j\Big)_{k\in\Z}\Big\|_{L^p(\ell^q)}\le C_{q,\delta}\|(g_j)_{j\in\Z}\|_{L^p(\ell^q)};
       \\\label{Eqn::Prem::ASumLemma::Bs}
       \Big\|\Big(\sum_{j\in\Z}2^{-\delta|j-k|}g_j\Big)_{k\in\Z}\Big\|_{\ell^q(L^p)}\le C_{p,q,\delta}\|(g_j)_{j\in\Z}\|_{\ell^q(L^p)}.
   \end{gather}
    for every sequence $g=(g_j)_{j\in\Z}:\Omega\subset\R^n\to\ell^q(\Z)$ of non-negative measurable functions.
\end{lemma}
See \cite[Lemma~2]{RychkovThmofBui}, in fact $C_{q,\delta}
       =\textstyle\big(\sum_{j\in\Z}2^{-\delta|j|\min(1,q)}\big)^{\min(1,q)^{-1}}$ and $C_{p,q,\delta}
       =\textstyle\big(\sum_{j\in\Z}2^{-\delta|j|\min(1,q)}\big)^{\frac{\min(1,p)}{\min(1,q)}}$.  Note that when $1\le p,q\le \infty$ (\ref{Eqn::Prem::ASumLemma::Fs}) and (\ref{Eqn::Prem::ASumLemma::Bs}) are just Young's inequalities.

\subsection{Proof of Theorem~\ref{Thm::MainThm} \ref{Item::MainThm::Bdd}}\label{Section::ProofThmBdd}

In Theorem~\ref{Thm::MainThm}, we assume that $\omega$ is a special Lipschitz domain and $\eta_j,\theta_j$ are supported in $-\Kb$. In the proof of Theorem~\ref{Thm::MainThm} \ref{Item::MainThm::Bdd}, the only place where these assumptions are needed is to show that $T^{\eta,\theta,r}_\omega$ is well-defined. 

More precisely, taking an extension $\tilde f\in\Ss'(\R^n)$ for $f\in\Ss'(\omega)$, we have $T^{\eta,\theta,r}_\omega f=\sum_{j=1}^\infty 2^{jr}\eta_j\ast(\1_\omega\cdot(\theta_j\ast \tilde f))$ and the value $ 2^{jr}\eta_j\ast(\1_\omega\cdot(\theta_j\ast \tilde f))$ does not depend on the choice of the extension.

Therefore by passing to their extensions we can consider a more general version of Theorem~\ref{Thm::MainThm} \ref{Item::MainThm::Bdd}:

\begin{prop}\label{Prop::Prem::TildeTBdd}
Let $r\in\R$, let $\eta=(\eta_j)_{j=1}^\infty,\theta=(\theta_j)_{j=1}^\infty\in\Gf$ and let $\Omega\subseteq\R^n$ be an arbitrary open set. Define the operator $\tilde T^{\eta,\theta,r}_\Omega$ as 
\begin{equation}\label{Eqn::Prem::TildeTBdd::DefTildeT}
    \tilde T^{\eta,\theta,r}_\Omega f:=\sum_{j=1}^\infty 2^{jr}\eta_j\ast(\1_\Omega\cdot(\theta_j\ast f)).
\end{equation}
Then
\begin{enumerate}[(i)]
    \item\label{Item::Prem::TildeTBdd::WellDef}$\tilde T^{\eta,\theta,r}_\Omega:\Ss'(\R^n)\to\Ss'(\R^n)$ is well-defined in the sense that for every $f\in\Ss'(\R^n)$, the sum \eqref{Eqn::Prem::TildeTBdd::DefTildeT} converges in $\Ss'(\R^n)$.
    \item \label{Item::Prem::TildeTBdd}
    For any $s\in\R$ and admissible $(p,q)$, there is a $C=C(\eta,\theta,p,q,r,s)>0$ that does not depend on $\Omega$, such that
    \begin{equation}\label{Eqn::Prem::TildeTBdd}
         \|\tilde T^{\eta,\theta,r}_\Omega  f\|_{\As_{pq}^{s-r}(\R^n)}\le C\| f\|_{\Bs_{pq}^{s}(\R^n)}.
    \end{equation}
\end{enumerate}
\end{prop}

\begin{proof}The proof of \ref{Item::Prem::TildeTBdd::WellDef} is standard. We give a proof for completeness.

It suffices to show  that the adjoint operator $(\tilde T^{\eta,\theta,r}_\Omega)^*:\Ss(\R^n)\to\Ss(\R^n)$ is continuous, since $\tilde T^{\eta,\theta,r}_\Omega$ is the adjoint of $(\tilde T^{\eta,\theta,r}_\Omega)^*$ as well.

By direct computation, we see that
\begin{equation}\label{Eqn::Prem::PfWellDef::Duality}
    (\tilde T^{\eta,\theta,r}_\Omega)^* f=\sum_{j=1}^\infty 2^{jr}\breve\theta_j\ast(\1_\Omega \cdot (\breve\eta_j\ast f)),\quad f\in\Ss(\R^n),\quad\text{where }\breve\eta_j(x):=\eta_j(-x),\ \breve\theta_j(x):=\theta_j(-x).
\end{equation}

Clearly $\breve\eta=(\breve\eta_j)_{j=1}^\infty $ and $\breve\theta=(\breve\theta_j)_{j=1}^\infty$ are also in $\Gf$, and we have $(\tilde T^{\eta,\theta,r}_\Omega)^*=\tilde T^{\breve\theta,\breve\eta,r}_\Omega$. Since $\eta,\theta\in\Gf$ are arbitrary, if we can prove the boundedness $\tilde T^{\eta,\theta,r}_\Omega:\Ss(\R^n)\to\Ss(\R^n)$, then by replacing $(\eta,\theta)$ with $(\breve\theta,\breve\eta)$ we have $(\tilde T^{\eta,\theta,r}_\Omega)^*=\tilde T^{\breve\theta,\breve\eta,r}_\Omega:\Ss(\R^n)\to\Ss(\R^n)$, which concludes $\tilde T^{\eta,\theta,r}_\Omega:\Ss'(\R^n)\to\Ss'(\R^n)$. 

\medskip
We now prove $\tilde T^{\eta,\theta,r}_\Omega:\Ss(\R^n)\to\Ss(\R^n)$. Using Schwartz seminorms, it suffices to show that for every $\alpha,\beta\in\N^n$ there are $M=M_{\alpha,\beta,\eta,\theta,r}\in\Z_+$ and $C=C_{\eta,\theta,r,M}>0$ such that (recall the notion $\|f\|_M$ in \eqref{Eqn::Prem::Heideman::AxuNorm}),
\begin{equation}\label{Eqn::Prem::PfWellDef::SchwartzBdd}
    \sup\limits_{x\in\R^n}|x^\alpha\partial^\beta\tilde T^{\eta,\theta,r}_\Omega f(x)|\le C\|f\|_M=C\sum_{|\gamma|\le M}\sup\limits_{x\in\R^n}|(1+|x|)^M\partial^\gamma f(x)|,\quad\forall  f\in\Ss(\R^n).
\end{equation}

Indeed, for every $x\in\R^n$,
\begin{align*}
    |x^\alpha\partial^\beta\tilde T^{\eta,\theta,r}_\Omega f(x)|\le&\sum_{j=1}^\infty 2^{jr}|x|^{|\alpha|}\int_\Omega|\partial^\beta\eta_j(x-y)||\theta_j\ast f(y)|dy
    \\
    \le&\sum_{j=1}^\infty 2^{jr}\int_{\R^n}\frac{|x|^{|\alpha|}}{(1+2^j|y|)^{|\alpha|}}|\partial^\beta\eta_j(x-y)||\theta_j\ast f(y)|(1+2^j|y|)^{|\alpha|}dy
    \\
    \le&\sum_{j=1}^\infty 2^{jr}\int_{\R^n}(1+2^j|x-y|)^{|\alpha|}|\partial^\beta\eta_j(x-y)||\theta_j\ast f(y)|(1+2^j|y|)^{|\alpha|}dy
    \\
    \le&\sum_{j=1}^\infty 2^{jr}\|(1+2^j|y|)^{|\alpha|}\partial^\beta\eta_j(y)\|_{L^\infty_y(\R^n)}\|(1+2^j|y|)^{|\alpha|}(\theta_j\ast f)(y)\|_{L^1_y(\R^n)}
    \\
    \lesssim&_\eta\sum_{j=1}^\infty 2^{j(r+n+|\beta|)}\|(1+2^j|y|)^{|\alpha|}(\theta_j\ast f)(y)\|_{L^1(\R^n)}.
\end{align*}

Applying Proposition~\ref{Prop::Prem::Heideman} (see \eqref{Eqn::Prem::Heideman::Rmk}) with $M'=\max(r+n+|\beta|+1,|\alpha|)$, we have
\begin{equation*}
    \|(1+2^j|y|)^{|\alpha|}(\theta_j\ast f)(y)\|_{L^1_y}\le\|(1+2^j|y|)^M(\theta_j\ast f)(y)\|_{L^1_y}\lesssim_{M'}\|f\|_{2M'+n}2^{-jM'}\le\|f\|_{2M'+n}2^{-j(r+n+|\beta|)-j}.
\end{equation*}
Therefore
\begin{equation*}
    \textstyle\sup_{x\in\R^n}|x^\alpha\partial^\beta\tilde T^{\eta,\theta,r}_\Omega f(x)|\lesssim\sum_{j=1}^\infty2^{-j}\|f\|_{2M'+n}\le \|f\|_{2M'+n}.
\end{equation*}

Taking $M=2M'+n$, we get \eqref{Eqn::Prem::PfWellDef::SchwartzBdd}, finishing the proof of \ref{Item::Prem::TildeTBdd::WellDef}. 

\medskip
To prove \ref{Item::Prem::TildeTBdd}, we follow the arguments in \cite[Section 2]{ExtensionLipschitz}.
We fix a classical dyadic resolution $\lambda=(\lambda_k)_{k=0}^\infty\in\Cf$ to define the Besov and Triebel-Lizorkin (quasi-)norms. Then for every admissible $(p,q)$, \eqref{Eqn::Prem::TildeTBdd} can be written as 
\begin{align}\label{Eqn::Prem::PfThmBdd::BsCase}
    \big\|\big(2^{k(s-r)}\lambda_k\ast \tilde T^{\eta,\theta,r}_\Omega f\big)_{k=0}^\infty\big\|_{\ell^q(L^p)}\lesssim_{\eta,\theta,\lambda,p,q,r,s}\|f\|_{\Bs_{pq}^s(\R^n)};
    \\\label{Eqn::Prem::PfThmBdd::TLCase}
    \big\|\big(2^{k(s-r)}\lambda_k\ast \tilde T^{\eta,\theta,r}_\Omega f\big)_{k=0}^\infty\big\|_{L^p(\ell^q)}\lesssim_{\eta,\theta,\lambda,p,q,r,s}\|f\|_{\Fs_{pq}^s(\R^n)}.
\end{align}

Here the implicit constants in \eqref{Eqn::Prem::PfThmBdd::BsCase} and \eqref{Eqn::Prem::PfThmBdd::TLCase} do not depend on $\Omega$. 

We now prove \eqref{Eqn::Prem::PfThmBdd::BsCase} and \eqref{Eqn::Prem::PfThmBdd::TLCase}.
Taking $M=|s-r|+1$ to \eqref{Eqn::Prem::Heideman::Rmk} for $k=0$ with $g=\lambda_0$ and to Corollary \ref{Cor::Prem::HeidemanCor} for $k\ge1$ with $\eta_k=\lambda_k$, we have for $k\in\N$, $j\in\Z_+$ and $x\in\R^n$,
    \begin{equation}\label{Eqn::Prem::PfThmBdd::Tmp0}
            \begin{aligned}
        &2^{k(s-r)} 2^{jr}|\lambda_k\ast \eta_j\ast (\1_\Omega \cdot (\theta_j\ast  f))(x)|\le2^{k(s-r)} 2^{jr}\int_\Omega|\lambda_k\ast \eta_j(x-y)||\theta_j\ast f(y)|dy
        \\
        &\quad\le2^{k(s-r)} 2^{jr}\int_{\R^n}|\lambda_k\ast \eta_j(x-y)|(1+2^j|x-y|)^{N}\sup\limits_{t\in\R^n}\frac{|\theta_j\ast f(t)|}{(1+2^j|x-t|)^N}dy
        \\
        &\quad\lesssim_{\lambda,\eta,M,N} 2^{k(s-r)} 2^{jr}2^{-M|j-k|}(\Pc^{\theta,N}_j f)(x)\le 2^{-|j-k|}2^{js}(\Pc^{\theta,N}_j f)(x),
    \end{aligned}
    \end{equation}
    where in the last inequality we use $2^{-M|j-k|}=2^{-|s-r||j-k|}2^{-|j-k|}\le2^{(j-k)(s-r)}2^{-|j-k|}$.

    Taking sum of \eqref{Eqn::Prem::PfThmBdd::Tmp0} over $j\in\Z_+$ we have for $k\in\N$ and $x\in\R^n$,
    \begin{equation}\label{Eqn::Prem::PfThmBdd::Tmp1}
        2^{k(s-r)}|\lambda_k\ast\tilde T^{\eta,\theta,r}_\Omega  f(x)|\le\sum_{j=1}^\infty2^{k(s-r)} 2^{jr}|\lambda_k\ast \eta_j\ast (\1_\Omega\cdot(\theta_j\ast  f))(x)|\lesssim_{\eta,\lambda,r,s,N}\sum_{j=1}^\infty2^{-|j-k|}2^{js}\Pc_j^{\theta,N} f(x).
    \end{equation}
    
    Taking $N>n/\min(p,q)$, applying Lemma~\ref{Lem::Prem::ASumLemma} and Proposition~\ref{Prop::Prem::PeetreMax},
    \begin{align}
        \label{Eqn::Prem::PfThmBdd::Tmp2::Bs}
        \bigg\|\Big(\sum_{j=1}^\infty 2^{-|j-k|}2^{js}\Pc_j^{\theta,N} f\Big)_{k=0}^\infty\bigg\|_{\ell^q(L^p)}&\lesssim_{p,q}\Big\|(2^{ks}\Pc_k^{\theta,N} f)_{k=1}^\infty\Big\|_{\ell^q(L^p)}\lesssim_{\theta,p,q,s,N}\| f\|_{\Bs_{pq}^s(\R^n)}.
        \\
    \label{Eqn::Prem::PfThmBdd::Tmp2::TL}
        \bigg\|\Big(\sum_{j=1}^\infty 2^{-|j-k|}2^{js}\Pc_j^{\theta,N} f\Big)_{k=0}^\infty\bigg\|_{L^p(\ell^q)}&\lesssim_q\Big\|(2^{ks}\Pc_k^{\theta,N} f)_{k=1}^\infty\Big\|_{L^p(\ell^q)}\lesssim_{\theta,p,q,s,N}\| f\|_{\Fs_{pq}^s(\R^n)}.
    \end{align}
    
    Combining \eqref{Eqn::Prem::PfThmBdd::Tmp1} and \eqref{Eqn::Prem::PfThmBdd::Tmp2::Bs} we get \eqref{Eqn::Prem::PfThmBdd::BsCase}.
    Combining \eqref{Eqn::Prem::PfThmBdd::Tmp1} and \eqref{Eqn::Prem::PfThmBdd::Tmp2::TL} we get \eqref{Eqn::Prem::PfThmBdd::TLCase}. We finish the proof of \ref{Item::Prem::TildeTBdd}.
\end{proof}
\begin{rem}\label{Rmk::Prem::TildeBdd::OmegaNotDepExplain}
We recap why the constants in \eqref{Eqn::Prem::PfThmBdd::BsCase} and \eqref{Eqn::Prem::PfThmBdd::TLCase} do not depend on $\Omega$. In \eqref{Eqn::Prem::PfThmBdd::Tmp0}, the constants depend on $\lambda,\eta,r,s$ but not on $\Omega$; and after \eqref{Eqn::Prem::PfThmBdd::Tmp0}, the problem is reduced to estimating the $\ell^q(L^p)$-norm and $L^p(\ell^q)$-norm of the Peetre maximal functions, where $\Omega$ is not involved.
\end{rem}

\begin{rem}When $p\ge1$, the proof of \eqref{Eqn::Prem::PfThmBdd::BsCase} can be done without using the maximal functions.

Indeed, let $\mu_l=\sum_{k=l-1}^{l+1}\lambda_k$ as in Remark~\ref{Rmk::Prem::DyaResRmk}, we have $f=\sum_{l=0}^\infty \mu_l\ast\lambda_l\ast f$. By Corollary \ref{Cor::Prem::HeidemanCor} we have $\|\lambda_k\ast\eta_j\|_{L^1}\lesssim_M2^{-M|k-j|}$ and $\|\theta_j\ast\mu_l\|_{L^1}\lesssim_M2^{-M|j-l|}$ for all $M$. Thus by Young's inequality with $M>|s|+|r|+1$ we have
\begin{gather*}
2^{k(s-r)+jr}\|\lambda_k\ast\eta_j\ast(\1_{\Omega}\cdot(\theta_j\ast  f))\|_{L^p(\R^n)}\le2^{ks-(k-j)r}\|\lambda_k\ast\eta_j\|_{L^1}\|\theta_j\ast  f\|_{L^p(\Omega)}\lesssim2^{-|k-j|+js}\|\theta_j\ast  f\|_{L^p(\R^n)}.
\\
    2^{js}\|\theta_j\ast f\|_{L^p(\R^n)}\le\sum_{l=0}^\infty2^{js}\|\theta_j\ast\mu_l\ast\lambda_l\ast  f\|_{L^p}\le\sum_{l=0}^\infty2^{js}\|\theta_j\ast\mu_l\|_{L^1}\|\lambda_l\ast  f\|_{L^p}\lesssim \sum_{l=0}^\infty2^{-|j-l|}2^{ls}\|\lambda_l\ast  f\|_{L^p(\R^n)}.
\end{gather*}

Therefore,
\begin{align*}
    2^{k(s-r)}\|\lambda_k\ast\tilde T^{\eta,\theta,r}_\Omega f\|_{L^p(\R^n)}
        &\le\sum_{j=1}^\infty2^{ks}2^{-(k-j)r}\|\lambda_k\ast\eta_j\ast(\1_{\Omega}\cdot(\theta_j\ast  f))\|_{L^p(\R^n)}
        \\
        &\lesssim\sum_{j=1}^\infty 2^{-|k-j|}2^{js}\|\theta_j\ast  f\|_{L^p(\R^n)}
        \\
        &\lesssim\sum_{j=1}^\infty\sum_{l=0}^\infty2^{-|k-j|-|j-l|}2^{ls}\|\lambda_l\ast  f\|_{L^p(\R^n)}.
\end{align*}

    By using \eqref{Eqn::Prem::ASumLemma::Bs} in Lemma~\ref{Lem::Prem::ASumLemma} twice we get \eqref{Eqn::Prem::PfThmBdd::BsCase}.\hfill\qedsymbol
\end{rem}
Theorem~\ref{Thm::MainThm} \ref{Item::MainThm::Bdd} is then an immediate consequence to Proposition~\ref{Prop::Prem::TildeTBdd}.
\begin{proof}[Proof of Theorem~\ref{Thm::MainThm} \ref{Item::MainThm::Bdd}]
Let $\eta,\theta\in\Gf(-\Kb)$ satisfy the assumption of Theorem~\ref{Thm::MainThm}.

Suppose $\tilde f_1,\tilde f_2\in\Ss'(\R^n)$ satisfy $\tilde f_1|_\omega=\tilde f_2|_\omega$. By Proposition~\ref{Prop::Prem::TildeTBdd} \ref{Item::Prem::TildeTBdd::WellDef} we know $\tilde T^{\eta,\theta,r}_\Omega\tilde f_1,\tilde T^{\eta,\theta,r}_\Omega\tilde f_2\in\Ss'(\R^n)$ are well-defined. 
Since $\omega=\{x_n>\rho(x')\}$ is special Lipschitz, we have $|\nabla\rho|<1$, so $\omega+\Kb\subseteq\omega$ (see Remark~\ref{Rmk::Intro::RmkSpeDom}). Therefore $(\theta_j\ast\tilde f_1-\theta_j\ast\tilde f_2)|_\omega=0$ for all $j\in\Z_+$ and we have $\tilde T^{\eta,\theta,r}_\omega\tilde f_1=\tilde T^{\eta,\theta,r}_\omega\tilde f_2$.

For $f\in\Ss'(\omega)$, taking an extension $\tilde f\in\Ss'(\R^n)$ of $f$, we have $$T^{\eta,\theta,r}_\omega f=\tilde T^{\eta,\theta,r}_\omega\tilde f\in\Ss'(\R^n),$$ and the value does not depend on the choice of the extension. Therefore $T^{\eta,\theta,r}_\omega:\Ss'(\omega)\to\Ss'(\R^n)$ is well-defined.

By \eqref{Eqn::Prem::TildeTBdd} we have $\|T^{\eta,\theta,r}_\omega f\|_{\As_{pq}^{s-r}(\R^n)}=\|\tilde T^{\eta,\theta,r}_\omega \tilde f\|_{\As_{pq}^{s-r}(\R^n)}\lesssim\|\tilde f\|_{\As_{pq}^{s}(\R^n)}$ for all $s\in\R$ and admissible $(p,q)$. Taking the infimum over all extension $\tilde f$ for $f$ and using Definition~\ref{Defn::Prem::BsFsDef} we get $\|T^{\eta,\theta,r}_\omega f\|_{\As_{pq}^{s-r}(\R^n)}\lesssim \| f\|_{\As_{pq}^{s}(\omega)}$, concluding the proof.
\end{proof}
\begin{rem}
For Theorem~\ref{Thm::MainThm} \ref{Item::MainThm::Bdd}, we can relax the assumption $\|\nabla\rho\|_{L^\infty}<1$ in Definition~\ref{Defn::Intro::SpeLipDom} to $\|\nabla\rho\|_{L^\infty}\le1$. The proof does not need any change, since $\omega+\Kb\subseteq\omega$ is still valid for such $\omega=\{x_n>\rho(x')\}$. 

However, later in the proof of Theorem~\ref{Thm::MainThm} \ref{Item::MainThm::Smoothing}, the assumption $\|\nabla\rho\|_{L^\infty}<1$ becomes critical.
\end{rem}
\subsection{The zeroth term of $E_\omega$}

For the Rychkov's extension operator $E_\omega f=\sum_{j=0}^\infty\psi_j\ast(\1_\omega \cdot (\phi_j\ast f))$ (see \eqref{Eqn::Intro::ExtOp}), recall that we have $E_\omega f=\psi_0\ast(\1_\omega \cdot (\phi_0\ast f))+T^{\psi,\phi,0}_\omega f$ where $\psi=(\psi_j)_{j=1}^\infty,\phi=(\phi_j)_{j=1}^\infty\in\Gf(-\Kb)$ (see \eqref{Eqn::Intro::EfromT}). We now show that the map $f\mapsto \psi_0\ast(\1_\omega \cdot (\phi_0\ast f))$ gains derivatives of arbitrarily large order: 

\begin{prop}\label{Prop::Prem::BddE0}
\begin{enumerate}[(i)]
    \item\label{Item::Prem::BddE0::TildeF0} Let $\Omega\subseteq\R^n$ be an arbitrary open set, and let $\eta_0,\theta_0\in\Ss(\R^n)$. Let $\tilde F^{\eta_0,\theta_0}_\Omega$ be the linear operator defined by 
\begin{equation}\label{Eqn::Prem::BddTildeE0}
     \tilde F^{\eta_0,\theta_0}_\Omega f:=\eta_0\ast(\1_\Omega\cdot(\theta_0\ast f)).
\end{equation}

Then $\tilde F^{\eta_0,\theta_0}_\Omega:\Ss'(\R^n)\to\Ss'(\R^n)$ is well-defined. Furthermore, $\tilde F^{\eta_0,\theta_0}_\Omega:\As_{pq}^{s}(\R^n)\to\As_{pq}^{s+m}(\R^n)$ for $s\in\R$, $m>0$ and all admissible $(p,q)$. 
\item\label{Item::Prem::BddE0::F0} Let $\omega\subset\R^n$ be a special Lipschitz domain, and let $\eta_0,\theta_0\in\Ss(\R^n)$ satisfy $\supp\eta_0,\supp\theta_0\subset-\Kb$. Let $F^{\eta_0,\theta_0}_\omega$ be the linear operator defined by
\begin{equation}\label{Eqn::Prem::BddE0}
     F^{\eta_0,\theta_0}_\omega h:=\eta_0\ast(\1_\omega\cdot (\theta_0\ast h)).
\end{equation}

Then $F^{\eta_0,\theta_0}_\omega:\Ss'(\omega)\to\Ss'(\R^n)$ is well-defined. Furthermore, $ F^{\eta_0,\theta_0}_\omega:\As_{pq}^{s}(\omega)\to\As_{pq}^{s+m}(\R^n)$ for $s\in\R$, $m>0$ and all admissible $(p,q)$.
\end{enumerate}

\end{prop}
\begin{rem}\label{Rmk::Prem::TildeF::OmegaNotDep}
\begin{enumerate}[(a)]
    \item We do not require $\eta_0$ and $\theta_0$ to have moment vanishing in this proposition.
    \item Similar to Remarks  \ref{Rmk::Prem::TildeBdd::OmegaNotDepExplain}, the operator norms of $\tilde F^{\eta_0,\theta_0}_\Omega$ do not depend on $\Omega$.
\end{enumerate} 

\end{rem}

\begin{proof}
\ref{Item::Prem::BddE0::TildeF0}: Clearly $[f\mapsto\theta_0\ast f]:\Ss'(\R^n)\to C^\infty\cap\Ss'(\R^n)$ is continuous, so $[f\mapsto\1_\Omega\cdot(\theta_0\ast f)]:\Ss'(\R^n)\to \Ss'(\R^n)$ is well-defined. By convoluting with $\eta_0$ we get the well-definedness to $\tilde F^{\eta_0,\theta_0}_\Omega:\Ss'(\R^n)\to\Ss'(\R^n)$.

Still let $\lambda=(\lambda_j)_{j=0}^\infty\in\Cf$ be the dyadic resolution that defines the Besov and Triebel-Lizorkin norms. Similar to \eqref{Eqn::Prem::PfThmBdd::Tmp0}, for every $M,N>0$ we have
\begin{align*}
    2^{j(s+m)}|\lambda_j\ast\eta_0\ast(\1_\Omega \cdot (\theta_0\ast f))(x)|&\le 2^{j(s+m)}\int_\Omega|\lambda_j\ast\eta_0(x-y)||\theta_0\ast f(y)|dy
    \\
    &\le 2^{j(s+m)}\int_{\R^n}|\lambda_j\ast\eta_0(x-y)|(1+|x-y|)^N\sup\limits_{t\in\Omega}\frac{|\theta_0\ast f(t)|}{(1+|x-t|)^N}dy
    \\
    &\overset{\eqref{Eqn::Prem::Heideman::Rmk}}{\lesssim}_{\lambda,\eta_0,M}2^{j(s+m)}2^{-Mj}\sup\limits_{t\in\R^n}\frac{|\theta_0\ast f(t)|}{(1+|x-t|)^N}.&
\end{align*}

Taking $M\ge s+m+1$ and $N>n/\min(p,q)$ from above and then taking sum over $j$ we have
\begin{equation}\label{Eqn::Prem::BddE0::Tmp1}
\begin{aligned}
    &\|\tilde F^{\eta_0,\theta_0}_\Omega f\|_{\Bs_{pq}^{s+m}(\R^n)}\approx_\lambda\|(2^{j(s+m)}\lambda_j\ast\tilde F^{\eta_0,\theta_0}_\Omega f)_{j=0}^\infty\|_{\ell^q(L^p)}
    \\
    \lesssim&_{s,m,N,\lambda,\eta_0}\Big\|\Big(2^{-j}\sup\limits_{t\in\R^n}\frac{|\theta_0\ast f(t)|}{(1+|\cdot-t|)^N}\Big)_{j=0}^\infty\Big\|_{\ell^q(L^p)}\lesssim_q\Big\|\sup\limits_{t\in\R^n}\frac{|\theta_0\ast f(t)|}{(1+|\cdot-t|)^N}\Big\|_{L^p(\R^n)}.
\end{aligned}
\end{equation}

Combining \eqref{Eqn::Prem::BddE0::Tmp1} and the first inequality in \eqref{Eqn::Prem::PeetreMaxSing} we get $\|\tilde F^{\eta_0,\theta_0}_\Omega f\|_{\Bs_{pq}^{s+m}(\R^n)}\lesssim\|f\|_{\Bs_{pq}^s(\R^n)}$, which proves the boundedness $\tilde F^{\eta_0,\theta_0}_\Omega:\Bs_{pq}^s(\R^n)\to\Bs_{pq}^{s+m}(\R^n)$.

Using the same argument as in \eqref{Eqn::Prem::BddE0::Tmp1} and by the second inequality in \eqref{Eqn::Prem::PeetreMaxSing} (also choosing $N>n/\min(p,q)$) we have 
\begin{equation*}
    \|\tilde F^{\eta_0,\theta_0}_\Omega f\|_{\Fs_{pq}^{s+m}(\R^n)}\lesssim_{s,m,N,\eta_0}\Big\|\Big(2^{-j}\sup\limits_{t\in\R^n}\frac{|\theta_0\ast f(t)|}{(1+|\cdot-t|)^N}\Big)_{j=0}^\infty\Big\|_{L^p(\ell^q)}\lesssim_q\Big\|\sup\limits_{t\in\R^n}\frac{|\theta_0\ast f(t)|}{(1+|\cdot-t|)^N}\Big\|_{L^p}\lesssim\|f\|_{\Fs_{pq}^s(\R^n)}.
\end{equation*}
Thus we prove boundedness $\tilde F^{\eta_0,\theta_0}_\Omega:\Fs_{pq}^s(\R^n)\to\Fs_{pq}^{s+m}(\R^n)$.

\medskip
\noindent\ref{Item::Prem::BddE0::F0}: Let $h\in\Ss'(\omega)$, since $\omega+\Kb\subseteq\omega$ and $\supp\theta_0\subset-\Kb$, we see that $\theta_0\ast h\in\Ss'(\omega)$ is defined. Taking an extension $\tilde h\in\Ss'(\R^n)$ of $h$, we have $F^{\eta_0,\theta_0}_\omega g=\tilde F^{\eta_0,\theta_0}_\omega\tilde h\in\Ss'(\R^n)$. So $F^{\eta_0,\theta_0}_\omega:\Ss'(\omega)\to\Ss'(\R^n)$ is well-defined.

The Besov and Triebel-Lizorkin boundedness of $F^{\eta_0,\theta_0}_\omega$ follows from the boundedness of $\tilde F^{\eta_0,\theta_0}_\omega$.
\end{proof}

\section{The Anti-derivatives: Proof of Theorem~\ref{Thm::MainThm} \ref{Item::MainThm::AntiDev}}\label{Section::AntiDev}
In this section we prove Theorem~\ref{Thm::MainThm} \ref{Item::MainThm::AntiDev}. 
We will use the homogeneous tempered distributions.
	\begin{note}
		We denote by $\dot\Ss'(\R^n)$  the dual space of $\dot\Ss(\R^n)$.
		
		For $s\in\R$ we use the \textit{Riesz potential} $(-\Delta)^\frac s2:\dot\Ss(\R^n)\to\dot\Ss(\R^n)$ given by $$(-\Delta)^\frac s2 f=(|2\pi\xi|^s\widehat f)^\vee.$$
	\end{note}
	
By \cite[Remark~2.5]{TriebelHomogeneous} we have  $\dot\Ss'(\R^n)=\Ss'(\R^n)/\{\text{polynomials}\}$. By \cite[Proposition~2.4]{TriebelHomogeneous}, for $f\in\dot\Ss'(\R^n)$, its Fourier transform $\widehat f$ makes sense as a (tempered) distribution on $\R^n\backslash\{0\}$.
	
	For a tempered distribution $h\in\Ss'(\R^n)$ and $s\in\R$, we denote by $(-\Delta)^\frac s2h\in\dot\Ss'(\R^n)$ the linear functional $\varphi\in\dot\Ss(\R^n)\mapsto\langle h,(-\Delta)^\frac s2\varphi\rangle_{\Ss',\Ss}$.

We first prove a lemma for homogeneous Littlewood-Paley families. 
	\begin{lemma}\label{Lem::AntiDev::S0Conv}
	Let $\eta_0\in\dot\Ss(\R^n)$ and $r\in\R$. Denote
\[
   \eta_j(x):=2^{jn}\eta_0(2^jx), \quad j  \in \Z. 
\]
Then
		\begin{enumerate}[(i)]
			\item\label{Item::AntiDev::S0Conv::WellDef} The sum $\sigma:=\sum_{j\in\Z}2^{jr}\eta_j$ converges in $\dot\Ss'(\R^n)$ and satisfies $|\xi|^{-r}\widehat\sigma(\xi)\in L^\infty(\R^n\backslash\{0\})$.
			\item\label{Item::AntiDev::S0Conv::AstS0}  If $f \in\dot\Ss(\R^n)$, then $\sigma\ast f = \sum_{j\in\Z}2^{jr} (\eta_j \ast f)$ converges in $\dot\Ss(\R^n)$. 
		\end{enumerate}
	\end{lemma}
	In fact,  one can show that $f \mapsto\sum_{j\in\Z}2^{jr}\eta_j\ast f$ is a continuous map from $\dot\Ss(\R^n)$ to itself.  
	
	Note that by definition $(\eta_j)_{j=0}^\infty\in\Gf$. It is easy to see that $\sum_{j=0}^\infty 2^{jr}\eta_j\ast f$ converges in $\Ss(\R^n)$ when $f\in\Ss(\R^n)$.
	
	It is not hard to see that $\sum_{j\in\Z}2^{jr}\eta_j$ converges in $\Ss'(\R^n)$ when $r>-n$. However, in general we do not know whether it converges in $\Ss'(\R^n)$ for $r\le-n$, as its Fourier transform blows up in the order of $|\xi|^r$ when $\xi\to0$. For our application it is enough to show the convergence in $\dot\Ss'$.
	\begin{proof}
		It suffices to prove the case $r=0$. Suppose this case is done, let $\theta_0:=(-\Delta)^{-\frac r2}\eta_0\in\dot\Ss(\R^n)$ and $\theta_j(x)=2^{jn}\theta_0(2^jx)$ for $j\in\Z$, we have $2^{jr}\eta_j=(-\Delta)^\frac r2\theta_j$. Thus $(\sum_{j\in\Z}2^{jr}\eta_j)^\wedge=|2\pi \xi|^r(\sum_{j\in\Z}\theta_j)^\wedge$ and $\sum_{j\in\Z}2^{jr}\eta_j\ast f=(\sum_{j\in\Z}\theta_j)\ast(-\Delta)^\frac r2f$. The results \ref{Item::AntiDev::S0Conv::WellDef} and \ref{Item::AntiDev::S0Conv::AstS0} for $\sum_j2^{jr}\eta_j$ then follow.

		\medskip
		Now for the case $r=0$, the sum $\sum_{j\in\Z}\widehat\eta_j(\xi)$ converges to an element in $L^\infty(\R^n)$ since
		\begin{equation*}
		\sum_{j\in\Z}|\widehat\eta_j(\xi)|=\sum_{j\in\Z}|\widehat\eta_0(2^{-j}\xi)|\lesssim_{\eta_0}\sum_{j\in\Z}\min(2^{-j}|\xi|,2^j|\xi|^{-1})\le2\sum_{k\in\Z}2^{-|k|}<\infty,\quad\text{uniformly for }\xi\in\R^n.
		\end{equation*}
We see that $\sum_{j\in\Z}\eta_j$ defines an element that belongs to $\Ss'(\R^n)$ whose Fourier transform is $L^\infty$. In particular $\sigma\in\dot\Ss'(\R^n)$ is well-defined with $\widehat\sigma\in L^\infty(\R^n\backslash\{0\})$. This gives \ref{Item::AntiDev::S0Conv::WellDef}.
		
		\medskip
		To prove \ref{Item::AntiDev::S0Conv::AstS0}, we need to show that for any $f\in\dot\Ss(\R^n)$, $M\in\Z$ and $\alpha,\beta\in\N^n$, the following series converges uniformly for $\xi\in\R^n$:
		\begin{equation}\label{Eqn::AntiDev::ProofCase=1::SchwartzConv}
		|\xi|^M \sum_{j\in\Z}|\partial^\alpha\widehat\eta_j(\xi))||\partial^\beta\widehat f(\xi)|.
		\end{equation}
	By taking $M\to+\infty$ we see that $\sum_j\partial^\alpha\widehat\eta_j\cdot\partial^\beta\widehat f$ has rapid decay when  $|\xi| \to+\infty$. By taking $M\to-\infty$, we see that $\sum_j\partial^\alpha\widehat\eta_j\cdot\partial^\beta\widehat f$ has infinite order of vanishing at $\xi=0$.

		By $\widehat\eta_j(\xi)=\widehat\eta_0(2^{-j}\xi)$, we have $\partial^\alpha\widehat\eta_j(\xi)=2^{-j|\alpha|}(\partial^\alpha\widehat\eta_0)(2^{-j}\xi)$ for each $j\in\Z$.
		Since both $\partial^\beta\widehat f$ and $\partial^\alpha\widehat\eta_0$ have rapid decay as $\xi\to\infty$ and is flat at $\xi=0$, we have $\sup_{\xi\in\R^n}|\xi|^{\tilde M}|\partial^\beta\widehat f(\xi)|+\sup_{\xi\in\R^n}|\xi|^{\tilde M}|\partial^\alpha\widehat f(\xi)|<\infty$ for all $\tilde M\in\R$. Therefore, by taking $\tilde M=|\alpha|-r-1$ and $\tilde M=|\alpha|-r+1$ we have
		\begin{align*}
		&\sum_{j\in\Z}^\infty |\xi|^M|\partial^\alpha\widehat\eta_j(\xi)||\partial^\beta\widehat f(\xi)|
		\\
		=&\sum_{j\ge1} 2^{-j} |2^{-j}\xi|^{|\alpha|-1}|\partial^\alpha\widehat\eta_0(2^{-j}\xi)||\xi|^{M-|\alpha|+1}|\partial^\beta\widehat f(\xi)|
		+\sum_{j\le0}2^j|2^{-j}\xi|^{|\alpha|+1}|\partial^\alpha\widehat\eta_0(2^{-j}\xi)||\xi|^{M-|\alpha|-1}|\partial^\beta\widehat f(\xi)|
		\\
		\le&\||\zeta|^{|\alpha|-1}\partial^\alpha\widehat\eta_0(\zeta)\|_{C^0_\zeta}\||\zeta|^{M-|\alpha|+1}\partial^\beta\widehat f(\zeta)\|_{C^0_\zeta}\sum_{j\ge1}2^{-j}+\||\zeta|^{|\alpha|+1}\partial^\alpha\widehat\eta_0(\zeta)\|_{C^0_\zeta}\||\zeta|^{M-|\alpha|-1}\partial^\beta\widehat f(\zeta)\|_{C^0_\zeta}\sum_{j\le0}2^{j}<\infty.
		\end{align*}
		
		The convergence is uniform in $\xi$, thus we get \ref{Item::AntiDev::S0Conv::AstS0} and complete the proof.
	\end{proof}
	
	To construct the anti-derivatives, we need a lemma on the support condition.
	
\begin{lemma}\label{Lem::AntiDev::AFourierLemma}
	Suppose $u\in \dot\Ss(\R)$ is supported in $\R_+$, then $( 2\pi i\tau)^{-1}\widehat u(\tau)$ is the Fourier transform of the $\dot\Ss$-function $v(t):=\int_{\R_+}u(t-s)ds$.  Moreover $\supp v\subset\R_+$.
\end{lemma}
\begin{proof}
	By assumption $\widehat u\in \Ss(\R)$ satisfies $\partial^\alpha\widehat u(0)=0$ for all $\alpha\in\N^n$, so $\tau\mapsto( 2\pi i\tau)^{-1}\widehat u(\tau)$ is a well-defined smooth function that has infinite order vanishing at $\tau=0$. Note that $( 2\pi i\tau)^{-1}\widehat u$ still has rapid decay, therefore $\big((2\pi i\tau)^{-1}\widehat u\big)^\vee\in\dot\Ss(\R)$.
	
	Since $\big(\operatorname{p.v.}( 2\pi i\tau)^{-1}\big)^\vee(t)=\frac12\operatorname{sgn}(t)$, we have 
	\begin{align*}
	\big((2\pi i\tau)^{-1}\widehat u(\tau)\big)^\vee(t)&=\frac12\operatorname{sgn}\ast u(t) =\frac12 \int_{\R_+}u(t-s)ds - \frac12\int_{\R_-}u(t-s)ds
	\\
	&=\int_{\R_+}u(t-s)ds-\frac12\int_\R u(t-s)ds=\int_{\R_+}u(t-s)ds=v(t).
	\end{align*}
	Here we use the fact that $\int_\R u(t-s)ds=\int_\R u=0$ for all $x\in\R$. Therefore $\widehat v(\tau)=( 2\pi i\tau)^{-1}\widehat u(\tau)$, in particular $v\in\dot\Ss(\R)$.

	Since $\supp u\subset\R_+$, we have $\supp u\subseteq[\eps_0,+\infty)$ for some $\eps_0>0$. Therefore $\int_{\R_+}u(t-s)ds=0$ when $t<\eps_0$, in particular $\supp v\subset\R_+$ as well.
\end{proof}
We now begin the proof of Theorem~\ref{Thm::MainThm} \ref{Item::MainThm::AntiDev} with the case $m=1$. 
\begin{prop} \label{Prop::AntiDev::Case=1} 
	Let $K\subset\R^n$ be a convex cone centered at $0$.
	There are $\sigma_1=\sigma^K_1,\dots,\sigma_n=\sigma^K_n\in \dot\Ss'(\R^n)$ that satisfy the following:
	\begin{enumerate}[(i)]
		\item\label{Item::AntiDev::Case=1::Schwartz} For $f\in\dot\Ss(\R^n)$, we have $\sigma_l\ast f\in\dot\Ss(\R^n)$, $l=1,\dots,n$. 
		
		\item\label{Item::AntiDev::Case=1::Support}
		If $f\in\dot\Ss(\R^n)$ is supported in $K$, then $\supp(\sigma_l\ast f)\subset K$ ($l=1,\dots,n$) as well. 
		\item\label{Item::AntiDev::Case=1::Sum} For $f\in\dot\Ss(\R^n)$,   $f=\sum_{l=1}^n\partial_{x_l}(\sigma_l\ast f)$.
	\end{enumerate}
\end{prop}

Note that Proposition~\ref{Prop::AntiDev::Case=1} is straightforward if we do not impose condition \ref{Item::AntiDev::Case=1::Support}. For example we have $f=\Delta\Delta^{-1}f=\sum_{l=1}^n\partial_{x_l}(\partial_{x_l}G\ast f)$ for $f\in\dot\Ss(\R^n)$, where $G$ is the fundamental solution to the Laplacian. But  $\partial_{x_l}G$ itself cannot be supported in any $\overline{K}$ unless $K=\R^n$.

\begin{proof}[Proof of Proposition~\ref{Prop::AntiDev::Case=1}]
	By passing to an invertible linear transformation and taking  suitable linear combinations, we can assume that  $K\supseteq(0,\infty)^n$. Thus by convexity we have $K+[0,\infty)^n\subseteq K$. 
	
	Let $g(t)$ be the function as in the proof of \cite[Theorem~4.1]{ExtensionLipschitz}. That is, we have a $g \in \Ss (\R)$ satisfying:
	$$\supp g \subseteq [1, \infty);\quad \int_{\R} g = 1,\text{ and }\int_{\R} t^k g(t)dt = 0\quad\text{for all }k \in \Z_+.$$
	
	Therefore
	$$u_j(t):=2^{j} g (2^jt) - 2^{j-1} g (2^{j-1}t),\quad j\in\Z,$$ 
	are Schwartz functions supported in $\R_+$ and have integral $0$. By the moment condition of $g$ (we have $\int t^kg(t)dt=0$, $k\ge1$) we have $\int t^ku_j(t)dt=0$ for all $k\ge1$ as well. Therefore $u_j\in\dot\Ss(\R)$. By Lemma~\ref{Lem::AntiDev::AFourierLemma}, $\left( \frac{d}{dt} \right)^{-1} u_j\in\dot\Ss(\R)$ is a well-defined moments vanishing Schwartz function supported in $\R_+$.
	
	For $l=1,\dots,n$ and $j\in\Z$, we define $\mu_{l,j}\in\Ss(\R^n)$ as
	\begin{equation}\label{Eqn::AntiDev::ProofCase=1::DefMu}
	\mu_{l,j}(x)
	:=\Big(\prod_{k=1}^{l-1}2^jg(2^jx_k)\Big)\cdot (\partial_t^{-1}u_j)(x_l)\cdot
	\Big(\prod_{k=l+1}^n2^{j-1}g(2^{j-1}x_k)\Big). 
	\end{equation}
	
	Taking Fourier transform of \eqref{Eqn::AntiDev::ProofCase=1::DefMu} we have
	\begin{align} \label{Eqn::AntiDev::ProofCase=1::MuFourier1}
	\widehat \mu _{l,j} ( \xi)
	&= \widehat{g} (2^{-j} \xi_1) \cdots \widehat{g}  (2^{-j} \xi_{l-1})   
	\frac{\widehat{g} (2^{-j} \xi_l)  - \widehat{g} (2^{1-j} \xi_l )}{2\pi i\xi_l }
	\widehat{g} (2^{1-j} \xi_{l+1}) \cdots \widehat{g}  (2^{ 1-j} \xi_n)
	\\\label{Eqn::AntiDev::ProofCase=1::MuFourier2} &= 2^{-j}\widehat\mu_{l,0}(2^{-j}\xi).
	\end{align}
Since $\widehat g(\tau)=1+O(\tau^\infty)$ as $\tau\to0$, from \eqref{Eqn::AntiDev::ProofCase=1::MuFourier1} we have
$\widehat\mu_{l,0}(\xi)=O(|\xi_l|^\infty)$ as $\xi\to0$. In particular $\widehat\mu_{l,0}(\xi)= O(|\xi|^\infty)$ have infinite order vanishing for each $l$. We conclude that $\mu_{l,0}\in\dot\Ss(\R^n)$, so are $\mu_{l,j}(x)= 2^{j (n-1) } \mu_{l,0} (2^jx)$.

	For $l = 1, \dots, n$, we define $\sigma_l$ by 
	\begin{equation}\label{Eqn::AntiDev::ProofCase=1::DefTheta}
	\sigma_l:=\sum_{j\in\Z}\mu_{l,j}.
	\end{equation} 

For each $l=1,\dots,n$ applying Lemma~\ref{Lem::AntiDev::S0Conv} \ref{Item::AntiDev::S0Conv::WellDef} with $\eta_j(x):=2^j\mu_{l,j}(x)=2^{jn}\mu_{l,0}(2^jx)$ and $r=-1$, we know $\sigma_l\in\dot\Ss'(\R^n)$ is defined. 
By Lemma~\ref{Lem::AntiDev::S0Conv} \ref{Item::AntiDev::S0Conv::AstS0}, we see that $\sigma_l\ast f=\sum_{j\in\Z}\mu_{l,j}\ast f$ converges $\dot\Ss(\R^n)$, which gives \ref{Item::AntiDev::Case=1::Schwartz}.

Since $g(t) $ and $\partial_t^{-1}u_j(t)$ are both supported in $\R_+$, by \eqref{Eqn::AntiDev::ProofCase=1::DefMu} we have $\supp\mu_{l,j}\subset[0,\infty)^n\subseteq K$.
	If $f\in\dot\Ss(\R^n)$ is supported in $K$, set $V:= K+\supp f$, then we have $V+K=V $ and $\overline{V}\subsetneq K$. Now $\supp(\mu_{l,j}\ast f)\subset V$ for each $j$, so by taking sum over $j\in\Z$ we get $\supp(\sigma_l\ast f)\subseteq\overline{V}\subset K$. We obtain \ref{Item::AntiDev::Case=1::Support}. 
	
	\medskip
	Finally we prove \ref{Item::AntiDev::Case=1::Sum}. By \eqref{Eqn::AntiDev::ProofCase=1::DefTheta} and \eqref{Eqn::AntiDev::ProofCase=1::MuFourier1} we have, for $\xi\neq0$,
	\begin{align*}
	\sum_{l=1}^n (\partial_{x_l}\sigma_l)^\wedge(\xi)
	&=\sum_{l=1}^n\sum_{j\in\Z}2\pi i\xi_l\widehat\mu_{l,j}(\xi)
	\\
	&=\sum_{l=1}^n\sum_{j\in\Z}\widehat{g} (2^{-j} \xi_1) \cdots \widehat{g}  (2^{-j} \xi_{l-1})   
	\big(\widehat{g} (2^{-j} \xi_l)  - \widehat{g} (2^{1-j} \xi_l )\big)
	\widehat{g} (2^{1-j} \xi_{l+1}) \cdots \widehat{g}  (2^{ 1-j} \xi_n)
	\\
	&=\sum_{j\in\Z}\widehat{g} (2^{-j} \xi_1) \cdots \widehat g(2^{-j}\xi_n)-\widehat{g} (2^{1-j} \xi_1) \cdots \widehat g(2^{1-j}\xi_n)
	\\
	&=\lim\limits_{J\to+\infty}\widehat g (2^{-J} \xi_1) \cdots \widehat g(2^{-J}\xi_n)-\widehat{g} (2^J \xi_1) \cdots \widehat g(2^J\xi_n)=1.
	\end{align*}
	Here we use $\widehat g(0)=1$ and $\lim_{\tau\to\pm\infty}\widehat g(\tau)=0$.
	
	So for $f\in\dot\Ss(\R^n)$, $\widehat f(\xi)=\sum_{l=1}^n (\partial_{x_l}\sigma_l)^\wedge(\xi)\widehat f(\xi)$ holds pointwisely for $\xi\in\R^n$. Taking Fourier inverse we get $f=\sum_{l=1}^n\partial_{x_l} (\sigma_l\ast f) $, giving \ref{Item::AntiDev::Case=1::Sum} and conclude the whole proof.
\end{proof}

We are now ready to prove Theorem~\ref{Thm::MainThm} \ref{Item::MainThm::AntiDev}. 

\begin{proof}[Proof of Theorem~\ref{Thm::MainThm} \ref{Item::MainThm::AntiDev}]
	Fix a convex cone $K\subset\R^n$ temporarily, and let $\sigma_1,\dots,\sigma_n$ be as in the assumption of Proposition~\ref{Prop::AntiDev::Case=1}.
	
	For $f\in\dot\Ss(\R^n)$, taking Fourier transform of $f=\sum_{l=1}^n\partial_{x_l}(\sigma_l\ast f)$ we have $\widehat f(\xi)=\sum_{l=1}^n2\pi i\xi_l\widehat\sigma_l(\xi)\widehat f(\xi)$. Taking the same sum $(m-1)$-more times, we have for every $\xi\in\R^n$ pointwise: 
	\begin{equation}\label{Eqn::AntiDev::ProofCase>1::SumEqn}
	\widehat f(\xi)=\prod_{k=1}^{m}\sum_{l_k=1}^n2\pi i\xi_{l_k}\widehat\sigma_{l_k}(\xi)\widehat f(\xi)=\sum_{|\alpha|=m}\alpha!\prod_{l=1}^n( 2\pi i\xi_l)^{\alpha_l}\widehat\sigma_l(\xi)^{\alpha_l}\widehat f(\xi).
	\end{equation}
	
		Thus we can write $f=\sum_{|\alpha|=m}\partial^\alpha(\sigma_\alpha\ast f)$, where for each $\alpha\in\N^n$ with $|\alpha|=m$ the function $\sigma_\alpha \ast f$ is given by 
		\begin{align}\notag
		\sigma_\alpha\ast f&\textstyle=\alpha!\big(\prod_{l=1}^n\widehat\sigma_l(\xi)^{\alpha_l}\widehat f(\xi)\big)^\vee
		\\
		\label{Eqn::AntiDev::ProofCase>1::SigmaAlphaAstf}&=\alpha!\cdot\underbrace{\sigma_1\ast(\dots\ast(\sigma_1}_{\alpha_1\text{ times}}\ast(\dots\ast (\underbrace{\sigma_n\ast(\dots\ast(\sigma_n}_{\alpha_n\text{ times}}\ast f)\dots).
		\end{align}
		
		By using Proposition~\ref{Prop::AntiDev::Case=1} \ref{Item::AntiDev::Case=1::Schwartz} and \ref{Item::AntiDev::Case=1::Support} recursively in \eqref{Eqn::AntiDev::ProofCase>1::SigmaAlphaAstf}, we have for each $\alpha\in\N^n$ with $|\alpha| = m$, 
		\begin{enumerate}[label=(\alph*)]
		    \item\label{Item::AntiDev::PfAtd::ThetaAlpSchwartz} For $f\in\dot\Ss(\R^n)$, we have $\sigma_\alpha\ast f\in\dot\Ss(\R^n)$. 
			\item \label{Item::AntiDev::PfAtd::ThetaAlpSupp} If $f\in\dot\Ss(\R^n)$ is supported in $K$, then $\supp(\sigma_\alpha\ast f)\subset K$ as well.
		\end{enumerate}
		
	Now we take $K=-\Kb$ in Proposition~\ref{Prop::AntiDev::Case=1}. We shall only construct $\tilde\theta^\alpha=(\tilde\theta^\alpha_j)_{j=1}^\infty \in \Gf(-\Kb)$, since the construction and the proof for $\tilde \eta^\alpha_1$ are the same. 
	
	Define $\tilde\theta^\alpha=(\tilde\theta^\alpha_j)_{j=1}^\infty$ by 
	\begin{equation}\label{Eqn::AntiDev::PfAtd::DefTildeEta}
	\tilde \theta^\alpha_1(x):=2^{m}\sigma_\alpha\ast\theta_1(x),\quad \tilde\theta^\alpha_j(x):=2^{(j-1)n}\tilde\theta^\alpha_1(2^{j-1}x),\quad j>1.
	\end{equation}
	Clearly $\tilde\theta^\alpha$ depends only on $\alpha$ and $\theta$.
	By property \ref{Item::AntiDev::PfAtd::ThetaAlpSchwartz} we have $\tilde\theta^\alpha_1\in\dot\Ss(\R^n)$. By property \ref{Item::AntiDev::PfAtd::ThetaAlpSupp} we have $\supp\tilde\theta_1^\alpha\subset-\Kb$. Therefore $(\tilde\theta^\alpha_j)_{j=1}^\infty\in\Gf(-\Kb)$. 
	
 In view of \eqref{Eqn::AntiDev::ProofCase>1::SigmaAlphaAstf} and  \eqref{Eqn::AntiDev::PfAtd::DefTildeEta} we have $\theta_1=\sum_{|\alpha|=m}\partial^\alpha(\sigma_\alpha\ast\theta_1)
	=2^{-m}\sum_{|\alpha|=m}\partial^\alpha\tilde\theta^\alpha_1$, so for $j\ge1$,  
	\begin{equation}\label{Eqn::Intro::PfAtd::Tmp}
	\theta_j=2^{-m}\sum_{|\alpha|=m}2^{-(j-1)m}\partial^\alpha\tilde\theta^\alpha_j=2^{-jm}\sum_{|\alpha|=m}\partial^\alpha\tilde\theta^\alpha_j,
	\end{equation}
	which is \eqref{Eqn::MainThm::AtDChar}. Thus for $f\in\Ss'(\omega)$,
	\begin{align*}
	T^{\eta,\theta,r}_\omega f
	&=\sum_{j=1}^\infty 2^{jr}\eta_j\ast(\1_\omega \cdot (\theta_j\ast f)) \\
	&= \sum_{|\alpha|=m}\sum_{j=1}^\infty 2^{jr-jm}\eta_j\ast(\1_\omega \cdot (\partial^\alpha\tilde\theta^\alpha_j\ast f)) \\ 
	&= \sum_{|\alpha|=m}\sum_{j=1}^\infty 2^{j(r-m)}\eta_j\ast(\1_\omega \cdot (\tilde\theta^\alpha_j\ast \partial^\alpha f)) 
	= \sum_{|\alpha|=m} T^{\eta,\tilde\theta^\alpha,r-m}_\omega \partial^\alpha f. 
	\end{align*}

	Similarly we can define $(\tilde\eta^\alpha_j)_j$, and by using the identity 
	$\eta_j = 2^{-jm}$ $ \sum_{|\alpha|=m} \partial^{\alpha} \tilde{\eta}_j^{\alpha}$, we have
	$T^{\eta,\theta,r}_\omega f = 
	\sum_{|\alpha|=m} \partial^{\alpha} (T^{\tilde\eta^{\alpha}, \theta, r-m}_\omega f)$.                    
\end{proof}

	\begin{rem}\label{Rmk::AntiDev::SigmaAlpha}
		In \eqref{Eqn::AntiDev::ProofCase>1::SigmaAlphaAstf}, the notation $\sigma_\alpha\ast f$ is indeed a convolution of $f$ and an element in $\dot\Ss'(\R^n)$. Moreover as a linear functional on $\dot\Ss(\R^n)$, $\sigma_\alpha$ is supported in $\overline{K}$. For the proof we leave details to reader.
	\end{rem}

\section{The Quantitative Smoothing Estimate: Proof of Theorem~\ref{Thm::MainThm} \ref{Item::MainThm::Smoothing}}\label{Section::Smoothing}
In this section we prove Theorem~\ref{Thm::MainThm} \ref{Item::MainThm::Smoothing}. Let $\omega=\{x_n>\rho(x')\}\subset\R^n$ be a fixed special Lipschitz domain. We need to prove that for any $m\in\N$, $s<m+r$ and $1\le p\le \infty$,
\begin{equation}\label{Eqn::Smooth::Goal}
\sum_{|\alpha|\le m}\Big(\int_{\overline{\omega}^c}|\dist(x,\omega)^{m+r-s}\partial^\alpha T^{\eta,\theta,r}_\omega f(x)|^pdx\Big)^\frac1p\lesssim_{\eta,\theta,\omega,r,m,s,p}\|f\|_{\Fs_{p\infty}^{s}(\omega)},\quad\forall f\in\Fs_{p\infty}^{s}(\omega).
\end{equation}
For $p=\infty$ we replace the $L^p$-integral by the essential supremum.

In the proof we adopt the following notations for dyadic strips. 

\begin{note}\label{Note::Intro::Strips}
	For a special Lipschitz domain $\omega=\{(x',x_n):x_n>\rho(x')\}$, we denote
	\begin{align*} 
	P_k=P^\omega_k:=&\{(x',x_n):2^{-\frac12-k}<x_n-\rho(x')<2^{\frac12-k}\} \subset \omega, & k\in\Z; 
	\\ 
	S_k=S^\omega_k:=&\{(x',x_n):-2^{\frac12-k}<x_n-\rho(x')<-2^{-\frac12-k}\} \subset \overline\omega^c,&k\in\Z;
	\\
	P_{>k}=P^\omega_{>k}:=&\{0<x_n-\rho(x')<2^{-\frac12-k}\}\subset\omega, & k\in\Z.
	\end{align*} 
\end{note}
Hence by neglecting zero measure sets we have disjoint unions 
\begin{equation}\label{Eqn::Smooth::DisjointStripUnion}
P_{>k}=\coprod_{j=k+1}^\infty P_j,\  \forall k\in\Z;\quad\omega = \coprod_{k \in \Z} P_k;\quad\overline\omega^c = \coprod_{k \in \Z} S_k.
\end{equation}
Moreover since $\|\nabla\rho\|_{L^\infty}<1$, we have
\begin{equation}\label{Eqn::Intro::DistEqv}
2^{-1-k}\le\dist(x,\omega)\le 2^{\frac12-k},\quad\forall \, k\in\Z,\quad x\in S_k.
\end{equation}
Therefore multiplying by a constant factor, we can replace $\dist_\omega$ by $\sum_{k\in\Z}2^{-k}\1_{S_k}$ in \eqref{Eqn::Smooth::Goal}, that is
\begin{equation}\label{Eqn::Smooth::Goal2}
\sum_{|\alpha|\le m}\Big(\sum_{k\in\Z}\big(2^{-k(m+r-s)}\|\partial^\alpha T^{\eta,\theta,r}_\omega f\|_{L^p(S_k)}^p\big)^\frac1p\lesssim_{\eta,\theta,\omega,r,m,s,p}\|f\|_{\Fs_{p\infty}^s(\omega)}.
\end{equation}

The proof of \eqref{Eqn::Smooth::Goal2} turns out to be quite technical. For readers' convenience we first give some sketches of proof in Section \ref{Section::SmoothingIdea} that illustrate all the essential ideas.  

\subsection{Sketch of the proof}\label{Section::SmoothingIdea}
For simplicity we only sketch the case $m=r=0$ and $\omega=\R_+^n=\{x_n>0\}$. That is, we will show the boundedness of the operator $\dist_\omega^{-s}\cdot T^{\eta,\theta,0}_\omega:\Fs_{p\infty}^s(\omega)\to L^p(\overline{\omega}^c)$ for $s<0$. 

By \cite[Theorem~3.2]{ExtensionLipschitz} $\Fs_{p\infty}^{s}(\omega)$ has an equivalent norm 
\begin{equation}\label{Eqn::Intro::FsNorm}
f\mapsto \Big\|\sup\limits_{j\in\N}2^{js}|\phi_j\ast f|\Big\|_{L^p(\omega)},
\end{equation}
where $\phi_0\in\Ss(-\Kb)$ and $(\phi_j)_{j=1}^\infty\in\Gf(-\Kb)$ satisfy $\sum_{j=0}^\infty\phi_j=\delta_0$. We Remark~that the assumption in \cite[Theorem~3.2]{ExtensionLipschitz} is $\phi_0\in C_c^\infty(-\Kb)$, but based on \cite[Theorem~4.1(b)]{ExtensionLipschitz}  same proof works for $\phi_0\in\Ss(-\Kb)$.

Using $\omega = \coprod_{k \in \Z} P_k $, we have $\|\sup_{j\in\N}2^{js}|\phi_j\ast f|\|_{L^p(\omega)}=\big\|(\|\sup_{j\in\N}2^{js}|\phi_j\ast f|\|_{L^p(P_k)})_{k\in\Z}\big\|_{\ell^p}$.
Thus by setting $t=-s$, what we need is to prove for $t>0$ and $1\le p\le\infty$,
\begin{equation}\label{Eqn::Intro::Idea::KeyEqn}
\Big(\sum_{k\in\Z}\big(2^{-kt}\|T^{\eta,\theta,0}_\omega f\|_{L^p(S_k)}\big)^p\Big)^\frac1p\lesssim_{\phi,\eta,\theta,t,p,\omega}\Big(\sum_{k\in\Z}\big\|\sup\limits_{j\in\N}2^{-jt}|\phi_j\ast f|\big\|_{L^p(P_k)}^p\Big)^\frac1p,\ \forall f\in\Fs_{p\infty}^{-t}(\omega).
\end{equation}
Here we take the usual modification when $p=\infty$.

To prove \eqref{Eqn::Intro::Idea::KeyEqn}, we consider the following decomposition
\begin{equation}\label{Eqn::Intro::Idea::KeyDecomp}
T^{\eta,\theta,0}_\omega f=\sum_{j=1}^\infty\sum_{l\in\Z}\eta_j\ast(\1_{P_l}\cdot(\theta_j\ast f)).
\end{equation}

By assumption $\eta_j$ is supported in $-\Kb$, and $\widehat\eta_j$ has infinite order vanishing at $\xi=0$. By neglecting the tail decays we can say $\eta_j$ is concentrated in $\{x_n\approx-2^{-j},|x'|\lesssim 2^{-j}\}$ and $\widehat\eta_j$ is concentrated in $\{|\xi|\approx 2^j\}$. Let us write
\begin{equation*}
    \operatorname{ess-supp}\eta_j\subset\{x_n\approx-2^{-j},|x'|\lesssim 2^{-j}\},\quad\operatorname{ess-supp}\widehat\eta_j\subset\{|\xi|\approx 2^j\}.
\end{equation*}
The same properties hold for $(\theta_j)_{j=1}^\infty$ and $(\phi_j)_{j=1}^\infty$, and we have $\operatorname{ess-supp}\phi_0\subset\{x_n\approx-1,|x'|\lesssim 1\}$, $\operatorname{ess-supp}\widehat\phi_0\subset\{|\xi|\lesssim1\}$.

From \eqref{Eqn::Intro::Idea::KeyDecomp} it is not difficult to see that 
\begin{equation}\label{Eqn::Intro::Idea::SumSupp}
\operatorname{ess-supp}(\eta_j\ast(\1_{P_l}\cdot(\theta_j\ast f)))\subseteq\operatorname{ess-supp}\eta_j+P_l\approx\{x_n\approx 2^{-l}-2^{-j}\}\approx\begin{cases}\{x_n\approx 2^{-l}\}&\text{if }j\gg l,\\\{|x_n|\lesssim 2^{-l}\}&\text{if }j\approx l,\\\{x_n\approx-2^{-j}\}&\text{if }j\ll l.\end{cases}
\end{equation}
Here by $j\gg l$ we mean $j\ge l+C$ where $C$ is a constant that does not depend on $j$ and $l$. Similar meanings hold for $j\approx l$ and $j\ll l$.

If $\1_{S_k}(\eta_j\ast(\1_{P_l}\cdot(\theta_j\ast f)))$ is a non-zero function then $\{x_n\approx-2^{-k}\}\subseteq\{x_n\approx 2^{-l}-2^{-j}\}$, which means $j\approx\min(k,l)$. So roughly
\begin{align}\notag
\1_{S_k} \cdot T^{\eta,\theta,0}_\omega f&\approx\1_{S_k} \cdot \sum_{l\in\Z}\sum_{j=\min(k,l)}\eta_j\ast(\1_{P_l}\cdot(\theta_j\ast f))
\\
\notag&=\1_{S_k} \cdot \sum_{l=k+1}^\infty\eta_k\ast(\1_{P_l}\cdot(\theta_k\ast f)) + \1_{S_k}\cdot \sum_{l=0}^k\eta_l \ast(\1_{P_l}\cdot(\theta_l\ast f))
\\\label{Eqn::Intro::Idea::MoralDecomp}
&=\1_{S_k} \cdot (\eta_k\ast(\1_{P_{>k}}\cdot(\theta_k\ast f))) + \1_{S_k} \cdot \sum_{0\le l\le 
		k}\eta_l\ast(\1_{P_l}\cdot(\theta_l\ast f)).
\end{align}

To estimate the first term, take a further decomposition $\theta_k\ast f=\sum_{j'=0}^\infty\theta_k\ast \phi_{j'}\ast f$, where $(\phi_j)_{j=0}^\infty$ is the same as in \eqref{Eqn::Intro::FsNorm}.

By the support condition of $\widehat{\theta}_k$ and $\widehat\phi_{j'}$, the function $\theta_k\ast\phi_{j'}$ is (essentially) non-zero only when $\{|\xi|\approx 2^k\}\cap\{|\xi|\approx 2^{j'}\}\neq\emptyset $, which is $j'\approx k$. In other words, $\theta_k\ast f\approx\theta_k\ast \phi_k\ast f$.

Note that $P_{>k}=\{0<x_n\lesssim 2^{-k}\}$, and $P_{>k}-\operatorname{ess-supp}\theta_k=\{x_n\approx 2^{-k}\}\approx P_k$. So roughly
\begin{equation}\label{Eqn::Intro::Idea::MoralDecomp2}
\1_{P_{>k}} \cdot (\theta_k\ast f)\approx\1_{P_{>k}} \cdot (\theta_k\ast\phi_k\ast f)\approx\1_{P_{>k}} \cdot (\theta_k\ast(\1_{P_k} \cdot  (\phi_k\ast f))),\quad k\in\Z_+.
\end{equation}

By Young's inequality we have
\begin{equation}\label{Eqn::Intro::Idea::Term1}
\|\theta_k\ast(\1_{P_k} \cdot  (\phi_k\ast f))\|_{L^p(P_{>k})}\le\|\theta_k\|_{L^1}\|\phi_k\ast f\|_{L^p(P_k)}\lesssim\|\phi_k\ast f\|_{L^p(P_k)}\le 2^{kt}\|\sup\limits_{q\in\N}2^{-qt}|\phi_q\ast f|\|_{L^p(P_k)}.
\end{equation}

Thus $2^{-kt}\|\eta_k\ast(\1_{P_{>k}}\cdot(\theta_k\ast f))\|_{L^p(S_k)}\lesssim\|\sup\limits_{q\in\N}2^{-qt}|\phi_q\ast f|\|_{L^p(P_k)}$.

Notice that we have not used the assumption $t>0$ so far. The assumption is required to estimate the second sum of \eqref{Eqn::Intro::Idea::MoralDecomp}, which is the main term.

Since $P_l-\operatorname{ess-supp}\theta_l=\{x_n\approx 2^{-l}\}\approx P_l$, we have
\begin{equation}\label{Eqn::Intro::Idea::MoralDecomp2b}
\1_{P_l} \cdot (\theta_l\ast f)\approx\1_{P_l} \cdot (\theta_l\ast\phi_l\ast f)\approx\1_{P_l} \cdot (\theta_l\ast(\1_{P_l} \cdot (\phi_l\ast f))),\quad l\in\Z_+.
\end{equation}
By Young's inequality we have
\begin{equation*}
\|\theta_l\ast f\|_{L^p(P_l)}\approx\|\theta_l\ast(\1_{P_l} \cdot (\phi_l\ast f))\|_{L^p(P_l)}\le\|\theta_l\|_{L^1}\|\phi_l\ast f\|_{L^p(P_l)}\lesssim\|\phi_l\ast f\|_{L^p(P_l)}\lesssim 2^{lt}\big\|\sup\limits_{q\in\N}2^{-qt}|\phi_q\ast f|\big\|_{L^p(P_l)}.
\end{equation*}

Therefore, taking $L^p$-norm we have
\begin{equation}\label{Eqn::Intro::Idea::Term2}
\begin{aligned}
&2^{-kt}\sum_{l\le k}\|\eta_l\ast(\1_{P_l}\cdot(\theta_l\ast f))\|_{L^p(S_k)}\le2^{-kt}\sum_{l\le k}\|\eta_l\|_{L^1(\R^n)}\|\theta_l\ast f\|_{L^p(P_l)}\lesssim 2^{-kt}\sum_{l\le k}\|\theta_l\ast f\|_{L^p(P_l)}
\\
\lesssim&\sum_{l\le k}2^{-t(k-l)}\big\|\sup\limits_{q\in\N}2^{-qt}|\phi_q\ast f|\big\|_{L^p(P_l)}\le\sum_{l=-\infty}^\infty 2^{-t|k-l|}\big\|\sup\limits_{q\in\N}2^{-qt}|\phi_q\ast f|\big\|_{L^p(P_l)}.
\end{aligned}
\end{equation} 

Combining \eqref{Eqn::Intro::Idea::Term2} with  $2^{-kt}\|\eta_k\ast(\1_{P_{>k}}\cdot(\theta_k\ast f))\|_{L^p(S_k)}\lesssim\|\sup\limits_{q\in\N}2^{-qt}|\phi_q\ast f|\|_{L^p(P_k)}$, we have for $k\in\Z$,
$$2^{-kt}\|T^{\eta,\theta,0}f\|_{L^p(S_k)}\lesssim\sum_{l\in\Z}2^{-t|k-l|}\big\|\sup\limits_{q\in\N}2^{-qt}|\phi_q\ast f|\big\|_{L^p(P_l)}=(2^{-t|\bullet|})\ast\big(\big\|\sup\limits_{q\in\N}2^{-qt}|\phi_q\ast f|\big\|_{L^p(P_\bullet)}\big)[k].$$

By Young's inequality, since $\ell^1(\Z)\ast\ell^p(\Z)\subseteq\ell^p(\Z)$ and $(2^{-t|k|})_{k\in\Z}\in\ell^1(\Z)$, we obtain \eqref{Eqn::Intro::Idea::KeyEqn}, and this completes the outline of proof.

\subsection{The computations}\label{Section::SmoothingPrecise}

Here we assume $1\le p\le\infty$ rather than $0<p\le\infty$.

Multiplying by a large constant if necessary, we can assume that the $\eta,\theta\in\Gf$ in Theorem~\ref{Thm::MainThm} satisfy $\supp\eta_j,\supp\theta_j\subset\{x_n<-2^{-j}\}$ for each $j\ge1$. 

\begin{lemma}\label{Lem::Smooth::ScalingRmk}
	Let $\eta\in\Gf(-\Kb)$.
	There is a linear coordinate change $x=\Phi(y)$, such that in the $y$-coordinates we have $\supp\eta_j\subset(-\Kb)\cap\{y_n<-2^{-j}\}$ for $j\ge1$ and $\omega=\{y_n>\tilde\rho(y')\}$ is still a special Lipschitz domain with respect to $y$.
\end{lemma}
\begin{proof}
	Since $\eta_1$ is supported in $-\Kb$, there is a $0<c\le1$ such that $\supp\eta_1\subset-\Kb\cap\{x_n<-\frac c2\}$.
	
	Take $\Phi(y)=c\cdot y$, we have $\supp(\eta_1\circ\Phi)\subset-\Kb\cap\{y_n<-\frac12\}$. By scaling $\eta_j\circ\Phi$ is supported in $-\Kb\cap\{y_n<-2^{-j}\}$. We obtain the support conditions of $\eta_j$ in $y$-coordinates.
	
	Note that $\Phi^{-1}(\omega)=\{(y',y_n):y_n>c^{-1}\rho(cy')\}$. We see that the function $\tilde \rho(y'):=c^{-1}\rho(cy')$ satisfies $\|\nabla\tilde\rho\|_{L^\infty}=\|\nabla\rho\|_{L^\infty}<1$. Therefore in $y$-coordinates $\omega=\{y_n>\tilde\rho(y')\}$ is still a special Lipschitz domain, finishing the proof.
\end{proof}

To illustrate \eqref{Eqn::Intro::Idea::SumSupp}, we separate the discussion between $j\ll\min(k,l)$ and $j\gg\min(k,l)$ in \eqref{Eqn::Intro::Idea::KeyDecomp}.

When $j\ll\min(k,l)$, the summand $\1_{S_k} \cdot (\eta_j\ast(\1_{P_l} \cdot (\theta_j\ast f)))=0$ is valid, see Remark~\ref{Rmk::Smooth::TrueDecomp}. It is based on the following:
\begin{lemma}\label{Lem::Smooth::StripPlus}
	Let $R_0\in\Z_+$ and let $\omega=\{x_n>\rho(x')\}$ be a special Lipschitz domain with $\|\nabla\rho\|_{L^\infty}\le 1-2^{-R_0}$. Then $(-\Kb\cap\{x_n<-2^{-j}\})+\{x_n-\rho(x')<a\}\subseteq\{x_n-\rho(x')<a-2^{-R_0-j}\}$ holds for every $j\in\N$ and $a\in\R$.
\end{lemma}
\begin{proof}
	Let $u\in\{x_n-\rho(x')<a\}$ and $v\in -\Kb\cap\{x_n<-2^{-j}\}$, we have $u_n-\rho(u')<a$ and $v_n<-\max(2^{-j},|v'|)$. Since $\sup|\nabla\rho|\le1-2^{-R_0}$, we have
	\begin{align*}
	u_n+v_n-\rho(u'+v')\le& u_n-\rho(u')+v_n+|\rho(u'+v')-\rho(u')|
	\\
	\le& u_n-\rho(u')+v_n+(1-2^{-R_0})|v'|\le a+2^{-R_0}v_n<a-2^{-R_0-j}.
	\end{align*}
	This finishes the proof.
\end{proof}
\begin{rem}\label{Rmk::Smooth::TrueDecomp}
	Following Lemma~\ref{Lem::Smooth::StripPlus}, suppose $\|\nabla\rho\|_{L^\infty}\le 1-2^{-R_0}$, then $\1_{S_k} \cdot (\eta_j\ast(\1_{P_l}(\theta_j\ast f)))=0$ for $j\le \min(k,l)-R_0-2$. Therefore we have the following decomposition (cf. \eqref{Eqn::Intro::Idea::MoralDecomp}):
	\begin{equation*}\label{Eqn::Smooth::KeyDecomp}
	\1_{S_k} \cdot (T^{\eta,\theta,r}_\omega f)=\1_{S_k} \cdot \sum_{j=1}^\infty\sum_{l=-\infty}^{j+R_0+1}2^{jr}\eta_j\ast(\1_{P_l}\cdot(\theta_j\ast f))+\1_{S_k} \cdot \sum_{j\ge\max(1,k-R_0-1)}2^{jr}\eta_j\ast(\1_{P_{>j+R_0+1}}\cdot(\theta_j\ast f)).
	\end{equation*}
  This fact is not used explicitly in the proof and we leave the details to the reader. 
\end{rem}

When $j\gg\min(k,l)$, in \eqref{Eqn::Intro::Idea::SumSupp} we say $\1_{S_k} \cdot (\eta_j\ast(\1_{P_l}(\theta_j\ast f)))\approx0$ and ignore this term in \eqref{Eqn::Intro::Idea::MoralDecomp}. But in reality $\1_{S_k} \cdot (\eta_j\ast(\1_{P_l} \cdot (\theta_j\ast f)))$ is not a zero function since the actual support of $\eta_j$ is unbounded\footnote{Unless we consider $\eta_j\in C_c^\infty(-\Kb)$ with only finitely many moments vanishing, which is the assumption in \cite[Section 2]{ExtensionLipschitz}. In this case the extension operator is not ``universal'' since it is not bounded is $\As_{pq}^s$-spaces if $|s|$ large or $p,q$ closed to $0$.}. We say it is negligible, in the way that we have a rapid decay on $j$.

\begin{prop}\label{Prop::Smooth::1Decay}
	Let $R_0\in\Z_+$ and let $\omega=\{x_n>\rho(x')\}$ satisfies $\|\nabla\rho\|_{L^\infty}\le 1-2^{-R_0}$. Let $P_k,P_{>k}$ $(k\in\Z)$ be as in Notation~\ref{Note::Intro::Strips}. Let $\eta=(\eta_j)_{j=1}^\infty\in\Gf(-\Kb)$ satisfies $\supp\eta_j\subseteq\{x_n<-2^{-j}\}$ for all $j$.
	Then for any $M>0$ there is a $C=C(R_0,\eta,M)>0$ such that for $1\le p\le\infty$, $j\in\Z_+$, $k\in\Z$ and $l\le j$,
	\begin{align}
	\label{Eqn::Smooth::1Decay::P>}
	\|\eta_j\ast(\1_{P_{>j}} \cdot g)\|_{L^p(S_k)}&\le C2^{-M\max(0,j-k)}\|g\|_{L^p(P_{>j})},&&\forall g\in L^p(P_{>j});
	\\
	\label{Eqn::Smooth::1Decay::Pl}
	\|\eta_j\ast(\1_{P_l} \cdot g)\|_{L^p(S_k)}&\le C2^{-M(\max(0,j-k)+(j-l))}\|g\|_{L^p(P_l)},&&\forall g\in L^p(P_l).
	\end{align}
\end{prop}Note that the assumption $\supp\eta_j\subseteq\{x_n<-2^{-j}\}$ is guaranteed by Lemma~\ref{Lem::Smooth::ScalingRmk}.
\begin{rem}
	If we take $g=\theta_j\ast f$ in \eqref{Eqn::Smooth::1Decay::Pl} and if $j\gg\min(k,l)$, then we have for arbitrary large $M$ the decay $\|\1_{S_k} \cdot (\eta_j\ast(\1_{P_l} \cdot (\theta_j\ast f)))\|_{L^p(\R^n)}\lesssim_M2^{-M(j-\min(k,l))}\|\theta_j\ast f\|_{L^p(\R^n)}$.
\end{rem}
\begin{proof}[Proof of Proposition~\ref{Prop::Smooth::1Decay}]
	First we claim that there is a $C'=C'(\eta,M)>0$ such that
	\begin{equation}\label{Eqn::Smooth::1Decay::Eta1Decay}
	\int_{|x|>2^{-k}}|\eta_j(x)|dx\le C'2^{-M\max(0,j-k)},\quad \forall j\in\Z_+,\quad k\in\Z.
	\end{equation}
	Since $\eta_1\in\Ss(\R^n)$, we have $\int_{|x|>2^{-k}}|\eta_1|\lesssim_M \min(1,2^{Mk})=2^{-M\max(0,-k)}$ for all $k\in\Z$. By the scaling assumption for $\eta$ we have, for $j\ge1$,
	\begin{equation*}
	\int_{|x|>2^{-k}}|\eta_j(x)|dx=\int_{|x|>2^{-k}}2^{(j-1)n}|\eta_1(2^{j-1}x)|dx=\int_{|\tilde x|>2^{j-1-k}}|\eta_1(\tilde x)|d\tilde x\lesssim_M2^{-M\max(0,j-1-k)}.
	\end{equation*}
	Clearly $2^{-M\max(0,j-1-k)}\approx_M2^{-M\max(0,j-k)}$, so \eqref{Eqn::Smooth::1Decay::Eta1Decay} follows. 
	
	For any open subsets $U,V\subseteq\R^n$, by Young's inequality we have $$\|\eta_j\ast(\1_V \cdot g)\|_{L^p(U)}\le\|\eta_j\|_{L^1\{|x|\ge\dist(U,V)\}}\|g\|_{L^p(V)},\quad g\in L^p(V).$$
	For $U=S_k$ and $V\in\{P_l,P_{>j}\}$ we have the following distance inequalities
	\begin{equation}\label{Eqn::Smooth::1Decay::DistFormTmp}
	\dist(P_l,S_k)\ge\tfrac1{\sqrt2}(2^{-\frac12-l}+2^{-\frac12-k})\ge2^{-1-\min(k,l)},\quad\dist(P_{>j},S_k)\ge\tfrac1{\sqrt2}2^{-\frac12-k}=2^{-1-k},\quad j,k,l\in\Z.
	\end{equation}
	
	Therefore we obtain \eqref{Eqn::Smooth::1Decay::P>} by
	\begin{align*}
	\|\eta_j\ast(\1_{P_{>j}} \cdot g)\|_{L^p(S_k)}\le&\|\eta_j\|_{L^1\{|x|>\dist(P_{>j},S_k)\}}\|g\|_{L^p(P_{>j})}\le\|\eta_j\|_{L^1\{|x|>2^{-1-k}\}}\|g\|_{L^p(P_{>j})}
	\\
	\lesssim&_{\eta,M}2^{-M\max(0,j-1-k)}\|g\|_{L^p(P_{>j})}
	\lesssim_{\eta,M}2^{-M\max(0,j-k)}\|g\|_{L^p(P_{>j})}.
	\end{align*}
	
	Similarly, \eqref{Eqn::Smooth::1Decay::Pl} follows by
	\begin{align*}
	\|\eta_j\ast(\1_{P_l} \cdot g)\|_{L^p(S_k)}\le&\|\eta_j\|_{L^1\{|x|>\dist(P_l,S_k)\}}\|g\|_{L^p(P_l)}\le\|\eta_j\|_{L^1\{|x|>2^{-1-\min(k,l)}\}}\|g\|_{L^p(P_l)}\\
	\lesssim&_{\eta,M}2^{-2M\max(0,j-1-\min(k,l))}\|g\|_{L^p(P_l)}
	\lesssim_{\eta,M}2^{-M(\max(0,j-k)+(j-l))}\|g\|_{L^p(P_l)}.
	\end{align*}    
	Thus we complete the proof.
\end{proof}

Next we prove the precise version of \eqref{Eqn::Intro::Idea::MoralDecomp2} and \eqref{Eqn::Intro::Idea::MoralDecomp2b}. 

\begin{prop}\label{Prop::Smooth::SecondDecomp}
	Let $R_0\in\Z_+$ and let $\omega=\{x_n>\rho(x')\}$ satisfy $\|\nabla\rho\|_{L^\infty}\le 1-2^{-R_0}$. Let $P_k,P_{>k}$ $(k\in\Z)$ be as in Notation~\ref{Note::Intro::Strips}.
	Let $(\phi_j,\psi_j)_{j=0}^\infty$ be a $\Kb$-pair (see Convention~\ref{Conv::Intro::KDyaPair}), and let $\theta=(\theta_j)_{j=1}^\infty\in\Gf(-\Kb)$ satisfies $\supp\theta_j\subseteq\{x_n<-2^{-j}\}$ for all $j\ge1$.
	
	Then for any $M>0$ there is a $C=C(R_0,\phi,\psi,\theta,M)>0$ such that for $1\le p\le\infty$,
	\begin{align}\label{Eqn::Smooth::SecondDecomp::Result1}
	\|\theta_j\ast f\|_{L^p(P_{>j})}&\le C\sum_{l'=-\infty}^{j+R_0}\sum_{j'=0}^\infty 2^{-M(|j-j'|+(j-l'))}\|\phi_{j'}\ast f\|_{L^p(P_{l'})},&\forall j\in\Z_+.
	\\\label{Eqn::Smooth::SecondDecomp::Result2}
	\|\theta_j\ast f\|_{L^p(P_l)}&\le C\sum_{l'=-\infty}^{l}\sum_{j'=0}^\infty 2^{-M(|j-j'|+(j-l'))}\|\phi_{j'}\ast f\|_{L^p(P_{l'})},&\forall j\in\Z_+,\ l\le j.
	\end{align}

\end{prop}

To prove Proposition~\ref{Prop::Smooth::SecondDecomp}, we consider the decomposition $\theta_j\ast f=\sum_{j'=0}^\infty\theta_j\ast \psi_{j'}\ast \phi_{j'}\ast f$. In the use of Young's inequality we need to deal with the $L^1$-norm of $\theta_j\ast \psi_{j'}$:
\begin{lemma}\label{Lem::Smooth::Psi2Decay}
	Let $\theta=(\theta_j)_{j=1}^\infty,\psi=(\psi_j)_{j=1}^\infty\in\Gf$ and let $\psi_0\in\Ss(\R^n)$. Then for any $M>0$ there are $C_1=C_1(\theta.\psi_0,M)>0$ and $C_2=C_2(\theta,\psi,M)>0$ such that
	\begin{align} 
	\label{Eqn::Smooth::Psi2Decay::phi}
	\int_{|x|>2^{-l}} |\theta_j \ast \psi_0(x)|dx\le & C_1 2^{-M(j+\max(0,j-l)) }, &\quad\forall j\in\Z_+, \quad  l\in\Z\cup\{\infty\};
	\\
	\label{Eqn::Smooth::Psi2Decay::psi}
	\int_{|x|>2^{-l}} |\theta_j \ast \psi_{j'} (x)|dx\le & C_2 2^{-M(|j-j'|+\max(0,j-l,j'-l)) },&\forall j,j'\in\Z_+,\quad l\in\Z\cup\{\infty\}. 
	\end{align}
Here, by $l=\infty$, we mean $\|\theta_j\ast\psi_0\|_{L^1 (\R^n)}\le C_12^{-Mj}$ and $\|\theta_j\ast\psi_{j'}\|_{L^1 (\R^n)}\le C_22^{-M|j-j'|}$.
\end{lemma} This is a special case to Proposition~\ref{Prop::Prem::Heideman}.
\begin{proof}
	We have $|\theta_j \ast \psi_0(x)|\le |\theta_j \ast \psi_0(x)|(1+|x|)^M$ for any $M>0$. So \eqref{Eqn::Smooth::Psi2Decay::phi} is the special case to \eqref{Eqn::Prem::Heideman::Main} when $l<\infty$ and \eqref{Eqn::Prem::Heideman::Rmk} when $l=\infty$.

	Note that \eqref{Eqn::Smooth::Psi2Decay::psi} is obtained by \eqref{Eqn::Smooth::Psi2Decay::phi} via the scaling argument: we take $\psi_0(x):=2^{-n}\psi_1(2^{-1}x)$, then $\psi_{j'}(x)=2^{j'n}\psi_0(2^{j'}x)$ for $j'\ge1$. Therefore by \eqref{Eqn::Smooth::Psi2Decay::phi},
	\begin{equation}\label{Eqn::Smooth::Psi2Decay::Tmp1}
	\int_{|x|>2^{-l}} |\theta_j \ast \psi_0(x)|dx\le C_1(\theta,\psi_0,M) 2^{-M(j+\max(0,j-l)) },\quad\forall j\in\N,\quad l\in\Z.
	\end{equation}
	
	When $j\ge j'$, by the scaling assumptions on $\theta$ and $\psi$ we have
	\begin{align*}
	&\theta_j\ast\psi_{j'}(x)=\int 2^{jn}\theta_0(2^j(x-y))2^{j'n}\psi_0(2^{j'}y)dy=\int 2^{j'n}\theta_{j-j'}(2^{j'}x-\tilde y)\psi_0(\tilde y)d\tilde y=2^{j'n}(\theta_{j-j'}\ast\psi_0)(2^{j'}x).
	\end{align*}
	
	Applying \eqref{Eqn::Smooth::Psi2Decay::Tmp1} we get 
	\begin{align*}
	\int_{|x|>2^{-l}} |\theta_j\ast\psi_{j'}(x)|dx=\int_{|x|>2^{j'-l}}|\theta_{j-j'} \ast \psi_0(\tilde x)|d\tilde x\le C_12^{-M(j-j'+\max(0,j-j'+j'-l))}.
	\end{align*}
	We have $2^{-M(j-j'+\max(0,j-j'+j'-l))}=2^{-M(|j-j'|+\max(0,j-l))}=2^{-M(|j-j'|+\max(0,j-l,j'-l))}$ since $j\ge j'$. This proves \eqref{Eqn::Smooth::Psi2Decay::psi} for $j\ge j'$. By symmetry we can swap $\theta$ for $\psi$, which gives \eqref{Eqn::Smooth::Psi2Decay::psi} for $j'\ge j$.
\end{proof}

\begin{proof}[Proof of Proposition~\ref{Prop::Smooth::SecondDecomp}]By assumption $\supp\phi_j,\supp\psi_j\subset-\Kb\cap\{x_n<-2^{-j}\}$ for all $j\ge0$, and we have decomposition $\theta_j\ast f=\sum_{j'=0}^\infty\theta_j\ast \psi_{j'}\ast\phi_{j'}\ast f$. 
	
We now claim to have the following decompositions:
	\begin{align}\label{Eqn::Smooth::SecondDecomp::Decomp1}
	\1_{P_{>j}} \cdot  (\theta_j\ast f)&=\1_{P_{>j}}  \sum_{l'=-\infty}^{j+R_0}\sum_{j'=0}^\infty\theta_j\ast\psi_{j'}\ast(\1_{P_{l'}}\cdot(\phi_{j'}\ast f)),& j\in\Z_+.
	\\\label{Eqn::Smooth::SecondDecomp::Decomp2}
	\1_{P_l} \cdot  (\theta_j\ast f)&=\1_{P_l}\sum_{l'=-\infty}^{l}\sum_{j'=0}^\infty\theta_j\ast\psi_{j'}\ast(\1_{P_{l'}}\cdot(\phi_{j'}\ast f)),&j\in\Z_+,\quad l\in\Z.
	\end{align}
	
	Indeed by ignoring zero measure sets, we have the following disjoint unions for every $j,l\in\Z$:
	\begin{align}\label{Eqn::Smooth::SecondDecomp::DisUni1}
	\R^n&\textstyle=\{x_n-\rho(x')<2^{-\frac12-j-R_0}\}\cup\bigcup_{l'=-\infty}^{j+R_0}P_{l'},
	\\
	&\label{Eqn::Smooth::SecondDecomp::DisUni2}
	\textstyle=\{x_n-\rho(x')<2^{-\frac12-l}\}\cup\bigcup_{l'=-\infty}^lP_{l'}.
	\end{align}
	For \eqref{Eqn::Smooth::SecondDecomp::DisUni1}, by Lemma~\ref{Lem::Smooth::StripPlus} we have 
	\begin{align*}
	\supp(\theta_j\ast\psi_{j'})+\{x_n-\rho(x')<2^{-\frac12-j-R_0}\}\subset&(-\Kb\cap\{x_n<-2^{-j}\})+\{x_n-\rho(x')<2^{-\frac12-j-R_0}\}
	\\
	\subseteq&\{x_n-\rho(x')<2^{-\frac12-j-R_0}-2^{-j-R_0}\}\subset\overline\omega^c.
	\end{align*}
	So $(\supp(\theta_j\ast\psi_{j'})+\{x_n-\rho(x')<2^{-\frac12-j-R_0}\})\cap P_{>j+R_0}=\emptyset$, which gives \eqref{Eqn::Smooth::SecondDecomp::Decomp1}. 

	Similarly for \eqref{Eqn::Smooth::SecondDecomp::DisUni2}, since $-\Kb+\{x_n-\rho(x')<2^{-\frac12-l}\}\subseteq\{x_n-\rho(x')<2^{-\frac12-l}\}$ is disjoint with $ P_l\subset\{x_n-\rho(x')>2^{-\frac12-l}\}$, we have $$(\supp(\theta_j\ast\psi_{j'})+\{x_n-\rho(x')<2^{-\frac12-l}\})\cap P_l=\emptyset,$$ which gives \eqref{Eqn::Smooth::SecondDecomp::Decomp2}.
	
	\medskip
	We now prove \eqref{Eqn::Smooth::SecondDecomp::Result2} by using \eqref{Eqn::Smooth::SecondDecomp::Decomp1}. 
	The proof for \eqref{Eqn::Smooth::SecondDecomp::Result1} is almost identical but simpler, where we replace \eqref{Eqn::Smooth::SecondDecomp::Decomp1} by \eqref{Eqn::Smooth::SecondDecomp::Decomp2}. We leave the proof of \eqref{Eqn::Smooth::SecondDecomp::Result1} to readers.
	
	Similar to \eqref{Eqn::Smooth::1Decay::DistFormTmp} we have
	\begin{equation}\label{Eqn::Smooth::SecondDecomp::DistTmp}
	\dist(P_l,P_{l'})\ge \tfrac1{\sqrt2}(2^{-\frac12-l'}-2^{\frac12-l})\ge 2^{-2-l'},\quad l'\le l-2;\quad \dist(P_l,P_{l'})=0=2^{-\infty},\quad l'\in\{l-1,l\}.
	\end{equation}
	
	By Young's inequality (since $1\le p\le\infty$), we have
	\begin{align*}
	\|\theta_j\ast f\|_{L^p(P_l)}
	&\le\sum_{l'\le l}\sum_{j'=0}^\infty\|\theta_j\ast\psi_{j'}\ast(\1_{P_{l'}} \cdot  (\phi_{j'}\ast f)\|_{L^p(P_l)}
	\\ &\le\sum_{l'\le l}\sum_{j'=0}^\infty\|\theta_j\ast\psi_{j'}\|_{L^1\{|x|>\dist(P_l,P_{l'})\}}\|\phi_{j'}\ast f\|_{L^p(P_{l'})}.
	\end{align*}
	
	Using \eqref{Eqn::Smooth::SecondDecomp::DistTmp} and applying \eqref{Eqn::Smooth::Psi2Decay::phi} on $\|\theta_j\ast\psi_0\|_{L^1}$ and \eqref{Eqn::Smooth::Psi2Decay::psi} on $\|\theta_j\ast\psi_{j'}\|_{L^1}$, $j'\ge1$, we get
	\begin{align*}
	\|\theta_j\ast f\|_{L^p(P_l)}
	&\lesssim_{\theta,\psi,\phi,M}\sum_{l'\le l}\sum_{j'=0}^\infty 2^{-M(|j-j'|+\max(0,j-l',j'-l'))}\|\phi_{j'}\ast f\|_{L^p(P_{l'})}
	\\ &\le \sum_{l'\le l}\sum_{j'=0}^\infty 2^{-M(|j-j'|+j-l')}\|\phi_{j'}\ast f\|_{L^p(P_{l'})}.
	\end{align*}
	
	This finishes the proof of \eqref{Eqn::Smooth::SecondDecomp::Result2}.
\end{proof}

By summing up the estimates in Propositions \ref{Prop::Smooth::1Decay} and \ref{Prop::Smooth::SecondDecomp}, we get the following:
\begin{cor}\label{Cor::Smooth::DecayCor}
	Let $R_0\in\Z_+$, $\omega=\{x_n>\rho(x')\}$, $P_k,P_{>k}$ $(k\in\Z)$, $\eta=(\eta_j)_{j=1}^\infty,\theta=(\theta_j)_{j=1}^\infty\in\Gf(-\Kb)$ and $\phi=(\phi_j)_{j=0}^\infty$ be as in Proposition~\ref{Prop::Smooth::SecondDecomp}.
	Then for any $s\in\R$ and $M>0$ there is a $C=C(R_0,s,M,\eta,\theta,\phi)>0$ such that for $1\le p\le\infty$, $j\in\Z_+$, $k\in\Z$ and $l\le j$,
	\begin{align}\label{Eqn::Smooth::DecayCor::Est1}
	2^{js}\|\eta_j\ast(\1_{P_{>j}} \cdot  (\theta_j\ast f))\|_{L^p(S_k)}&\le C\sum_{l'\in\Z} 2^{-M(\max(0,j-k)+|j-l'|)}\|\sup\limits_{q\in\N}2^{qs}|\phi_l\ast f|\|_{L^p(P_{l'})};
	\\
	\label{Eqn::Smooth::DecayCor::Est2}
	2^{js}\|\eta_j\ast(\1_{P_l} \cdot  (\theta_j\ast f))\|_{L^p(S_k)}&\le C\sum_{l'\in\Z} 2^{-M(\max(0,j-k)+|j-l'|+|j-l|)}\|\sup\limits_{q\in\N}2^{qs}|\phi_{q}\ast f|\|_{L^p(P_{l'})}.
	\end{align}
\end{cor}
\begin{proof}
	By taking $g=\theta_j\ast f$ in \eqref{Eqn::Smooth::1Decay::P>} and using \eqref{Eqn::Smooth::SecondDecomp::Result1}, we have for every $\tilde M>0$
	\begin{equation} \label{Eqn::Smooth::DecayCor::Tmp1}
	\begin{aligned}
	&2^{js}\|\eta_j\ast(\1_{P_{>j}} \cdot  (\theta_j\ast f))\|_{L^p(S_k)}
	\\
	  &\quad \lesssim_{R_0,\eta,\tilde M} 2^{js}2^{-\tilde M\max(0,j-k)}\|\theta_j\ast f\|_{L^p(P_{>j})}
	\\
	&\quad \lesssim_{R_0,\theta,\phi,\tilde M}\sum_{l'=-\infty}^{j+R_0}\sum_{j'=0}^\infty2^{js-\tilde M(\max(0,j-k)+|j-j'|+(j-l'))}\|\phi_{j'}\ast f\|_{L^p(P_{l'})}\\
	&\quad  \le \sum_{l'=-\infty}^{j+R_0}\sum_{j'=0}^\infty2^{js-\tilde  M(\max(0,j-k)+|j-j'|+(|j-l'|+R_0+1))}\|\phi_{j'}\ast f\|_{L^p(P_{l'})}
	\\
	&\quad \lesssim_{R_0} 
	\sum_{l'\in\Z}\sum_{j'=0}^\infty2^{-\tilde M(\max(0,j-k)+|j-l'|)}2^{js-\tilde M|j-j'|}\|\phi_{j'}\ast f\|_{L^p(P_{l'})}.
	\end{aligned}
	\end{equation}
	Similarly, by taking  $g=\theta_j\ast f$ in \eqref{Eqn::Smooth::1Decay::Pl} and using \eqref{Eqn::Smooth::SecondDecomp::Result2} we have, for $l\le j$
	\begin{equation} \label{Eqn::Smooth::DecayCor::Tmp2}
	\begin{aligned} 
	& 2^{js}\|\eta_j\ast(\1_{P_{l}} \cdot  (\theta_j\ast f))\|_{L^p(S_k)}
	\lesssim_{R_0,\eta,\tilde M}2^{-\tilde M(\max(0,j-k)+j-l)}\|\theta_j\ast f\|_{L^p(P_l)}
	\\
	&\quad \lesssim _{R_0,\theta,\phi,\tilde M}\sum_{l'=-\infty}^l\sum_{j'=0}^\infty2^{js-\tilde M(\max(0,j-k)+(j-l)+|j-j'|+(j-l'))}\|\phi_{j'}\ast f\|_{L^p(P_{l'})}
	\\
 	&\quad \le \sum_{l'\in\Z}2^{-\tilde M(\max(0,j-k)+|j-l|+|j-l'|)}\sum_{j'=0}^\infty2^{js-\tilde M|j-j'|}\|\phi_{j'}\ast f\|_{L^p(P_{l'})}.
	\end{aligned}
	\end{equation}
	Here we use the fact $j-l=|j-l|$.
	
	Now for $\tilde M\ge 1+|s|$ we have $2^{js-\tilde M|j-j'|}\le 2^{-|j-j'|+j's}$, so in this case
	\begin{equation} \label{Eqn::Smooth::DecayCor::Tmp3}
	\begin{aligned}
	&\sum_{j'=0}^\infty2^{js-\tilde M|j-j'|}\|\phi_{j'}\ast f\|_{L^p(P_{l'})}    
	\le\sum_{j'=0}^\infty2^{-|j-j'|}\|2^{j's}\phi_{j'}\ast f\|_{L^p(P_{l'})}
	\\ &\qquad\le\sum_{j'=0}^\infty2^{-|j-j'|}\|\sup\limits_{q\in\N}2^{qs}|\phi_{q}\ast f|\|_{L^p(P_{l'})}
	\lesssim\|\sup\limits_{q\in\N}2^{qs}|\phi_{q}\ast f|\|_{L^p(P_{l'})}.
	\end{aligned}
	\end{equation}
	For a given $M>0$, we take $\tilde M=\max(M,1+|s|)$. Combining \eqref{Eqn::Smooth::DecayCor::Tmp1} and \eqref{Eqn::Smooth::DecayCor::Tmp3} we get \eqref{Eqn::Smooth::DecayCor::Est1}; combining \eqref{Eqn::Smooth::DecayCor::Tmp2} and \eqref{Eqn::Smooth::DecayCor::Tmp3} we get \eqref{Eqn::Smooth::DecayCor::Est2}.
\end{proof}

We are now ready to prove Theorem~\ref{Thm::MainThm} \ref{Item::MainThm::Smoothing}.
\begin{proof}[Proof of Theorem~\ref{Thm::MainThm} \ref{Item::MainThm::Smoothing}]
	First, we claim that it is sufficient to prove the case $m=0$.
	
	Indeed, suppose we already have $T^{\eta,\theta,r}_\omega:\Fs_{p\infty}^{s}(\omega)\to L^p(\overline\omega^c,\dist_\omega^{r-s})$ for all $1\le p\le\infty$, $s<r$ and all $\eta,\theta\in\Gf(-\Kb)$. For $\alpha\in\N^n$, we define $\eta^\alpha=(\eta^\alpha_j)_{j=1}^\infty$ as follows
	\begin{equation*}
	\eta^\alpha_j(x):=2^{-j|\alpha|}(\partial^\alpha\eta_j)(x),\quad j\in\Z_+.
	\end{equation*}
	Clearly $\supp\eta^\alpha_j\subseteq\supp\eta_j\subset-\Kb$. Also $\eta_j^{\alpha}$ satisfies the scaling property
	\begin{equation*}
	\eta^\alpha_j(x)=2^{-j|\alpha|}2^{(j-1)n}(\partial^\alpha\eta_1(2^{j-1}\cdot))(x)=2^{(j-1)n}(\partial^\alpha\eta_1)(2^{j-1}x)=2^{(j-1)n}\eta^\alpha_1(2^{j-1}x),\quad j\ge1.
	\end{equation*}
	We conclude that $\eta^\alpha\in\Gf(-\Kb)$. On the other hand, 
	\begin{equation*}
	\partial^\alpha T^{\eta,\theta,r}_\omega f=\sum_{j=1}^\infty2^{jr}(\partial^\alpha\eta_j)\ast(\1_\omega \cdot  (\theta_j\ast f))=\sum_{j=1}^\infty2^{j(r+|\alpha|)}\eta^\alpha_j\ast(\1_\omega \cdot (\theta_j\ast f))=T^{\eta^\alpha,\theta,r+|\alpha|}_\omega f,\quad f\in\Ss'(\omega).
	\end{equation*}
	
	Now for $m\ge1$ and $s<m+r$, we have $s+|\alpha|-m<|\alpha|+r$. By assumption, we have boundedness $[f\mapsto T^{\eta^\alpha,\theta,r+|\alpha|}_\omega f|_{\overline{\omega}^c}]:\Fs_{p\infty}^{s+|\alpha|-m}(\omega)\to L^p(\overline\omega^c,\dist_\omega^{r+|\alpha|-(s+|\alpha|-m)})=L^p(\overline\omega^c,\dist_\omega^{m+r-s})$. Therefore
	\begin{align*}
	\|T^{\eta,\theta,r}_\omega f\|_{W^{m,p}(\overline\omega^c,\dist_\omega^{m+r-s})}\approx &\sum_{|\alpha|\le m}\|\partial^\alpha T^{\eta,\theta,r}_\omega f\|_{L^p(\overline\omega^c,\dist_\omega^{m+r-s})}=\sum_{|\alpha|\le m}\|T^{\eta^\alpha,\theta,r+|\alpha|}_\omega f\|_{L^p(\overline\omega^c,\dist_\omega^{m+r-s})}
	\\
	\lesssim&_{\eta,\theta,\omega,r,s,p}\sum_{|\alpha|\le m}\|f\|_{\Fs_{p\infty}^{s+|\alpha|-m}(\omega)}\lesssim\sum_{|\alpha|\le m}\|f\|_{\Fs_{p\infty}^{s}(\omega)}\approx\|f\|_{\Fs_{p\infty}^{s}(\omega)}.
	\end{align*}
	This completes the proof for $m\ge1$.
	
	\medskip
	To prove the case for $m=0$, write $\omega=\{(x',x_n):x_n>\rho(x')\}$ where $\|\nabla\rho\|_{L^\infty}<1$. There is a (sufficiently large) $R_0\in\Z_+$ such that
	\begin{equation*}
	\|\nabla\rho\|_{L^\infty}\le1-2^{-R_0}.
	\end{equation*}
	
	By Lemma~\ref{Lem::Smooth::ScalingRmk}, we can assume that $\supp\eta_j,\supp\theta_j\subseteq-\Kb\cap\{x_n<-2^{-j}\}$ for all $j\ge1$.
	
	Let $(\phi_j,\psi_j)_{j=0}^\infty$ be a $\Kb$-pair (see Convention~\ref{Conv::Intro::KDyaPair}). By \cite[Theorem~3.2]{ExtensionLipschitz} we have
	\begin{equation}\label{Eqn::Smooth::PfThm::ReductionTmp1}
	\big\|\sup\limits_{j\in\N}2^{js}|\phi_j\ast f|\big\|_{L^p(\omega)}\approx_{\omega,\phi,p,s}\|f\|_{\Fs_{p\infty}^s(\omega)},\quad 1\le p\le\infty,\quad s\in\R.
	\end{equation}
	
	Let $P_k,P_{>k},S_k$ be given in Notation~\ref{Note::Intro::Strips}. By ignoring a zero measure set we have disjoint unions $\omega=\coprod_{k\in\Z}P_k$ and $\overline\omega^c=\coprod_{k\in\Z}S_k$. Therefore,
	\begin{equation}\label{Eqn::Smooth::PfThm::ReductionTmp2}
	\big\|\sup\limits_{j\in\N}2^{js}|\phi_j\ast f|\big\|_{L^p(\omega)}=\big\|\big(\|\sup\limits_{q\in\N}2^{sq}|\phi_q\ast f|\|_{L^p(P_l)}\big)_{l\in\Z}\big\|_{\ell^p}.
	\end{equation}
	
	By \eqref{Eqn::Intro::DistEqv} we have for $x\in\bigcup_{k\in\Z}S_k$ (hence for almost every $x\in\overline\omega^c$) that
	\begin{equation}\label{Eqn::Smooth::PfThm::ReductionTmp3}
	\tfrac12\sum_{k\in\Z}2^{-k}\1_{S_k}(x)=\tfrac1{\sqrt2}\sum_{k\in\Z}2^{-\frac12-k}\1_{S_k}(x)\le\dist(x,\omega)\le\sum_{k\in\Z}2^{\frac12-k}\1_{S_k}(x)=\sqrt2\sum_{k\in\Z}2^{-k}\1_{S_k}(x).
	\end{equation}
	Therefore by enlarging the constant we can replace $\dist_\omega$ with $\sum_{k\in\Z}2^{-k}\1_{S_k}$.
	
	Our goal is to show that for such $\eta,\theta,\phi$, and for $s<r$ we have
	\begin{equation}\label{Eqn::Smooth::PfThm::Goal}
	\big\|\big(2^{k(s-r)}\|T^{\eta,\theta,r}_\omega f\|_{L^p(S_k)}\big)_{k\in\Z}\big\|_{\ell^p}\lesssim_{R_0,\eta,\theta,\phi,r,s}\big\|\big(\|\sup\limits_{q\in\N}2^{sq}|\phi_q\ast f|\|_{L^p(P_l)}\big)_{l\in\Z}\big\|_{\ell^p}.
	\end{equation}
	Assuming for now that \eqref{Eqn::Smooth::PfThm::Goal} holds, by \eqref{Eqn::Smooth::PfThm::ReductionTmp1}, \eqref{Eqn::Smooth::PfThm::ReductionTmp2} and \eqref{Eqn::Smooth::PfThm::ReductionTmp3} we get the boundedness $[f\mapsto( T^{\eta,\theta,r}_\omega f)|_{\overline{\omega}^c}]:\Fs_{p\infty}^{s}(\omega)\to L^p(\overline{\omega}^c,\dist_\omega^{r-s})$ for $1\le p\le\infty$, $r\in\R$ and $s<r$. This completes the proof. 
	
	We now prove \eqref{Eqn::Smooth::PfThm::Goal}. 
	Using \eqref{Eqn::Smooth::DisjointStripUnion}, for each $j\ge1$ we have decomposition $\1_\omega=\1_{P_{>j}}+\sum_{l=-\infty}^j\1_{P_l}$ almost everywhere, so
	\begin{equation*}
	T^{\eta,\theta,r}_\omega f=\sum_{j=1}^\infty 2^{jr}\eta_j\ast(\1_\omega \cdot  (\theta_j\ast f))=\sum_{j=1}^\infty\Big( 2^{jr}\eta_j\ast(\1_{P_{>j}} \cdot  (\theta_j\ast f))+\sum_{l=-\infty}^j 2^{jr}\eta_j\ast(\1_{P_l} \cdot  (\theta_j\ast f))\Big).
	\end{equation*}
It follows that for each $k\in\Z$,
	\begin{multline*}
	2^{k(s-r)}\|T^{\eta,\theta,r}_\omega f\|_{L^p(S_k)}\le\sum_{j=1}^\infty 2^{(k-j)(s-r)}2^{js}\Big(\|\eta_j\ast(\1_{P_{>j}} \cdot (\theta_j\ast f))\|_{L^p(S_k)}
 \\+\sum_{l=-\infty}^j\|\eta_j\ast(\1_{P_l} \cdot  (\theta_j\ast f))\|_{L^p(S_k)}\Big).
	\end{multline*}
	
	Applying Corollary \ref{Cor::Smooth::DecayCor} to  $2^{js}\|\eta_j\ast(\1_{P_{>j}} \cdot (\theta_j\ast f))\|_{L^p(S_k)}$ (via \eqref{Eqn::Smooth::DecayCor::Est1}) and $2^{js}\|\eta_j\ast(\1_{P_l}\cdot(\theta_j\ast f))\|_{L^p(S_k)}$ (via \eqref{Eqn::Smooth::DecayCor::Est2}), we get, for every $M>0$,
	\begin{align*}
	&2^{k(s-r)}\|T^{\eta,\theta,r}_\omega f\|_{L^p(S_k)}
 \\\lesssim&_M \sum_{j=1}^\infty2^{(k-j)(s-r)}\sum_{l'\in\Z}2^{-M(\max(0,j-k)+|j-l'|}\big\|\sup_{q\in\N}2^{qs}|\phi_q\ast f|\big\|_{L^p(P_{l'})}
	\\
	&+\sum_{j=1}^\infty2^{(k-j)(s-r)}\sum_{l=-\infty}^j\sum_{l'\in\Z} 2^{-M(\max(0,j-k)+|j-l|+|j-l'|}\big\|\sup_{q\in\N}2^{qs}|\phi_q\ast f|\big\|_{L^p(P_{l'})}
	\\
	=&\sum_{l'\in\Z} \big\|\sup_{q\in\N}2^{qs}|\phi_q\ast f|\big\|_{L^p(P_{l'})}\sum_{j=1}^\infty2^{(k-j)(s-r)-M(\max(0,j-k)+|j-l'|)}\Big(1+\sum_{l=-\infty}^j2^{-M|j-l|}\Big)
	\\
	\lesssim&\sum_{l'\in\Z} \big\|\sup_{q\in\N}2^{qs}|\phi_q\ast f|\big\|_{L^p(P_{l'})}\sum_{j=-\infty}^\infty2^{(j-k)(r-s)-M(\max(0,j-k)+|j-l'|)}.
	\end{align*}
	Now by assumption $r-s>0$, we can take $M=2(r-s)+1$, thus $2^{(j-k)(r-s)-M\max(0,j-k)}\le 2^{-|j-k|(r-s)}$ and $2^{-M|j-l'|}\le 2^{-|j-l'|}$. Therefore
	\begin{equation*}
	2^{k(s-r)}\|T^{\eta,\theta,r}_\omega f\|_{L^p(S_k)}\lesssim \sum_{l'\in\Z}\big\|\sup_{q\in\N}2^{qs}|\phi_q\ast f|\big\|_{L^p(P_{l'})}\sum_{j\in\Z}2^{-|j-k|(r-s)-|j-l'|}.
	\end{equation*}
	
	Set $u[k]:=2^{-k(r-s)}\|T^{\eta,\theta,r}_\omega f\|_{L^p(S_k)}$ and $v[k]:=\|\sup\limits_{q\in\N}2^{qs}|\phi_q\ast f|\|_{L^p(P_{k})}$ for $k\in\Z$, and then the above calculation can be rewritten as convolutions in $\Z$:
	\begin{equation*}
	u[k]\lesssim_{\eta,\theta,\phi,M,R_0,r,s}\left((2^{-(r-s)|j|})_{j\in\Z}\ast(2^{-(r-s)|j|})_{j\in\Z}\ast v\right)[k],\quad\forall k\in\Z.
	\end{equation*}
	
	By Young's inequality, we get $\|u\|_{\ell^p(\Z)}\lesssim\|v\|_{\ell^p(\Z)}$ since $(2^{-(r-s)|j|})_{j\in\Z},(2^{-|j|})_{j\in\Z}\in\ell^1(\Z)$. In other words, we obtain \eqref{Eqn::Smooth::PfThm::Goal} and finish the proof.
\end{proof}

\section{The Results on Bounded Lipschitz Domains}\label{Section::SecCor}
We now apply our results for special Lipschitz domains to bounded Lipschitz domains. 
The following convention will be used frequently in this section. 
\begin{note}\label{Note::PartitionUnity}
Let $\Omega\subset\R^n$ be a bounded Lipschitz domain, and let $\Uc$ be an open set containing $\overline{\Omega}$.

We use the following objects, which can all be obtained by standard partition of unity argument:
\begin{itemize}
    \item $(U_\nu)_{\nu=1}^M$ are finitely many bounded open sets in $\Uc\subseteq\R^n$.
    \item$(\Phi_\nu:\R^n\to\R^n)_{\nu=1}^M$ as invertible affine linear transformations.
    \item $(\chi_\nu)_{\nu=0}^M$ are $C_c^\infty$-functions on $\R^n$ that take value in $[0,1]$.
    \item $(\omega_\nu)_{\nu=1}^M$ are special type domains on $\R^n$. That is $\omega_\nu=\{x_n>\rho_\nu(x')\}$ for some $\|\nabla\rho_\nu\|_{L^\infty}<1$.
\end{itemize}
They have the following properties: 
    \begin{enumerate}[label=(P.\alph*)]
        \item $\partial\Omega\subset\bigcup_{\nu=1}^MU_\nu \subset \subset \Uc$.
        \item\label{Item::PartitionUnity::chi} $\chi_0\in C_c^\infty(\Omega)$, $\chi_\nu\in C_c^\infty(U_\nu)$ for $1\le \nu\le M$, and $\chi_0+\sum_{\nu=1}^M\chi_\nu^2\equiv1$ in a neighborhood of $\overline\Omega$.
        \item\label{Item::PartitionUnity::Phi} For each $1\le \nu\le M $, $U_\nu=\Phi_\nu(\B^n)$ is the image of the unit ball under $\Phi_\nu$, and $U_\nu\cap\Omega=U_\nu\cap\Phi_\nu(\omega_\nu)$.
    \end{enumerate}
\end{note}

\begin{rem}\label{Rmk::SecCor::BddFact}
The condition \ref{Item::PartitionUnity::chi} and \ref{Item::PartitionUnity::Phi} ensures that for $\nu=1,\dots,M$, $[f\mapsto(\chi_\nu f)\circ\Phi_\nu]$ defines bounded linear maps $\Bs_{pq}^s(\Omega)\to\Bs_{pq}^s(\omega_\nu)$ for all $0<p,q\le\infty$, $s\in\R$ and $\Fs_{pq}^s(\Omega)\to\Fs_{pq}^s(\omega_\nu)$ for all $0<p<\infty$, $0<q\le\infty$, $s\in\R$. This can be done by passing to the extension $\tilde f$ of $f$ (which is defined in $\R^n$) and then using \cite[Theorems 2.8.2]{TriebelTheoryOfFunctionSpacesI} with \cite[Theorem~4.3.2]{TriebelTheoryOfFunctionSpacesII}. Cf. \cite[Proposition~3.2.3]{TriebelTheoryOfFunctionSpacesI}.

Similarly $[g\mapsto \chi_\nu(g\circ\Phi_\nu^{-1})]:\Bs_{pq}^s(\R^n)\to\Bs_{pq}^s(\R^n)$ ($\Fs_{pq}^s(\R^n)\to\Fs_{pq}^s(\R^n)$, respectively) are also bounded linear.
\end{rem} 

The following construction of extension operator for bounded domains is standard, and we provide the details here for readers' convenience. 
\begin{lemma}\label{Lem::SecCor::ConstructEc}
   Let $\Omega$ be a bounded Lipschitz domain and let $\Uc\supset\supset\Omega$ be an open subset. Let $(\chi_\nu)_{\nu=0}^M$, $(\omega_\nu)_{\nu=1}^M$ and $(\Phi_\nu)_{\nu=1}^M$ be as in Notation~\ref{Note::PartitionUnity}.

   Suppose $E_\nu$ is an extension operator for functions defined on $\omega_\nu$, $\nu=1,\dots,M$. Define $\Ec_\Omega$ by
   \begin{equation}\label{Eqn::SecCor::ConstructEc::DefEc}
       \Ec_\Omega f:=\chi_0f+\sum_{\nu=1}^M\chi_\nu (E_\nu[(\chi_\nu f)\circ\Phi_\nu]\circ\Phi_\nu^{-1}),
   \end{equation}
       where $f$ is a function on $\Omega$. Then $\Ec_\Omega$ is an extension operator such that $\Ec_\Omega f$ is supported in $\Uc$.

    Moreover if for every $\nu=1,\dots,M$, $E_\nu:\As_{pq}^s(\omega_\nu)\to\As_{pq}^s(\R^n)$ is bounded for some $\As\in\{\Bs,\Fs\}$, $s\in\R$ and admissible $(p,q)$ (recall it from Convention~\ref{Conv::Intro::AdmissibleIndex}), then $\Ec_\Omega:\As_{pq}^s(\Omega)\to\As_{pq}^s(\R^n)$ is also a bounded linear map.
\end{lemma}

\begin{proof}
In view of properties \ref{Item::PartitionUnity::chi} and \ref{Item::PartitionUnity::Phi}, we have for a function $f$ defined on $\Omega$: 
\begin{align*}
    \Ec_\Omega f|_\Omega
    &\textstyle=\chi_0f+\sum_{\nu=1}^M \chi_\nu (E_\nu[(\chi_\nu f)\circ\Phi_\nu]\circ\Phi_\nu^{-1})|_{U_{\nu} \cap\Omega}
   \\  
   &\textstyle=\chi_0f+\sum_{\nu=1}^M \chi_\nu (E_\nu[(\chi_\nu f)\circ\Phi_\nu]|_{\omega_\nu}\circ\Phi_\nu^{-1})
    \\
    &\textstyle=\chi_0f+\sum_{\nu=1}^M \chi_\nu ((\chi_\nu f)\circ\Phi_\nu\circ\Phi_\nu^{-1})=\chi_0f+\sum_{\nu=1}^M \chi_\nu^2f=f.
\end{align*}

Clearly, $\supp\Ec_\Omega f\subseteq\bigcup_{\nu=0}^M\supp\chi_\nu\subset\Omega\cup\bigcup_{\nu=1}^MU_\nu\subset\subset\Uc$. We see that $\Ec_\Omega$ is an extension operator supported in $\Uc$.

\medskip Assume that for some $s$ and some admissible $(p,q)$, we have $E_\nu:\As_{pq}^s(\omega_\nu)\to\As_{pq}^s(\R^n)$ for all $1\le\nu\le M$. 

By the assumption and the fact that $[f\mapsto (\chi_\nu f)\circ\Phi_\nu]:\As_{pq}^s(\Omega)\to\As_{pq}^s(\omega_\nu)$ is a bounded linear operator (see Remark~\ref{Rmk::SecCor::BddFact}), we have boundedness for the map $\big[f\mapsto E_\nu[(\chi_\nu f)\circ\Phi_\nu]\big]:\As_{pq}^s(\Omega)\to\As_{pq}^s(\R^n)$.

Since $[g\mapsto \chi_\nu( g\circ\Phi_\nu^{-1})]:\As_{pq}^s(\R^n)\to \As_{pq}^s(\R^n)$ is also bounded linear, taking  composition we get the boundedness of $\big[f\mapsto \chi_\nu (E_\nu[(\chi_\nu f)\circ\Phi_\nu]\circ\Phi_\nu^{-1})\big]:\As_{pq}^s(\Omega)\to\As_{pq}^s(\R^n)$. Summing over $1 \le \nu\le M$ and using the boundedness of $[f\mapsto \chi_0f]:\As_{pq}^s(\Omega)\to\As_{pq}^s(\R^n)$ we obtain the boundedness of $\Ec_\Omega$. 
\end{proof}

\subsection{Proof of Theorem~\ref{Thm::Intro::Folklore!}}\label{Section::PfFolkLore}

We start with the proof for special Lipschitz domain using Theorem~\ref{Thm::MainThm} \ref{Item::MainThm::Bdd} and \ref{Item::MainThm::AntiDev}.

First we formulate the analogy of $\partial^\alpha \circ S=S_\alpha\circ\partial^\alpha$ and \eqref{Eqn::Intro::SAlphaEqn} in the Lipschitz setting: for $\alpha\in\N^n\backslash\{0\}$ and for a special Lipschitz domain $\omega$, we can write $\partial^\alpha\circ E_\omega=\sum_\beta E^{\alpha,\beta}_\omega\circ\partial^\beta$ such that $E^{\alpha,\beta}_\omega$ has the similar expression to \eqref{Eqn::Intro::ExtOp}.

\begin{prop}\label{Prop::AntiDev::ForklorePrep}
Let $\omega\subset\R^n$ be a special Lipschitz domain, and let $E_\omega:\Ss'(\omega)\to\Ss'(\R^n)$ be the extension operator given by \eqref{Eqn::Intro::ExtOp}. Then, for any $\alpha\in\N^n\backslash\{0\}$, there are operators $E^{\alpha,0}_\omega,E^{\alpha,\beta}_\omega:\Ss'(\omega)\to\Ss'(\R^n)$ for $\beta\in\N^n$, $|\beta|=|\alpha|$ such that
\begin{enumerate}[(i)]
    \item\label{Item::AntiDev::ForklorePrep::Bdd} $E^{\alpha,0}_\omega, E^{\alpha,\beta}_\omega:\As_{pq}^s(\omega)\to\As_{pq}^s(\R^n)$ are bounded linear maps for $\As\in\{\Bs,\Fs\}$, $s\in\R$, and admissible $(p,q)$.
    \item\label{Item::AntiDev::ForklorePrep::Sum} For every $f\in\Ss'(\omega)$, \begin{equation}\label{Eqn::AntiDev::ForklorePrep}
        \partial^\alpha(E_\omega f)=E^{\alpha,0}_\omega f+\sum_{\beta:|\beta|=|\alpha|}E^{\alpha,\beta}_\omega(\partial^\beta f).
    \end{equation}
\end{enumerate}
\end{prop}

\begin{proof}By assumption $(\phi_j,\psi_j)_{j=0}^\infty$ is a $\Kb$-pair  (see Convention~\ref{Conv::Intro::KDyaPair}) that satisfies
 $\phi_0,\psi_0\in\Ss(\R^n)$ and $\phi=(\phi_j)_{j=1}^\infty,\psi=(\psi_j)_{j=1}^\infty\in\Gf(-\Kb)$. As before, we adopt the  notation $E_\omega=F^{\psi_0,\phi_0}_\omega  + T^{\psi,\phi,0}_\omega$, where 
$F^{\psi_0,\phi_0}_\omega  f := \psi_0 \ast (\1_\omega  \cdot (\phi_0\ast f)) $ and $T^{\psi, \phi,0}_\omega  f := \sum_{j=1}^\infty \psi_j \ast(\1_\omega\cdot(\phi_j\ast f))$.

Define $\psi^\alpha=(\psi^\alpha_j)_{j=1}^\infty$ as 
\begin{equation*}
    \psi^\alpha_j(x):=2^{-j|\alpha|}(\partial^\alpha\psi_j)(x)
    = 2^{-|\alpha|} 2^{(j-1)n}(\partial^\alpha\psi_1)(2^{j-1}x),\quad j\in\Z_+.
\end{equation*}
Clearly $\psi^\alpha\in\Gf(-\Kb)$, and we have for $f\in\Ss'(\omega)$, 
\begin{align*}
    \partial^\alpha E_\omega f=&\partial^\alpha\psi_0\ast(\1_\omega  \cdot (\phi_0\ast f))+\sum_{j=1}^\infty(\partial^\alpha\psi_j)\ast(\1_\omega  \cdot (\phi_j\ast f))
    \\
    =&\partial^\alpha\psi_0\ast(\1_\omega  \cdot (\phi_0\ast f))+\sum_{j=1}^\infty2^{j|\alpha|}\psi^\alpha_j\ast(\1_\omega  \cdot (\phi_j\ast f))=F^{\partial^\alpha\psi_0,\phi_0}_\omega f+T^{\psi^\alpha,\phi,|\alpha|}_\omega f.
\end{align*}

Applying Theorem~\ref{Thm::MainThm} \ref{Item::MainThm::AntiDev} with $\theta=\phi$ we can find some\footnote{In the proof of Theorem~\ref{Thm::MainThm} \ref{Item::MainThm::AntiDev}, see \eqref{Eqn::AntiDev::PfAtd::DefTildeEta}, the construction of $\tilde\eta^\beta=\tilde\phi^\beta$ does not depend on the $\eta=\psi^\alpha$, which is why we do not write it as ``$\tilde\phi^{\alpha,\beta}$''.} $\tilde\phi^\beta=(\tilde\phi^\beta_j)_{j=1}^\infty\in\Gf(-\Kb)$, where $\beta\in\N^n$, $|\beta|=|\alpha|$, such that $T^{\psi^\alpha,\phi,|\alpha|}_\omega=\sum_{|\beta|=|\alpha|}T^{\psi^\alpha,\tilde\phi^\beta,0}_\omega\circ\partial^\beta$. Therefore we have
\begin{equation}\label{Eqn::AntiDev::ForklorePrep::ProofTmp}
    \partial^\alpha E_\omega =F^{\partial^\alpha\psi_0,\phi_0}_\omega +\sum_{\beta:|\beta|=|\alpha|}T^{\psi^\alpha,\tilde\phi^\beta,0}_\omega\circ\partial^\beta .
\end{equation}

By Proposition~\ref{Prop::Prem::BddE0} \ref{Item::Prem::BddE0::F0} $F^{\partial^\alpha\psi_0,\phi_0}_\omega:\Ss'(\omega)\to\Ss'(\R^n)$ is well-defined and $F^{\partial^\alpha\psi_0,\phi_0}_\omega:\As_{pq}^s(\omega)\to\As_{pq}^{s+m}(\R^n)$ is bounded for all $s\in\R$, $m>0$ and admissible $(p,q)$. Similarly by Theorem~\ref{Thm::MainThm} \ref{Item::MainThm::Bdd}, for each $\beta$, $T^{\psi^\alpha,\tilde\phi^\beta,0}_\omega:\Ss'(\omega)\to\Ss'(\R^n)$ is well-defined and $T^{\psi^\alpha,\tilde\phi^\beta,0}_\omega:\As_{pq}^s(\omega)\to\As_{pq}^{s}(\R^n)$ is bounded for all $s\in\R$ and admissible $(p,q)$. 

Therefore by \eqref{Eqn::AntiDev::ForklorePrep::ProofTmp} and the boundedness of $F^{\partial^\alpha\psi_0,\phi_0}_\omega $ and $T^{\psi^\alpha,\tilde\phi^\beta,0}_\omega$, we complete the proof by taking
\begin{equation*}
    E^{\alpha,0}_\omega :=F^{\partial^\alpha\psi_0,\phi_0}_\omega,\quad E^{\alpha,\beta}_\omega:=T^{\psi^\alpha,\tilde\phi^\beta,0}_\omega,\quad\text{for }\alpha,\beta\in\N^n\backslash\{0\},\quad |\beta|=|\alpha|.\qedhere
\end{equation*}
\end{proof}
\begin{rem}\label{Rmk::AntiDev::ForklorePrep}
\begin{enumerate}[(a)]
    \item\label{Item::AntiDev::ForklorePrep::Constant} In view of \eqref{Eqn::Prem::BddE0} and \eqref{Eqn::MainThm::DefT}, we observe that both $E^{\alpha,0}_\omega$ and $E^{\alpha,\beta}_\omega $ take the form of double convolutions as in $E_\omega$. 

    Therefore by Remarks \ref{Rmk::Prem::TildeBdd::OmegaNotDepExplain} and \ref{Rmk::Prem::TildeF::OmegaNotDep}, and by taking restrictions to domains, we see that the operator norms $\|E^{\alpha,0}_\omega\|_{\As_{pq}^s(\omega)\to\As_{pq}^s(\R^n)}$ and $\|E^{\alpha,\beta}_\omega\|_{\As_{pq}^s(\omega)\to\As_{pq}^s(\R^n)}$ depend only on $p,q,s,\alpha$, the choice of $\Kb$-pair $(\phi,\psi)$, and a fixed version of Besov and Triebel-Lizorkin norm $\As_{pq}^s(\lambda)$ in Definition~\ref{Defn::Prem::BsFsDef}, but not on the special domain $\omega$. 
        \item For each $\alpha\neq0$, 
    unlike in the case of the half-plane extension where we have $\partial^\alpha\circ S=S_\alpha\circ\partial^\alpha$, we do not know whether it is possible to construct a single operator $E^\alpha_\omega$ such that $E^\alpha_\omega:\As_{pq}^s(\omega)\to\As_{pq}^s(\R^n)$ and $\partial^\alpha\circ E_\omega=E^\alpha_\omega\circ\partial^\alpha$. Nevertheless, the construction $\partial^\alpha\circ E_\omega=\sum_\beta E^{\alpha,\beta}_\omega\circ \partial^\beta$ with  $E^{\alpha,\beta}_\omega:\As_{pq}^s(\omega)\to\As_{pq}^s(\R^n)$ is sufficient for proving Theorem~\ref{Thm::Intro::Folklore!}.
\end{enumerate}
\end{rem}

We now prove Theorem~\ref{Thm::Intro::Folklore!}, first for special Lipschitz domains and then for general bounded Lipschitz domains. We state the following results for a slightly large class than special Lipschitz functions. 
\begin{prop}\label{Prop::SecCor::SpeLipFolklore}
Let $\omega\subset\R^n$ be an open set with the form $\omega=\{(x',x_n):x_n>\rho(x')\}$, where $\rho:\R^{n-1}\to\R$ is a Lipschitz function of bounded gradient. Then the corresponding thesis in Theorem~\ref{Thm::Intro::Folklore!} is true for $\Omega=\omega$. Namely, for $\As\in\{\Bs,\Fs\}$, for every $s\in\R$, $m\in\Z$ and every admissible $(p,q)$, there is a $C>0$ such that
\begin{gather}\label{Eqn::SecCor::SpeLipFolklore}
    C^{-1}\|f\|_{\As_{pq}^s(\omega)}\le\sum_{|\alpha|\le m}\|\partial^\alpha f\|_{\As_{pq}^{s-m}(\omega)}\le C\|f\|_{\As_{pq}^s(\omega)}.
\end{gather}
\end{prop}
\begin{rem}\label{Rmk::SecCor::SpeLipFolkloreConstant}
The constant $C$ in \eqref{Eqn::SecCor::SpeLipFolklore} does not depend on the choice of the special Lipschitz domain $\omega$: in the proof below we only use the boundedness of $E_\omega$ and $E_\omega^{\alpha,\beta}$, whose operator norms do not depend on $\omega$ due to Remark~\ref{Rmk::AntiDev::ForklorePrep} \ref{Item::AntiDev::ForklorePrep::Constant}. 
\end{rem}

\begin{proof}
By taking a scaling $(x',x_n)\mapsto (x',cx_n)$ where $0<c<\|\nabla\rho\|_{L^\infty}^{-1}$, we can assume $\|\nabla\rho\|_{L^\infty}<1$. In other words, we can assume $\omega$ to be a special Lipschitz domain without loss of generality. 

If $\tilde f\in\As_{pq}^s(\R^n)$ is an extension of $f\in\As_{pq}^s(\omega)$, then $\partial^\alpha\tilde f$ is also an extension of $\partial^\alpha f$, so $\|\partial^\alpha f\|_{\As_{pq}^{s-m}(\omega)}\le\|\partial^\alpha\tilde f\|_{\As_{pq}^{s-m}(\R^n)}$. By \cite[Theorem~2.3.8]{TriebelTheoryOfFunctionSpacesI} for $|\alpha|\le m$ we have $\|\partial^\alpha\tilde f\|_{\As_{pq}^{s-m}(\R^n)}\lesssim_{p,q,s,m}\|\tilde f\|_{\As_{pq}^s(\R^n)}$. Taking infimum of $\tilde f$ over $f$ we get $\sum\limits_{|\alpha|\le m}\|\partial^\alpha f\|_{\As_{pq}^{s-m}(\omega)}\lesssim_{p,q,s,m}\|f\|_{\As_{pq}^s(\omega)}$, which is the inequality on the right hand side of \eqref{Eqn::SecCor::SpeLipFolklore}.  

For the inequality of the left hand side, let $E_\omega:\As_{pq}^{s}(\omega)\to\As_{pq}^s(\R^n)$ be the Rychkov extension operator (\ref{Eqn::Intro::ExtOp}). By Proposition~\ref{Prop::AntiDev::ForklorePrep} we can find operators $E^{\alpha,0}_\omega, E^{\alpha,\beta}_\omega:\As_{pq}^s(\omega)\to\As_{pq}^{s}(\R^n)$ for $\alpha\in\N^n\backslash\{0\}$ and $|\beta|=|\alpha|$ such that $\partial^\alpha\circ E_\omega=E^{\alpha,0}_\omega+\sum_{\beta:|\beta|=|\alpha|}E^{\alpha,\beta}_\omega\circ\partial^\beta $. 

By \cite[Theorem~2.3.8]{TriebelTheoryOfFunctionSpacesI} we have $\|E_\omega f\|_{\As_{pq}^s(\R^n)}\approx\sum_{|\alpha|\le m}\|\partial^\alpha E_\omega f\|_{\As_{pq}^{s-m}(\R^n)}$. Applying Proposition~\ref{Prop::AntiDev::ForklorePrep} we get  
\begin{align*}
    \|f\|_{\As_{pq}^s(\omega)} \approx &_{p,q,s}\|E_\omega f\|_{\As_{pq}^s(\R^n)}\approx_{p,q,s,m}\sum_{|\alpha|\le m}\|\partial^\alpha E_\omega f\|_{\As_{pq}^{s-m}(\R^n)} 
    \\\le&\sum_{|\alpha|\le m}\Big(\|E^{\alpha,0}_\omega f\|_{\As_{pq}^{s-m}(\R^n)}+\sum_{\beta:|\beta|=|\alpha|}\|E^{\alpha,\beta}_\omega\partial^\beta f\|_{\As_{pq}^{s-m}(\R^n)}\Big)
    \\
    \lesssim&_{p,q,s,m,\omega}\|f\|_{\As_{pq}^{s-m}(\omega)}+\sum_{|\beta|\le m}\|\partial^\beta f\|_{\As_{pq}^{s-m}(\omega)}\approx\sum_{|\alpha|\le m}\|\partial^\alpha f\|_{\As_{pq}^{s-m}(\omega)}.
\end{align*}
This completes the proof.
\end{proof}

By combining Lemma~\ref{Lem::SecCor::ConstructEc} and Proposition~\ref{Prop::SecCor::SpeLipFolklore}, we can prove Theorem~\ref{Thm::Intro::Folklore!}.

\begin{proof}[Proof of Theorem~\ref{Thm::Intro::Folklore!}]
 Let $\Omega$ be a bounded Lipschitz domain, we show that
 \[ 
    C^{-1}\|f\|_{\As_{pq}^s(\Omega)}\le \sum_{|\alpha|\le m}\|\partial^\alpha f\|_{\As_{pq}^{s-m}(\Omega)}\le C\|f\|_{\As_{pq}^s(\Omega)}, 
 \]  
 for $\As\in\{\Bs,\Fs\}$, $s\in\R$, $m\in\Z_+$ and admissible $(p,q)$.

Again, the inequality $\sum_{|\alpha|\le m}\|\partial^\alpha f\|_{\As_{pq}^{s-m}(\Omega)}\lesssim\|f\|_{\As_{pq}^s(\Omega)}$ is easy. Let $\tilde{f} $ be an extension of $f$. Then $\partial^\alpha \tilde f $ is an extension of  $\partial^\alpha f$. By \cite[Theorem~2.3.8]{TriebelTheoryOfFunctionSpacesI} $\|\partial^\alpha\tilde f\|_{\As_{pq}^{s-m}(\R^n)}\lesssim_{p,q,s,m}\|\tilde f\|_{\As_{pq}^s(\R^n)}$, taking infimum on both sides over all such $\tilde f$ gives the inequality.

To prove $\|f\|_{\As_{pq}^s(\Omega)}\lesssim\sum_{|\alpha|\le m}\|\partial^\alpha f\|_{\As_{pq}^{s-m}(\Omega)}$, we apply Proposition~\ref{Prop::SecCor::SpeLipFolklore} along with the partition of unity argument.

Using Notation~\ref{Note::PartitionUnity} and Remark~\ref{Rmk::SecCor::BddFact} we have
\begin{equation}\label{Eqn::Seccor::ProofFolklore::Est0}
    \begin{aligned}
     \| f \|_{\As_{pq}^s(\Omega)} 
    &=\textstyle\big\|\chi_0f+\sum_{\nu=1}^M\chi_\nu^2f\big\|_{\As_{pq}^s(\Omega)}\\
    &\lesssim_{p,q} \textstyle\|\chi_0f\|_{\As_{pq}^s(\R^n)}+\sum_{\nu=1}^M\|\chi_\nu^2f\|_{\As_{pq}^s(\Phi_\nu(\omega_\nu))}
    \\ &\textstyle\approx\|\chi_0f\|_{\As_{pq}^s(\R^n)}+\sum_{\nu=1}^M\|(\chi_\nu^2f)\circ\Phi_\nu\|_{\As_{pq}^s(\omega_\nu)}.
\end{aligned}
\end{equation}

By \cite[Theorem~2.3.8]{TriebelTheoryOfFunctionSpacesI} we have
\begin{equation*}
    \textstyle\|\chi_0f\|_{\As_{pq}^s(\R^n)}\lesssim\sum_{|\alpha|\le m}\|\partial^\alpha(\chi_0f)\|_{\As_{pq}^{s-m}(\R^n)}
\end{equation*}

Since $\supp\chi_0\subset\Omega$, for each $\alpha$ we have
\begin{equation}\label{Eqn::SecCor::ProofFolklore::EstChi0}
    \textstyle\|\partial^\alpha(\chi_0f)\|_{\As_{pq}^{s-m}(\R^n)}\approx\|\partial^\alpha(\chi_0f)\|_{\As_{pq}^{s-m}(\Omega)}\lesssim\sum_{\beta+\gamma=\alpha}\|\partial^\beta\chi_0\cdot\partial^\gamma f\|_{\As_{pq}^{s-m}(\Omega)}\lesssim\sum_{\gamma\le\alpha}\|\partial^\gamma f\|_{\As_{pq}^{s-m}(\Omega)}.
\end{equation}
Taking sum over $\alpha$ we get $\|\chi_0f\|_{\As_{pq}^s(\R^n)}\lesssim\sum_{|\alpha|\le m}\|\partial^{\alpha} f\|_{\As_{pq}^{s-m}(\Omega)}$. 

For the terms $(\chi_\nu^2f)\circ\Phi_\nu$ in \eqref{Eqn::Seccor::ProofFolklore::Est0}, by Proposition~\ref{Prop::SecCor::SpeLipFolklore} and the fact that $\Phi_\nu$ are linear maps,
\begin{align*}
    \|(\chi_\nu^2f)\circ\Phi_\nu\|_{\As_{pq}^s(\omega_\nu)}\lesssim\sum_{|\alpha|\le m}\|\partial^\alpha((\chi_\nu^2f)\circ\Phi_\nu)\|_{\As_{pq}^{s-m}(\omega_\nu)}\approx\sum_{|\alpha|\le m}\|(\partial^\alpha(\chi_\nu^2f))\circ\Phi_\nu\|_{\As_{pq}^{s-m}(\omega_\nu)}.
\end{align*}

By \ref{Item::PartitionUnity::chi} and \ref{Item::PartitionUnity::Phi}, $(\partial^\alpha(\chi_\nu^2f))\circ\Phi_\nu$ is supported in $\B^n\cap\overline{\omega_\nu}=\Phi_\nu^{-1}(U_\nu\cap\overline{\Omega})$. Similar to \eqref{Eqn::SecCor::ProofFolklore::EstChi0} we have
\begin{align*}
    &\|(\partial^\alpha(\chi_\nu^2f))\circ\Phi_\nu\|_{\As_{pq}^{s-m}(\omega_\nu)}
    \approx\|(\partial^\alpha(\chi_\nu^2f))\circ\Phi_\nu\|_{\As_{pq}^{s-m}(\B^n\cap\omega_\nu)}
    \\
    \approx&\textstyle\|\partial^\alpha(\chi_\nu^2f)\|_{\As_{pq}^{s-m}(U_\nu\cap\Omega)}
    \lesssim\sum_{\beta+\gamma = \alpha}\|\partial^\beta(\chi_\nu^2)\partial^\gamma f\|_{\As_{pq}^{s-m}(U_\nu\cap\Omega)}
    \lesssim\sum_{\gamma\le\alpha}\|\partial^\gamma f\|_{\As_{pq}^{s-m}(U_\nu\cap\Omega)}
    \\
    \le&\textstyle\sum_{\gamma\le\alpha}\|\partial^\gamma f\|_{\As_{pq}^{s-m}(\Omega)}\le\sum_{|\gamma|\le m}\|\partial^\gamma f\|_{\As_{pq}^{s-m}(\Omega)}.
\end{align*}

Taking sum over $|\alpha|\le m$ and $1\le \nu\le M$ we get $\sum_{\nu=1}^M\|(\chi_\nu^2 f)\circ\Phi_\nu\|_{\As_{pq}^s(\omega_\nu)}\lesssim\sum_{|\gamma|\le m}\|\partial^\gamma f\|_{\As_{pq}^{s-m}(\Omega)}$. In view of the estimates and \eqref{Eqn::Seccor::ProofFolklore::Est0}, the proof is now complete. 
\end{proof}

\begin{rem}\label{Rmk::SecCor::RmkFolklore}
	Alternatively, one can apply partition of unity argument on Proposition~\ref{Prop::AntiDev::ForklorePrep} to construct operators $\Ec_\Omega^{\alpha,\beta}$ for $\alpha,\beta\in\N^n$ such that 
	\begin{itemize}
		\item $\Ec_\Omega^{\alpha,\beta}:\As_{pq}^s(\Omega)\to\As_{pq}^s(\R^n)$ are all bounded linear.
		\item $\partial^\alpha\circ\Ec_\Omega=\sum_{\beta:|\beta|\le|\alpha|}\Ec_\Omega^{\alpha,\beta}\circ\partial^\beta$.
	\end{itemize}
	Theorem~\ref{Thm::Intro::Folklore!} then follows immediately. We leave the details of the construction to reader.
\end{rem} 
\begin{rem}\label{Rmk::SecCor::FolkloreConstant}
    In the proof above we see that the constants $C$ in \eqref{Eqn::Intro::Folklore!::Bs} and \eqref{Eqn::Intro::Folklore!::Fs} are (upper) stable under small Lipschitz perturbation of the domain. More precisely, the upper stability of $C$ can be stated as follows: 

    \smallskip
    \textit{Let $\Omega$ be a bounded Lipschitz domain. Let $\sigma:\R^n\to\R$ be a Lipschitz function\footnote{When $\Omega$ and $\sigma$ are both $C^1$, such $\sigma$ is called a \textit{defining function} of $\Omega$. For bounded Lipschitz domain in general, our definition of $\sigma$ is different from the \textit{Lipschitz defining function} appeared in \cite[Chapter 0]{ShawBVPLipschitz}.} that satisfies the following:}
    \begin{itemize}
        \item \textit{$\Omega=\{x:\sigma (x)<0\}$ and $\overline\Omega^c=\{x:\sigma(x)>0\}$.}
        \item \textit{There exist an $\epsilon_0>0$ and a continuous vector field $X$ defined on $\{x:-\epsilon_0<\sigma(x)<\epsilon_0\}$ such that $\nabla\sigma(x)\cdot X(x)>\epsilon_0$ for almost every $x$ such that $|\sigma(x)|<\epsilon_0$.} 
    \end{itemize}
    
    \textit{Then there exists a $\delta_0=\delta_0(\Omega,\sigma,\epsilon_0)>0$ such that, if $\sigma':\R^n\to\R$ satisfies $\|\sigma'-\sigma\|_{C^{0,1}(\R^n)}<\delta_0$, then the new domain $\Omega' =\{\sigma'<0\}$ is still bounded Lipschitz. Thus Theorem~\ref{Thm::Intro::Folklore!} holds with $\Omega$ replaced by $\Omega'$.} 
    
    \textit{Moreover, for every $p,q,s,m$ ($p<\infty$ for $\Fs$-cases) there is a $C_*=C_*(\Omega,\delta_0,p,q,s,m)>0$ (depending also on the choice of $\As_{pq}^s$-norm in Definition~\ref{Defn::Prem::BsFsDef}), such that for all such $\Omega'$ and $f\in\As_{pq}^s(\Omega')$, we have $C_*^{-1}\|f\|_{\As_{pq}^s(\Omega')}\le \sum_{|\alpha|\le m}\|\partial^\alpha f\|_{\As_{pq}^{s-m}(\Omega')}\le C_*\|f\|_{\As_{pq}^s(\Omega')}$.}
    
    \smallskip
    We sketch the proof here and leave the details to the reader. Let $(U_\nu,\Phi_\nu,\chi_\nu,\omega_\nu)$ be the objects of partition of unity for $\Omega$ given in Notation~\ref{Note::PartitionUnity}. When $\delta_0$ is small enough, $\Omega'$ is still a Lipschitz domain, and there exist special Lipschitz domains $\omega'_\nu$ ($1\le \nu\le M$) such that $U_\nu\cap \Omega'=U_\nu\cap\Phi_\nu(\omega'_\nu)$. Using Remark~\ref{Rmk::SecCor::SpeLipFolkloreConstant} and keeping track of the constants in the proof of Theorem~\ref{Thm::Intro::Folklore!} above, we see that $C_*$ depends only on $\Omega,\chi_\nu,\Phi_\nu$ and the norm definitions of spaces $\As_{pq}^s(\R^n)$, $\As_{pq}^{s-m}(\R^n)$, but not on $\Omega'$ itself.
\end{rem}
\subsection{Proof of Theorem~\ref{Prop::Intro::SmoothBddDom}}\label{Section::SecCor::PfSmoothBddDom}

First we prove the analog of Theorem~\ref{Prop::Intro::SmoothBddDom} on special Lipschitz domains using Theorem~\ref{Thm::MainThm} \ref{Item::MainThm::Smoothing}.

\begin{prop}[Quantitive smoothing estimate for $E_\omega$] \label{Thm::Intro::SmoothSpeDom}
  Let $\omega\subset\R^n$ be a special Lipschitz domain. Let $E_\omega$ be given by \eqref{Eqn::Intro::ExtOp} or \eqref{Eqn::Intro::EfromT}. Then we have the following estimate:  \begin{equation}\label{Eqn::MainThm::SmooBddEx}
  	\big[f \mapsto (E_\omega f)|_{\overline\omega^c}\big]:\Fs_{p\infty}^{s}(\omega)\to W^{m,p}(\overline\omega^c,\dist_\omega^{m-s}),\quad 1\le p\le\infty,\quad m\in\N,\quad s<m.
  	\end{equation} 
  	In particular, $E_\omega f \in C^\infty_\loc(\overline{\omega}^c)$ if $f\in\Fs_{p\infty}^s(\omega)$ for some $s\in\R$ and $1\le p\le\infty$.
\end{prop}  
\begin{proof}
We can write $E_\omega f=F^{\psi_0,\phi_0}_\omega f+T^{\psi,\phi,0}_\omega f$ from \eqref{Eqn::Intro::EfromT}. Here
$F^{\psi_0,\phi_0}_\omega  f := \psi_0 \ast (\1_\omega  \cdot (\phi_0\ast f)) $, $T^{\psi, \phi,0}_\omega  f := \sum_{j=1}^\infty \psi_j \ast(\1_\omega\cdot(\phi_j\ast f))$, $(\phi_j,\psi_j)_{j=0}^\infty$ is a $\Kb$-pair and $\phi=(\phi_j)_{j=1}^\infty,\psi=(\psi_j)_{j=1}^\infty\in\Gf$.
	
	By Theorem~\ref{Thm::MainThm} \ref{Item::MainThm::Smoothing} we have $T^{\psi,\phi,0}_\omega :\Fs^{s}_{p\infty}(\omega)\to W^{m,p}(\overline\omega^c,\dist_\omega^{m-s})$. 
	
	By Proposition~\ref{Prop::Prem::BddE0} \ref{Item::Prem::BddE0::F0} we have $F^{\psi_0,\phi_0}_\omega:\Fs^{s}_{p\infty}(\omega)\to\Fs_{p\infty}^{m+1}\subset W^{m,p}(\R^n)$ is bounded. In particular it is bounded as $\Fs^{s}_{p\infty}(\omega)\to W^{m,p}(\overline\omega^c,\dist_\omega^{m-s})$. This gives the estimate for $E_\omega:\Fs^{s}_{p\infty}(\omega)\to W^{m,p}(\overline\omega^c,\dist_\omega^{m-s})$.
	
    Since $W^{m,p}(\overline\omega^c,\dist_\omega^{m-s})\subset W^{m,p}_\loc(\overline{\omega}^c)$ and $C^\infty_\loc(\overline{\omega}^c)=\bigcap_{m=1}^\infty W^{m,p}_\loc(\overline{\omega}^c)$ for every $1\le p\le\infty$. Let $m\to+\infty$ we see that $E_\omega f\in C^\infty_\loc(\overline{\omega}^c)$.
\end{proof}

With the setup of Lemma~\ref{Lem::SecCor::ConstructEc} we can now state and prove the blow-up estimate for Rychkov's extension operator on bounded Lipschitz domain.

\begin{proof}[Proof of Theorem~\ref{Prop::Intro::SmoothBddDom}]
Let $\Ec$ be as in \eqref{Eqn::SecCor::ConstructEc::DefEc} with $E_\nu=E_{\omega_\nu}$ be the extension operators given by \eqref{Eqn::Intro::ExtOp}. By Proposition~\ref{Thm::Intro::SmoothSpeDom} we have $E_\nu:\Fs_{p\infty}^{s}(\omega_\nu)\to W^{m,p}(\overline{\omega_\nu}^c,\dist_{\omega_\nu}^{m-s})$ for $1\le p\le\infty$, $m\in\N$ and $s<m$. By \ref{Item::PartitionUnity::Phi} we can find a $C>0$ such that
\begin{equation*}
    C^{-1}\dist_{\Omega}(\Phi_\nu(x))\le\dist_{\omega_\nu}(x)\le C\dist_{\Omega}(\Phi_\nu(x)),\quad 1\le\nu\le M,\quad x\in\B^n(=\Phi_\nu^{-1}(U_\nu)).
\end{equation*}
Therefore by Remark~\ref{Rmk::SecCor::BddFact}, for each $1 \leq \nu \leq M$, 
\begin{align*}
    \| \chi_{\nu} E_\nu[(\chi_\nu f)\circ\Phi_\nu]\circ\Phi_\nu^{-1}\|_{W^{m,p}(\overline\Omega^c ,\dist_\Omega^{m-s})}
    &\approx\|E_\nu[(\chi_\nu f)\circ\Phi_\nu]\|_{W^{m,p}(\overline{\omega_\nu}^c \cap \B^n ,\dist_{\omega_\nu}^{m-s})}
    \\ &\lesssim\|(\chi_\nu f)\circ\Phi_\nu\|_{\Fs_{p\infty}^s(\omega_\nu)}\lesssim\|f\|_{\Fs_{p\infty}^s(\Omega)}.
\end{align*} 
Here we recall that when $p=\infty$ we have $\Fs_{\infty\infty}^s=\Bs_{\infty\infty}^s$, the boundedness result is still valid. Since $\Ec_\Omega = \sum_{\nu=1}^M  \chi_{\nu} E_\nu[(\chi_\nu f)\circ\Phi_\nu]\circ\Phi_\nu^{-1}$, we get $\Ec_\Omega: \Fs_{p\infty}^{s}(\Omega)\to W^{m,p}(\overline\Omega^c,\dist_{\Omega}^{m-s})$ for $m > s$.

\medskip
Let $f\in\Ss'(\Omega)$, since $\Omega$ is a bounded domain, $f$ can be written as the sum of derivatives of bounded continuous in $\Omega$. Therefore $f\in\Bs_{\infty\infty}^{-M}(\Omega)$ for some $M>0$. Take $s=-M$ and $p=\infty$ in \eqref{Eqn::SecCor::SmoothBddDom::Bdd}, we have $\Ec_\Omega f\in W^{m,\infty}_\loc(\overline{\Omega}^c)$ for all $m\in\Z_+$, thus $\Ec_\Omega f\in C^\infty_\loc(\overline{\Omega}^c)$.
\end{proof}

\subsection{Proof of Proposition~\ref{Prop::Intro::AntiDevBddDom}}\label{Section::SecCor::AtdBddDom}

In this part we prove Proposition~\ref{Prop::Intro::AntiDevBddDom}.

Our goal  is to construct a class of operators $(\Sc_\Omega^{m,\alpha})_{|\alpha|\le m}:\Ss'(\R^n)\to\Ss'(\R^n)$ associated with a bounded Lipschitz domain $\Omega\subset\R^n$ and an $m\in\Z_+$, such that
\begin{itemize}
    \item $\Sc_\Omega^{m,\alpha}$ gain $m$-derivatives in Besov and Triebel-Lizorkin spaces.
    \item $\Sc_\Omega^{m,\alpha}g|_\Omega=0$ hold whenever $g|_\Omega=0$.
\end{itemize} 

As usual we begin the discussion with special Lipschitz domains: 

\begin{lemma}\label{Lem::SecCor::AntiDevSpeDom}
There are convolution operators\footnote{For a convolution operator $S$, we mean $Sf=\phi\ast  f$ for some $\phi\in\Ss'(\R^n)$. In particular it is linear.} $S_0,S_\alpha:\Ss'(\R^n)\to\Ss'(\R^n)$, $\alpha\in\N^n\backslash\{0\}$, such that:
\begin{enumerate}[(i)]
    \item\label{Item::SecCor::AntiDevSpeDom::Bdd} For each $\alpha\neq0$, $S_0,S_\alpha:\As_{pq}^s(\R^n)\to\As_{pq}^{s+|\alpha|}(\R^n)$ are bounded for $\As\in\{\Bs,\Fs\}$, for all $s\in\R$ and admissible $(p,q)$.
    \item\label{Item::SecCor::AntiDevSpeDom::Sum}  $ f=S_0f+\sum_{|\alpha|=m}\partial^\alpha S_\alpha  f$ holds for every $m\in\Z_+$ and $f\in\Ss'(\R^n)$.
    \item \label{Item::SecCor::AntiDevSpeDom::Vanish}Let $\omega\subset\R^n$ be a special Lipschitz domain. If $f\in\Ss'(\R^n)$ satisfies $ f|_\omega \equiv 0$, then $S_0f|_\omega=S_\alpha f|_\omega \equiv 0$ for all $\alpha\in\N^n\backslash\{0\}$.
\end{enumerate}
\end{lemma}
\begin{proof}
Let $(\phi_j,\psi_j)_{j=0}^\infty$ be a $\Kb$-pair (see Convention~\ref{Conv::Intro::KDyaPair}).

By Theorem~\ref{Thm::MainThm} \ref{Item::MainThm::AntiDev} we can find $\tilde\eta^\alpha\in\Gf(-\Kb)$ for $\alpha\in\N^n\backslash\{0\}$ such that $\psi_j=2^{-jm}\sum_{|\alpha|=m}\partial^\alpha\tilde\eta^\alpha_j$ for each $j\in\Z_+$, $m\in\Z_+$. Therefore using notations in \eqref{Eqn::Prem::TildeTBdd::DefTildeT} and \eqref{Eqn::Prem::BddTildeE0} with $\Omega=\R^n$, we have for $f\in\Ss'(\R^n)$,
\begin{equation}\label{Eqn::SecCor::AntiDevSpeDom::PfDecomp}
    f=\sum_{j=0}^\infty\psi_j\ast \phi_j\ast f=\psi_0\ast\phi_0\ast f+\sum_{j=1}^\infty\sum_{|\alpha|=m}2^{-jm}\partial^\alpha\tilde\eta^\alpha_j\ast \phi_j\ast f=\tilde F^{\psi_0,\phi_0}_{\R^n}f+\sum_{|\alpha|=m}\partial^\alpha\tilde T^{\tilde\eta^\alpha,\phi,-m}_{\R^n}f.
\end{equation}
In this case $\tilde F^{\psi_0,\phi_0}_{\R^n}$ and $\tilde T^{\tilde\eta^\alpha,\phi,-m}_{\R^n}$ ($\alpha\neq0$) are all convolution operators.

Now we take $S_0$ and $S_\alpha$ as the following, which gives \ref{Item::SecCor::AntiDevBddDom::Sum} immediately by \eqref{Eqn::SecCor::AntiDevSpeDom::PfDecomp}:
\begin{equation*}
    S_0:=\tilde F^{\psi_0,\phi_0}_{\R^n},\quad S_\alpha:=\tilde T^{\tilde\eta^\alpha,\phi,-|\alpha|}_{\R^n},\quad \alpha\in\N^n\backslash\{0\}.
\end{equation*}

By Proposition~\ref{Prop::Prem::BddE0} \ref{Item::Prem::BddE0::TildeF0}, $\tilde F^{\psi_0,\phi_0}_{\R^n}:\Ss'(\R^n)\to\Ss'(\R^n)$ has boundedness $\tilde F^{\psi_0,\phi_0}_{\R^n}:\As_{pq}^s(\R^n)\to\As_{pq}^{s+m}(\R^n)$. By Proposition~\ref{Prop::Prem::TildeTBdd}, $\tilde T^{\tilde\eta^\alpha,\phi,-m}_{\R^n}:\Ss'(\R^n)\to\Ss'(\R^n)$ has boundedness $\tilde F^{\psi_0,\phi_0}_{\R^n}:\As_{pq}^s(\R^n)\to\As_{pq}^{s+m}(\R^n)$. Replacing $\tilde F^{\psi_0,\phi_0}_{\R^n},\tilde T^{\tilde\eta^\alpha,\phi,-m}_{\R^n}$ by $S_0,S_\alpha$ we obtain the well-definedness and the conclusion \ref{Item::SecCor::AntiDevSpeDom::Bdd}.

By assumption $\supp\phi_j,\supp\psi_j\subset-\Kb$ for $j\ge0$ and $\supp\tilde\eta^\alpha_j\subset-\Kb$ for $j\ge1$.
The assumption $ f|_\omega=0$ implies $\supp f\subseteq{\omega^c}$. And since ${\omega^c}-\Kb\subseteq{\omega^c}$, we get $\supp S_0 f,\supp S_\alpha f\subseteq{\omega^c}$, which gives \ref{Item::SecCor::AntiDevSpeDom::Vanish}.
\end{proof}
\begin{rem}
    Alternatively, there is a simpler way to define $S_0,S_\alpha$ using \eqref{Eqn::AntiDev::ProofCase>1::SigmaAlphaAstf}:
    \begin{equation*}
         \textstyle f=\phi_0\ast f+\sum_{|\alpha|=m}\partial^\alpha(\sigma_\alpha-\sigma_\alpha\ast\phi_0)\ast  f,\quad \forall m\in\Z_+.
    \end{equation*}
    
    One can see that $\|(\sigma_\alpha-\sigma_\alpha\ast\phi_0)\ast  f\|_{\As_{pq}^{s+|\alpha|}(\R^n)}\lesssim\| f\|_{\As_{pq}^s(\R^n)}$  for all such $p,q,s$, therefore $S_0:=\phi_0\ast(-)$ and $S_\alpha:=(\sigma_\alpha-\sigma_\alpha\ast\phi_0)\ast (-)$ also satisfy the conclusions of Lemma~\ref{Lem::SecCor::AntiDevSpeDom}. We omit the details.
\end{rem}
\begin{proof}[Proof of Proposition~\ref{Prop::Intro::AntiDevBddDom}]
Let $(\omega_\nu)_{\nu=1}^M$, $(\Phi_\nu)_{\nu=1}^M$ and $(\chi_\nu)_{\nu=0}^M$ be as in Notation~\ref{Note::PartitionUnity}. 

Define $\chi_\infty: = 1- \chi_0 - \sum_{\nu=1}^M \chi_\nu^2$. Let $\zeta_\infty\in C^\infty(\R^n)$ be such that 
\begin{itemize}
    \item $1-\zeta_\infty\in C_c^\infty(\R^n)$,\quad $\supp\zeta_\infty\subseteq\R^n\backslash\overline\Omega$,\quad and $\zeta_\infty\chi_\infty\equiv\chi_\infty$.
\end{itemize}

Let $S_0,S_\alpha:\Ss'(\R^n)$ ($\alpha\in\N^n\backslash\{0\}$) be the operators as in Lemma~\ref{Lem::SecCor::AntiDevSpeDom}. For every $m\in\Z_+$ and $f\in\Ss'(\R^n)$,
\begin{equation}\label{Eqn::SecCor::PfAtdBd::Sum1}
\begin{aligned}
    f=&\chi_0f+\zeta_\infty\chi_\infty f+\sum_{\nu=1}^M\chi_\nu^2f
    =\chi_0f+\zeta_\infty\chi_\infty f+\sum_{\nu=1}^M\chi_\nu((\chi_\nu f)\circ\Phi_\nu\circ\Phi_\nu^{-1})
    \\
    =&\textstyle S_0[\chi_0f]+\zeta_\infty\cdot S_0[\chi_\infty f]+\sum_{\nu=1}^M\chi_\nu\cdot( S_0[(\chi_\nu f)\circ\Phi_\nu]\circ\Phi_\nu^{-1})
    \\
    &\textstyle+\sum_{|\alpha|=m}\Big(\partial^\alpha S_\alpha(\chi_0f)+\zeta_\infty\cdot\partial^\alpha S_\alpha(\chi_\infty f)+\sum_{\nu=1}^M\chi_\nu\cdot((\partial^\alpha S_\alpha[(\chi_\nu f)\circ\Phi_\nu])\circ\Phi_\nu^{-1})\Big).
\end{aligned}
\end{equation}

Since $\Phi_\nu$ are all affine linear transformations, we can find constants $c^\alpha_{\nu,\beta}\in\R$ for $1\le\nu\le M$ and $\alpha,\beta\in\N^n$, such that
\begin{equation*}
    (\partial^\alpha g)\circ\Phi_\nu^{-1}=\sum_{\beta:|\beta|=|\alpha|}c^\alpha_{\nu,\beta}\cdot\partial^\beta(g\circ\Phi_\nu^{-1}),\quad\forall\alpha\in\N^n,\quad g\in\Ss'(\R^n).
\end{equation*}

Therefore for $|\alpha|=m$,
\begin{equation}\label{Eqn::SecCor::PfAtdBd::Sum2}
    \begin{aligned}
    &\partial^\alpha S_\alpha[\chi_0f]+\zeta_\infty\cdot\partial^\alpha S_\alpha[\chi_\infty f]+\sum_{\nu=1}^M\chi_\nu\cdot((\partial^\alpha S_\alpha[(\chi_\nu f)\circ\Phi_\nu])\circ\Phi_\nu^{-1})
    \\
    =&\partial^\alpha S_\alpha[\chi_0f]+\zeta_\infty\cdot\partial^\alpha S_\alpha[\chi_\infty f]
    +\sum_{|\beta|=m}\sum_{\nu=1}^Mc^\alpha_{\nu,\beta}\chi_\nu\cdot\partial^\beta( S_\alpha[(\chi_\nu f)\circ\Phi_\nu]\circ\Phi_\nu^{-1})
    \\
    =&\partial^\alpha S_\alpha[\chi_0f]+\sum_{\gamma\le\alpha}{\alpha\choose\gamma}(-1)^{|\alpha-\gamma|}\partial^\gamma\Big((\partial^{\alpha-\gamma}\zeta_\infty)\cdot S_\alpha[\chi_\infty f]\Big)
    \\
    &+\sum_{|\beta|=m}\sum_{\gamma\le\beta}{\beta\choose\gamma}(-1)^{|\beta-\gamma|}
    \partial^\gamma\sum_{\nu=1}^M(c^\alpha_{\nu,\beta}\partial^{\beta-\gamma}\chi_\nu)\cdot( S_\alpha[(\chi_\nu f)\circ\Phi_\nu]\circ\Phi_\nu^{-1}).
\end{aligned}
\end{equation}

We now define $\Sc^{m,\gamma}_\Omega$ for $|\gamma|\le m$ as the following:
\begin{equation}\label{Eqn::SecCor::PfAtdBd::DefFc1}
\begin{aligned}
    \Sc_\Omega^{m,0}f:=&\textstyle\widetilde\Sc^{m,0}_\Omega f+S_0[\chi_0f]+\zeta_\infty\cdot S_0[\chi_\infty f]
    +\sum_{\nu=1}^M\chi_\nu\cdot( S_0[(\chi_\nu f)\circ\Phi_\nu]\circ\Phi_\nu^{-1}),&\gamma=0;\\
    \Sc_\Omega^{m,\gamma}f:=&\widetilde\Sc^{m,\gamma}_\Omega f,&0<|\gamma|<m;\\
    \Sc_\Omega^{m,\gamma}f:=&\widetilde\Sc^{m,\gamma}_\Omega f+S_\gamma[\chi_0f],&|\gamma|=m;
\end{aligned}
\end{equation}
where 
\begin{equation*}
    \widetilde\Sc^{m,\gamma}_\Omega f:=\sum_{\substack{|\alpha|=m\\\alpha\ge\gamma}}{\alpha\choose\gamma}(-1)^{|\alpha-\gamma|}\bigg((\partial^{\alpha-\gamma}\zeta_\infty)\cdot S_\alpha[\chi_\infty f]
    \\+\sum_{\nu=1}^M \sum_{|\beta|=m}
   (c^\beta_{\nu,\alpha}\partial^{\alpha-\gamma}\chi_\nu) 
    \cdot( S_\beta[(\chi_\nu f)\circ\Phi_\nu]\circ\Phi_\nu^{-1})\bigg).
\end{equation*}


By \eqref{Eqn::SecCor::PfAtdBd::Sum1} and \eqref{Eqn::SecCor::PfAtdBd::Sum2} we get $f=\sum_{|\gamma|\le m}\partial^\gamma(\Sc^{m,\gamma}_\Omega f)$, which is \ref{Item::SecCor::AntiDevBddDom::Sum}.

By Lemma~\ref{Lem::SecCor::AntiDevSpeDom}, $ S_0, S_\alpha$ $(|\alpha|=m)$ are bounded $S_0,S_\alpha:\As_{pq}^s(\R^n)\to\As_{pq}^{s+m}(\R^n)$ for $\As\in\{\Bs,\Fs\}$, $s\in\R$, $m\in\Z_+$ and $(p,q)$ admissible, so the same hold for $\Sc^{m,\gamma}_\Omega$ for $|\gamma|\le m$. This finishes the proof of \ref{Item::SecCor::AntiDevBddDom::Bdd}.

\medskip
To prove \ref{Item::SecCor::AntiDevBddDom::Vanish}, suppose $f|_\Omega \equiv 0$, then we have $\chi_0f \equiv 0$. 
By \ref{Item::PartitionUnity::chi} and \ref{Item::PartitionUnity::Phi} in Notation~\ref{Note::PartitionUnity}, $\supp((\chi_\nu f)\circ\Phi_\nu)\subseteq\overline{\omega_\nu}^c$ for $1\le\nu\le M$.
By Lemma~\ref{Lem::SecCor::AntiDevSpeDom} \ref{Item::SecCor::AntiDevSpeDom::Vanish}, we have $\supp( S_\alpha((\chi_\nu f)\circ\Phi_\nu)))\subseteq\omega_\nu^c$ for all $\alpha\in\N^n$ and $1\le\nu\le M$, and so 
\begin{equation*}
    \supp\big((\partial^\beta\chi_\nu)\cdot( S_\alpha[(\chi_\nu f)\circ\Phi_\nu]\circ\Phi_\nu^{-1})\big)
    \subseteq  \supp\chi_\nu \cap \Phi_\nu(\omega_\nu^c) 
    \subseteq\Omega^c, \quad\text{for all $\alpha,\beta\in \N^n, 1 \le \nu \le M$.}
\end{equation*}

Note that $\supp((\partial^\beta\zeta_\infty)\cdot S_\alpha[\chi_\infty f])\subseteq\supp\zeta_\infty\subset \overline\Omega^c$ for all $\alpha,\beta\in\N^n$. Since every term in \eqref{Eqn::SecCor::PfAtdBd::DefFc1} that involves $\chi_0$ vanishes (as we assume $\supp f\subseteq\Omega^c$), we see that $\Sc^{m,\alpha}_\Omega f$ is supported in $\Omega^c$ for every $|\alpha|\le m$. In other words, $\Sc^{m,\alpha}_\Omega f|_\Omega=0$, finishing the proof.
\end{proof}
\begin{rem}
    Similar to Remark~\ref{Rmk::SecCor::FolkloreConstant}, the constant in Proposition~\ref{Prop::Intro::AntiDevBddDom} is (upper) stable under small Lipschitz perturbation of the domain. We leave the precise statement and the proof to the reader. 
\end{rem}

\addtocontents{toc}{\protect\setcounter{tocdepth}{1}}
\section{Further Remarks and Open Questions}\label{Section::Further}
Given an arbitrary extension operator $\Ec$ on a domain $\Omega$ not necessarily bounded Lipschitz, apart from considering the boundedness on certain function spaces, we can ask the following questions:
\begin{itemize}
    \item For $\alpha\in\N^n$, can we find some operator $\Ec^{\alpha,\beta}$ ($\beta\in\N^n$) that has similar properties as $\Ec$, such that $\partial^\alpha\circ \Ec=\sum_{\beta}\Ec^{\alpha,\beta}\circ\partial^\beta$?  
    \item For a nonsmooth function $f$ on the domain, does $\Ec f$ define a smooth function outside the domain? If yes, can we obtain some kinds of quantitative characterization of the smoothing effect?
\end{itemize}

Affirmative answers to these questions will allow us to prove the analogs of Theorems \ref{Thm::Intro::Folklore!} and \ref{Prop::Intro::SmoothBddDom} for the domain $\Omega$. 
In particular, it would be interesting to consider these questions on locally uniform domains.

\subsection*{Some boundedness of $\tilde T^{\eta,\theta,r}_\Omega$}

In this paper we do not mention the $\Fs_{\infty q}^s$-spaces for $0<q<\infty$. This is due to the fact that these spaces have rather complicated characterizations, as we cannot use $\|(\sum_{j\in\N}2^{jsq}|\lambda_j\ast f|^q)^\frac1q\|_{L^\infty}$ to define the $\Fs_{\infty q}^s$-norm. See \cite[Remark~2.84]{TriebelTheoryOfFunctionSpacesIV} for some discussion. Nevertheless the analogous results are done in \cite{YaoMorrey}.



More generally, we can ask whether we have the boundedness of $\tilde T^{\eta,\theta,r}_\Omega$ and $T^{\eta,\theta,r}_\omega$ on other types of function spaces. For example, we can consider the anisotropic spaces (see \cite{JohnsenHansenSickelAnisotropicLocalMeans}), or the  variable-exponent spaces (see \cite{VariableLipschitz}). 

\subsection*{Other types of quantitative smoothing estimate}
In Theorem~\ref{Thm::MainThm} \ref{Item::MainThm::Smoothing}, we consider the estimates when the codomain of $T_\omega^{\eta,\theta,r}$ is the weighted Sobolev space  $W^{m,p}(\overline{\omega}^c,\dist_\omega^t)$. It is possible to use some other types of weighted function spaces on $\overline{\omega}^c$ that give an alternative characterizations of the smoothing effect of the extension operator. 

For example, take $(\phi_j)_{j=0}^\infty$ to be the family in \eqref{Eqn::Intro::FsNorm}. Denote $\breve\phi_j(x):=\breve\phi_j(-x)$, we can consider the space on $\overline{\omega}^c$ with norm 
\begin{equation*} 
    \|f\|=\|f\|_{s,p,q,\{\varphi_j\}}:=\Big(\sum_{j=1}^\infty2^{jqs}\|\breve\phi_j\ast f\|_{L^p(\overline\omega^c,\varphi_j)}^q\Big)^\frac1q,
\end{equation*}for a family of weights $\varphi_j:\overline{\omega}^c\to\R_+$ that are certain powers of the distance function $\dist_\omega$. 

One can ask whether Theorem~\ref{Thm::MainThm} \ref{Item::MainThm::Smoothing} still holds with $W^{m,p}(\overline{\omega}^c,\dist_\omega^t)$ replaced by the space with norm $\| \cdot  \|_{s,p,q,\{\varphi_j\}}$. 

\begin{ack} 
We would like to express our appreciation to Xianghong Gong, Brian Street and Andreas Seeger for many helpful discussions and advice. We would like to thank Hans Triebel and the referees for their interest and valuable comments. We would also like to thank Tomas Miguel Rodriguez for the discussion on Lipschitz defining functions.

\end{ack}

\bibliographystyle{amsalpha}
\bibliography{reference} 

\end{document}